\documentclass[11pt,english,reqno]{amsart}

\usepackage{amsmath}
\usepackage{amsthm}
\usepackage{amssymb}
\usepackage{graphicx}
\usepackage{color}
\usepackage[margin=1in]{geometry}
\usepackage[nobysame]{amsrefs}

\usepackage[nomarkers,nolists]{endfloat}

\numberwithin{equation}{section} 

\newtheorem{Theorem}{Theorem}[section]

\newtheorem{Proposition}[Theorem]{Proposition}

\theoremstyle{remark}
\newtheorem{rmk}{Remark}[section]
\newtheorem{thm}{Theorem}[section]
\newtheorem{lem}{Lemma}[section]

\theoremstyle{definition}

\renewcommand{\tilde}{\widetilde}
\renewcommand{\hat}{\widehat}

\newcommand{\nn}{\nonumber}
\newcommand{\R}{{\mathbb R}}

\newcommand{\N}{{\mathbb N}}

\newcommand{\del}{\partial}
\newcommand{\dx}{ \, {\rm d} x}
\newcommand{\dy}{ \, {\rm d} y}

\newcommand{\dt}{ \, {\rm d} t}

\newcommand{\ds}{\, {\rm d} s}

\newcommand{\dtau}{\, {\rm d} \tau}
\newcommand{\dmu}{\, {\rm d} \mu}

\newcommand{\CalL}{{\mathcal{L}}}
\newcommand{\CalO}{{\mathcal{O}}}
\newcommand{\CalP}{{\mathcal{P}}}
\newcommand{\CalR}{{\mathcal{R}}}

\newcommand{\Eps}{\epsilon}

\newcommand{\VecL}{{\CalL}}

\newcommand{\NullL} {{\rm Null} \, \VecL}
\newcommand{\BigO}{{\mathcal{O}}}

\newcommand{\mc}[1]{\mathcal{#1}}

\newcommand{\ud}{\,\mathrm{d}}
\newcommand{\rd}{\mathrm{d}}

\newcommand{\abs}[1]{\left\lvert#1\right\rvert}
\newcommand{\norm}[1]{\left\lVert#1\right\rVert}
\newcommand{\average}[1]{\bigl\langle#1\bigr\rangle}

\newcommand{\vint}[1]{\langle #1 \rangle}
\newcommand{\vpran}[1]{\left( #1 \right)}

\newcommand{\IBL}{\mathrm{IBL}}
\newcommand{\inte}{\mathrm{int}}

\DeclareMathOperator{\Span}{span}

\newcommand{\wt}[1]{\widetilde{#1}}

\graphicspath{{../}}

\date{\today}

\begin{document}

\title[Domain Decomposition Kinetic-Fluid]{Diffusion Approximations and Domain Decomposition Method of Linear Transport Equations: Asymptotics and Numerics}

\author{Qin Li}
\address{Computing and Mathematical Sciences, California Institute of Technology, 1200 E California Blvd. MC 305-16, Pasadena, CA 91125 USA}
\email{qinli@caltech.edu}
\author{Jianfeng Lu}
\address{Departments of Mathematics, Physics, and Chemistry, Duke University, Box 90320, Durham, NC 27708 USA}
\email{jianfeng@math.duke.edu}
\author{Weiran Sun}
\address{Department of Mathematics, Simon Fraser University, 8888 University Dr., Burnaby, BC V5A 1S6, Canada}
\email{weirans@sfu.ca}

\date{\today}

\thanks{The research of Q.L.~was supported in part by the AFOSR MURI grant FA9550-09-1-0613 and the National Science Foundation under award DMS-1318377. The research of J.L.~was supported in part by the Alfred P.~Sloan Foundation and the National Science Foundation under award DMS-1312659. The research of W.S.~was supported in part by the Simon Fraser University President's Research Start-up Grant PRSG-877723 and NSERC Discovery Individual Grant \#611626.}

\begin{abstract}
  In this paper we construct numerical schemes to approximate linear
  transport equations with slab geometry by diffusion equations. We
  treat both the case of pure diffusive scaling and the case where
  kinetic and diffusive scalings coexist. The diffusion equations and
  their data are derived from asymptotic and layer analysis which
  allows general scattering kernels and general data. We apply the
  half-space solver in \cite{LiLuSun} to resolve the boundary layer
  equation and obtain the boundary data for the diffusion
  equation. The algorithms are validated by numerical experiments and
  also by error analysis for the pure diffusive scaling case.
\end{abstract}

\maketitle

\section{Introduction} 

Linear transport equations are widely used to model the interaction of
particles with background media through various processes such as
scattering, absorption, and emission. Many interesting physical
systems exhibit heterogeneity that involve multiple temporal or
spatial scales. In this paper we focus on efficient numerical
simulations for linear transport equations which exhibit diffusive
regime in part of or the whole domain.  More precisely, with slab
geometry, the particle density function $f$ in our model depends on a
one-dimensional spatial variable $x \in [a, b]$ and a one-dimensional
angular variable $\mu \in [-1, 1]$. The transport equation has the
reduced form
\begin{equation}
    \label{eq:kinetic}
\begin{aligned} 
   \Eps \, \del_t f + \mu &\del_x f + \frac{\sigma(x)}{\Eps} \, \CalL f = 0 \,,
\\
   f|_{x=a} &= \phi_a (t, \mu) \,,  \qquad \mu > 0 \,,
\\
   f|_{x=b} &= \phi_b (t, \mu) \,,  \qquad \mu < 0 \,,
\\
   f|_{t=0} &= \phi_0 (x, \mu) \,,
\end{aligned}
\end{equation}
where $\phi_a, \phi_b$ are given incoming data at the boundary and $\phi_0$ is the given initial data. The collision operator considered in this paper has the form
\begin{align} \label{def:CalL}
    \CalL f = f -  \int_{-1}^1 \kappa(\mu, \mu') f(\mu') \dmu' \,,
\end{align}
where $\dmu'$ is the regular Lebesgue measure. The scattering kernel $\kappa$ satisfies that 
\begin{align*}
   \kappa \geq 0 \,,
\qquad
   \kappa(\mu, \mu') = \kappa(\mu', \mu) \,,
 \qquad \int_{-1}^1 \kappa(\mu, \mu') \dmu' = 1
 \quad \, \text{for all $\mu \in [-1,1]$}\,.
\end{align*}
The parameter $\Eps$ is the mean free path which is small compared with typical macroscopic length scale in the diffusion region. We will consider two cases for the coefficient $\sigma(x)$: 
\begin{itemize}
\item[1)]
$\sigma(x) = 1$ for all $x \in [a, b]$ such that we have diffusive scaling over the whole domain;  \smallskip
\item[2)]
the system contains two scales such that
\begin{align*}
    \sigma(x) 
    = \begin{cases}
         \Eps \,, & x \in [a, x_m] \,, \\[2pt]
         1 \,, & x \in [x_m, b] \,,
       \end{cases}
\end{align*}
for some $x_m \in (a, b)$, such that we have
a kinetic scaling on the left domain and a diffusive scaling on the
right. These two regions are coupled together at the interface $x =
x_m$.
\end{itemize} 
We note that while we focus on system~\eqref{eq:kinetic} which is one-dimensional in both the spatial and the angular variables, it is possible to extend to higher dimensional systems with simple geometry. Particular examples include the upper-half space $(\R^n)^+$ in the pure diffusion case and flat interface in the kinetic-fluid coupling case. We also comment that our method can be applied to general linear or linearized kinetic equations such as the linearized
Boltzmann equations.

In this paper we aim at designing efficient multiscale algorithms for
\eqref{eq:kinetic} based on asymptotic analysis and domain
decomposition.  It is well-known that direct simulations of the
transport equation in the diffusion region are usually rather
expensive thus unfavorable. On the other hand, diffusion equations
with proper data can provide good approximations to the kinetic
equation when $\Eps$ is small (see for example \cite{GJL99,
  GJL03}). We will follow the latter route and use diffusion
approximations wherever applicable. The main difficulty of this method
lies in obtaining accurate matching boundary and initial conditions
for the diffusion equation. The easier part is to obtain the initial
data: at the leading order, it can be derived by directly projecting
the given kinetic initial state onto the null space of the scattering
operator $\CalL$.

Finding the boundary data on the other hand is more involved both
asymptotically and numerically. For given kinetic incoming data, one
can show by formal asymptotic analysis that the matching boundary data
of the diffusion equation is determined by the end-state of the
solution to a half-space equation. Therefore, accurate solvers of
half-space equations will provide crucial tools for our
approximation. Having this in mind, we developed a numerical method
in \cite{LiLuSun} that can efficiently solve the half-space equations. In the current paper, we will apply this half-space
solver to obtain the boundary data for the
diffusion equations numerically.

In summary, in the numerical scheme for the pure diffusion case, we
will resolve the boundary layer equations at the two endpoints $x = a, b$ to
retrieve boundary data for the diffusion equation and use the
projected kinetic initial data as its initial condition. We will
compare thus-obtained diffusion solution to the solution of kinetic
equation and show convergence rates in terms of $\Eps$. We will also
derive some formal error estimates in the $L^2$-spaces. The error analysis
follows the classical methodology of constructing approximate solution
that involves all the layers \cite{BLP:79}. Since we are studying time-dependent
case, normally there will be three types of layers involved: boundary,
initial, and initial-boundary layers. The major
assumption that we make here is to assume that the initial-boundary layer equation is well-posed and its solution decays at least as fast as the reciprocal of time.  The initial and boundary layers on the other hand can be shown to have an exponential decay. We will treat the general cases where the derived data for the diffusion
equation are allowed to be incompatible so that the derivatives of the
heat solution can be unbounded.

In the formal asymptotic analysis of the kinetic-diffusion coupling case with general initial data, there are boundary, initial, and initial-boundary layers that form at the interface in the diffusion region. Solutions to these layers will have influence on the kinetic part. Our numerical scheme, however, only takes the boundary-layer feedback into account and ignores the other feedbacks from the initial and initial-boundary layers at the interface. This way we can decouple the kinetic and diffusion parts at the leading order. This decoupling idea is a feature of the domain-decomposition method developed in \cite{GJL03}.  In particular, at the leading order, the kinetic part satisfies a closed system whose boundary condition at the interface is given by the Albedo operator defined in~\eqref{def: Albedo}. By this we can fully solve the leading order decoupled kinetic equation in the kinetic region. Using the solution from the kinetic regime at the interface as the given incoming data, we then approximate the kinetic equation in the diffusion region by the diffusion equation via the same scheme for the pure diffusion case. We comment that although we do not have rigorous analysis for estimating the errors induced by ignoring the initial and initial-boundary layer feedback at the interface, these feedback only depend on a time scale of order $\CalO(\Eps^2)$. Therefore their effect is expected to be negligible after an initial layer. To provide some justification, we perform a stability test by adding a perturbation that only depends on $\frac{t}{\Eps^2}$ to the Albedo operator in the boundary condition for the kinetic equation. The numerical result indeed shows the decay of the error after some time.

Numerical methods for kinetic equations that exploit the fluid
approximation were discussed in many previous
work~\cite{BesseBorghol:11,GJL99,GJL03,LM12,BM:02,klar96,AokiTakata_1997,DJ05,CDL:03,DDM:07}. In
the engineering literature, the particle methods or Lagrangian type of
approaches are often used (see
e.g.~\cite{pullin76,hadjiconstantinou10}). The advantage of such
methods is that they are easy to implement. However, the convergence
rate of this type of methods could be slow and sometimes is hard to
analyze. This happens especially when kinetic and fluid regions
coexist. The second class of methods are based on deterministic domain
decomposition, where one treats different regions separately and then
couples them together via the interface. Several systematic numerical
methods in this direction are proposed in~\cite{GJL99,GJL03,LM12} in
which H-function or generalized H-functions are used in order to study
the boundary-layer equation at the interface. This restricts the
application of these methods to isotropic collisions or collision
kernels with particular structures, while our approach applies to
general kernels.  In~\cite{BM:02} the authors relaxed the constraints
by exploring a particular weak formulation and using iterations
between kinetic and fluid regions (see also in~\cite{AokiTakata_1997}
for a similar method). Another work in a similar spirit is proposed
in~\cite{klar96} where the corresponding half-space problem is solved
by iterating incorporated with the Chapman-Enskog
expansion~\cite{GolseKlar:95}.  The generalization of these methods to
time-dependent problems would lead to high computational cost since
the iteration is performed at every time step. In a series of
work~\cite{DJ05,CDL:03,DDM:07} the idea of using a buffer zone is
pursued where a smooth transition is proposed to connect the fluid and
kinetic regime. The numerical error caused by the smooth transition is
yet to be investigated. Adaptive choices of the fluid regime have also
been studied~\cite{CD:13, TK98,FR14}. Here we take the classical
domain decomposition approach and note that it is possible to explore
the adaptive criteria.  We also mention the asymptotic preserving (AP)
scheme~\cites{Jin99, Jin12, LM12} for multiscale kinetic problem that
mainly deals with the time stepping issue. While the focus of the
current work is domain decomposition by using fluid approximation, it
can be combined with AP scheme for the kinetic regime.

This paper is laid out as follows. In Section 2, we show the formal asymptotic derivation of the diffusion equation for both the pure and coupling cases. In Section 3, we explain our numerical algorithm. Various numerical results for both the pure and coupling cases are shown in Section 4.  In the Appendix, we show some $L^2$-estimates of the asymptotic errors for the pure diffusion case based on the assumption that the initial-boundary layer solution decays fast enough. 

\section{Asymtotics} \label{section:asymptotics}

In this section we apply asymptotic expansions to derive approximate equations for both the pure diffusion and the kinetic-diffusion coupling cases. 

\subsection{Notations} Throughout the paper, we will use $L^2(\dmu)$ to denote the $L^2$-space in $\mu$ with Lebesgue measure on $[-1,1]$. We also denote
\begin{align*}
     \vint{f} = \frac{1}{2}\int_{-1}^1 f \dmu \,,
\end{align*}
for any $f \in L^1(-1,1)$.

\subsection{Properties of $\CalL$} In this part we summarize the basic properties of the collision operator $\CalL$ defined in~\eqref{def:CalL}.  Denote $\NullL$ as the null space of $\CalL$. Let $\CalP: L^2(\dmu) \to \NullL$ be the projection onto $\NullL$. 
The main properties on $\CalL$
are as follows:
\begin{enumerate}
\item[(P1)]  
   $\VecL: L^2(\dmu) \to L^2(\dmu)$ is self-adjoint, nonnegative, and bounded operator; \medskip
\item[(P2)] 
  $\NullL = \Span\{1\}$; \medskip
\item[(P3)] $\VecL$ has a spectral gap: there exists $\sigma_0 > 0$ such that
\begin{equation*}
      \int_{-1}^1 f \, \VecL f \dmu
 \geq \sigma_0 \left\|\CalP^\perp f \right\|_{L^2(\dmu)}^2  
\qquad \text{for any $f \in L^2(\dmu)$} \,, 
\end{equation*}
where $\CalP^\perp = {\mathcal{I}} - \CalP$ is the projection onto the null orthogonal space $(\NullL)^\perp$.
\end{enumerate}

Recall that the collision operator $\CalL$ in this paper has the particular form in~\eqref{def:CalL}, which implies that every Legendre polynomial is an eigenvalue of $\CalL$ and $\NullL = \Span\{1\}$. Hence if we denote $\lambda_n$ as the associated eigenvalue of the normalized Legendre polynomial $p_n$, then
\begin{align*}
    \lambda_n  \norm{p_n}_{L^2(\dmu)}^2 = \int_{-1}^1 p_n \CalL p_n \dmu \geq \sigma_0 \norm{p_n}_{L^2(\dmu)}^2 \,, \qquad 
\text{for any $n \geq 1$.}
\end{align*}
This implies that for the family of collision operators considered in this paper, we have
\begin{align} \label{cond:lambda_n}
     \lambda_n \geq \sigma_0 > 0 \,, \text{ for any $n \geq 1$.}
\end{align}

\subsection{Pure Diffusion Approximation} We will follow the classical balance argument to construct the leading-order interior diffusion equation. The boundary and initial data for the diffusion equation will be derived from the boundary and initial layers. 

\subsubsection{Interior Equation} Recall that $\CalP$ is the projection from $L^2(\dmu)$ onto $\NullL$ and $\CalP^\perp$ is its orthogonal projection. 
Decompose $f$ such that 
\begin{align*}
     f = \CalP f + \CalP^\perp f \,.
\end{align*}
For simplicity we denote $\theta = \CalP f = \vint{f}$ and write
\begin{align*}
     f = \theta + \CalP^\perp f \,.
\end{align*}
Applying $\CalP$ and $\CalP^\perp$ to~\eqref{eq:kinetic} respectively, we have
\begin{align} \label{eq:P-f}
     \Eps \, \del_t \CalP f + \del_x \CalP(\mu f) =  0 \,,
\end{align}
and 
\begin{align} \label{eq:P-orthogonal-f}
     \Eps \, \del_t \CalP^\perp f + \del_x \CalP^\perp (\mu f) 
     + \frac{1}{\Eps} \CalL f =  0 \,.
\end{align}
Rewrite~\eqref{eq:P-orthogonal-f} as
\begin{align*}
     \CalL f = - \Eps^2 \, \del_t \CalP^\perp f 
                    - \Eps \del_x \CalP^\perp (\mu \CalP f)
                    - \Eps \del_x \CalP^\perp (\mu \CalP^\perp f) \,.
\end{align*}
Assume that derivatives of $\CalP^\perp f$ are of lower orders than derivatives of $\CalP f$. Then the leading order closure can be constructed as
\begin{align} \label{eq:f-1}
     \CalL f_1 = - \Eps \, \del_x \CalP^\perp (\mu \CalP f)
                    = - \Eps \, \mu \, \del_x(\CalP f) \,,
\end{align}
where $f_1 \in \vpran{\NullL}^\perp$ is the first correction to $\CalP f$. Solving~\eqref{eq:f-1} gives
\begin{align} \label{eq:f-1-1}
     f_1 =  - \Eps \, \CalL^{-1} (\mu) \, \del_x(\CalP f)
           = - \Eps \, \CalL^{-1} (\mu) \, \del_x \theta \,.
\end{align}
Since $\CalP(\mu) = 0$, equation ~\eqref{eq:P-f} can be written as
\begin{align} \label{eq:P-f-simp}
    \Eps \del_t \theta + \del_x \CalP(\mu \CalP^\perp f) = 0 \,.
\end{align}
Using $f_1$ as the approximation for $\CalP^\perp f$ and inserting the leading order approximation~\eqref{eq:f-1-1} into~\eqref{eq:P-f-simp}, we thus obtain the interior equation for $\theta$ as
\begin{align} \label{eq:theta}
    \del_t \theta - \vint{\mu \CalL^{-1} (\mu)} \, \del_{xx} \theta = 0 \,.
\end{align}

\subsubsection{Boundary Conditions} The boundary conditions for~\eqref{eq:theta} will be derived from boundary layer analysis \cite{CoronGolseSulem:88}. We will show the details for the boundary layer at $x=a$. The derivation at $x = b$ is similar thus omitted. 

Within the layer width of order $\CalO(\Eps)$ at $x=a$, let $y$ be the rescaled spatial variable such that $y = \frac{x-a}{\Eps}$. Denote the boundary layer solution at $x=a$ as $f^b_L(t, y, \mu)$ and assume it has the asymptotic expansion
\begin{align} \label{eq:f-b-L-asymp}
     f^b_L = f^b_{0,L} + \Eps f^b_{1,L} + \cdots \,.
\end{align}
Then $f^b_L$ satisfies the rescaled kinetic equation
\begin{align}  \label{eq:BL-a}
   \Eps \, \del_t f^b_L + \frac{1}{\Eps} \mu &\del_y f^b_L 
   + \frac{1}{\Eps} \, \CalL f^b_L = 0  \,.
\end{align}
Apply~\eqref{eq:f-b-L-asymp} in~\eqref{eq:BL-a} and compare the terms to each order. Then the leading order term $f^b_{0,L}$ satisfies the half-space equation
\begin{equation}
\begin{aligned} \label{eq:half-space-a}
   \mu \del_y f^b_{0,L} &+ \CalL f^b_{0,L} = 0  \,,
\\
   f^b_{0,L}|_{y=0} &= \phi_a(t, \mu) \,, \qquad \mu > 0 \,,
\\
   f^b_{0,L} &\to \theta_a \,, \hspace{1.5cm} \text{as $y \to \infty$}. 
\end{aligned}
\end{equation}
for some constant $\theta_a \in \R$. Similarly, the leading order boundary layer solution at $x=b$, denoted as $f^b_{0, R}$, satisfies 
\begin{equation}
\begin{aligned} \label{eq:half-space-b}
   -\mu \del_z f^b_{0,R} &+ \CalL f^b_{0,R} = 0  \,,
\\
   f^b_{0,R}|_{z=0} &= \phi_b(t, \mu) \,, \qquad \mu < 0 \,,
\\
   f^b_{0,R} &\to \theta_b \,, \hspace{1.5cm} \text{as $z \to \infty$} \,,
\end{aligned}
\end{equation}
where the rescaled spatial variable $z$ is defined as $z = \frac{b-x}{\Eps}$. The well-posedness of the half-space equations~\eqref{eq:half-space-a} and~\eqref{eq:half-space-b} is shown in \cite{CoronGolseSulem:88}. We will use the constants $\theta_a, \theta_b$ as the boundary conditions for the interior diffusion equation~\eqref{eq:theta} for $\theta$.

\subsubsection{Initial Conditions} The initial condition for $\theta$ is derived from the initial layer. Specifically, let $\tau = \frac{t}{\Eps^2}$ be the rescaled time variable.  Suppose the initial layer corrector $f^I$ satisfies the asymptotic expansion
\begin{align*}
     f^I = f^I_0 + \Eps f^I_1 + \cdots \,.
\end{align*}
Then $f^I$ satisfies the rescaled kinetic equation
\begin{align*} 
   \frac{1}{\Eps} \, \del_t f^I +  \mu &\del_y f^I + \frac{1}{\Eps} \, \CalL f^I = 0  \,.
\end{align*}
Hence the leading order $f^I_0$ satisfies  that 
\begin{equation} \label{eq:initial-layer-0}
\begin{aligned} 
   \del_\tau f^I_0 &+ \CalL f^I_0 = 0  \,,
\\
   f^I_0|_{\tau=0} &= \phi_0 - \vint{\phi_0} \,,  
\\
   f^I_0 &\to 0 \,, \qquad \text{as $\tau \to \infty$. }
\end{aligned}
\end{equation}
The following Proposition shows that $f^I_0$ is well-defined. 
\begin{Proposition} \label{prop:f-I-0}
Suppose $\phi_0 \in L^2(\dmu)$. Let $\{p_n\}_{n=0}^\infty$ be the set of normalized Legendre polynomials. Then there exists a unique solution $f^I_0$ to the initial layer equation~\eqref{eq:initial-layer-0}. Moreoever, the decay of $f^I_0$ is exponential. 
\end{Proposition}
\begin{proof}
Decompose $\phi_0 - \vint{\phi_0}$ as
\begin{align*}
    \phi_0 - \vint{\phi_0} = \sum_{n=1}^\infty \phi_{0, n} p_n(\mu) \,.
\end{align*}
where $p_n$'s are the normalized Legendre polynomials. Define
\begin{align*}
    f^I_0 = \sum_{n=1}^\infty e^{-\lambda_n \, \tau} \phi_{0, n} p_n(\mu) \in L^\infty(0, \infty; L^2(\dmu))\,.
\end{align*}
where $\lambda_n$ satisfies~\eqref{cond:lambda_n}.
By direct calculation one can show that $f^I_0$ is the unique solution to~\eqref{eq:initial-layer-0}. 
\end{proof}

We will use $\vint{\phi_0}(x)$ as the initial data for $\theta$. In summary, the leading order interior equation has the form
\begin{equation}
\begin{aligned} \label{eq:interior-full}
    \del_t \theta - \vint{\mu \CalL^{-1} (\mu)} \, &\del_{xx} \theta = 0 \,,
\\
    \theta|_{x=a} = \theta_a \,, 
\qquad 
    &\theta |_{x=b} = \theta_b \,, 
\\
    \theta|_{t=0} =  \vint{\phi_0&}(x) \,.
\end{aligned}
\end{equation}
where $\theta_a, \theta_b$ are defined in~\eqref{eq:half-space-a} and~\eqref{eq:half-space-b} respectively. In the numerical computation we will  compare $\theta$ with the solution $f$ to the kinetic equation~\eqref{eq:kinetic} and show the rate of convergence of $\theta$ toward $f$ in terms of $\Eps$.

\subsection{Kinetic-Diffusion Couplings} In this part we derive the approximate system when there are both kinetic and fluid regions present. The kinetic equation in this case has the form
\begin{equation} \label{eq:kinetic-coupling}
\begin{aligned}
   \Eps \, \del_t f + \mu &\del_x f + \frac{\sigma(x)}{\Eps} \, \CalL f = 0 \,,
\\
   f|_{x=a} &= \phi_a (t, \mu) \,,  \qquad \mu > 0 \,,
\\
   f|_{x=b} &= \phi_b (t, \mu) \,,  \qquad \mu < 0 \,,
\\
   f|_{t=0} &= \phi_0 (x, \mu) \,,
\end{aligned}
\end{equation}
where 
\begin{align*}
     \sigma(x) 
     = \begin{cases}
          \Eps, & x \in (a, x_m) \,,\\[4pt]
          1,      & x \in (x_m, b) \,,
     \end{cases}
\end{align*}
for some $x_m \in (a, b)$. Therefore, the left part of the region is kinetic while the right one can be well approximated by the diffusion equation. Our main goal is to understand the coupling of these two regions through the interface at $x_m$. The asymptotic analysis here is along a similar line as in \cite{GJL03}. In particular, to the leading order of the approximation, the kinetic equation forms a closed system, with the boundary condition at $x_m$ given by an operator that relates outgoing data from the kinetic region to the data feeding back from various layers at the interface. The main difference compared with \cite{GJL03} is that we will work with general data instead of well-prepared ones.  Let $f_{M, +} (t, \mu)$ and $f_{M, -}(t, \mu)$ be the outgoing and incoming part of the density function from the kinetic part at the interface $x_m$ respectively.  Then $f_{M, +}$ will be the incoming data for the diffusion region. Apply the same asymptotic analysis to the fluid region. We could construct the leading-order approximate solution by defining:
\begin{align*}
    f^A = & \theta(t, x) 
             + f^b_{0,m} (t, \tfrac{x-x_m}{\Eps}, \mu)
             + f^b_{0,R} (t, \tfrac{b-x}{\Eps}, \mu)
             + f^{I}_0(\tfrac{t}{\Eps^2}, x, \mu)
\\
           & + f^{\IBL}_m(\tfrac{t}{\Eps^2}, \tfrac{x-x_m}{\Eps}, \mu)
             + f^{\IBL}_R(\tfrac{t}{\Eps^2}, \tfrac{b-x}{\Eps}, \mu) \,,
\end{align*}
where the boundary layer corrector $f^b_{0,m}(t, y, \mu)$ at $x=x_m$ satisfy
\begin{equation} \label{eq:f-b-0-m}
\begin{aligned}
   \mu \del_y f^b_{0,m} + &\CalL f^b_{0,m} = 0  \,,
\\
   f^b_{0,m} = f_{M, +} &(t, \mu) - \theta_m \,, \qquad \mu > 0 \,,
\\
   f^b_{0,m} &\to 0 \,, \hspace{1.7cm} \text{as $y \to \infty$} \,,
\end{aligned}
\end{equation}
where $y = \frac{x-x_m}{\Eps}$ and $\theta_m$ is the end-state at $y = \infty$ given the incoming data as $f_{M, +}(t, \mu)$. The boundary layer solution $f^b_{0,R}(t, z, \mu)$ at $x=b$ satisfy 
\begin{equation} \label{eq:f-b-0-R}
\begin{aligned}
   -\mu \del_z f^b_{0,R} + &\CalL f^b_{0,R} = 0  \,,
\\
   f^b_{0,R} = \phi_b &(t, \mu) - \theta_b \,, \qquad \mu < 0 \,,
\\
   f^b_{0,R} &\to 0 \,, \hspace{1.6cm} \text{as $z \to \infty$} \,,
\end{aligned}
\end{equation}
where $z = \frac{b-x}{\Eps}$ and $\theta_b$ is the end-state at $z = \infty$ given the incoming data as $\phi_b(t, \mu)$. The initial layer corrector $f^I_0(\tau, x, \mu)$ satisfies 
\begin{align*} 
   \del_\tau f^I_0 &+ \CalL f^I_0 = 0  \,,
\\
   f^I_0|_{\tau=0} &= \phi_0 - \vint{\phi_0} \,,  \qquad x \in (x_m, b) \,,
\\
   f^I_0 &\to 0 \,, \hspace{2cm} \text{as $\tau \to \infty$} \,,
\end{align*}
where $\tau = \frac{t}{\Eps^2}$. The initial-boundary layer corrector $f^{\IBL}_m(\tau, y, \mu)$ at $x=x_m$ satisfies the equation 
\begin{align*}
    \del_\tau f^{\IBL}_m + \mu \del_y &f^{\IBL}_m
    + \CalL f^{\IBL}_m = 0 \,,
\\
   f^{\IBL}_m|_{y = 0 } &= - f^I_{0, m} (\tau, 0, \mu) \,, \qquad \mu > 0 \,,
\\
   f^{\IBL}_m|_{\tau = 0} &= -f^{b}_{0, m}(0, y, \mu) \,,
\\
   f^{\IBL}_m &\to 0 \,, \hspace{2.7cm} \text{as $\tau, y \to \infty$} \,.
\end{align*}
Lastly, the initial-boundary layer corrector $f^{\IBL}_R(\tau, y, \mu)$ at $x=b$ satisfies the equation 
\begin{align*}
    \del_\tau f^{\IBL}_R - \mu \del_z &f^{\IBL}_R
    + \CalL f^{\IBL}_R = 0 \,,
\\
   f^{\IBL}_R|_{z = 0 } &= - f^I_{0, R} (\tau, 0, \mu) \,, \qquad \mu < 0 \,,
\\
   f^{\IBL}_R|_{\tau = 0} &= -f^{b}_{0, R}(0, z, \mu) \,,
\\
   f^{\IBL}_R &\to 0 \,, \hspace{2.7cm} \text{as $\tau, z \to \infty$} \,.
\end{align*}
We assume here the initial-boundary layer equations are well-posed with a certain algebraic decay rate of the solution in $\tau$. Given these layer correctors, we set up the leading-order equation for the interior solution for the kinetic region $\hat f^{\inte}_0$ as
\begin{equation}
\begin{aligned} \label{eq:hat-f-int-0}
    \Eps \, \del_t \hat f^{\inte}_0 + &\mu \del_x \hat f^{\inte}_0
    + \CalL \hat f^{\inte}_0 = 0 \,,
\\
   \hat f^{\inte}_0|_{x=x_m} = \CalR(f_{M, +}&(t, x_m, -\mu)) + f^I_0(\tfrac{t}{\Eps^2}, 0, \mu)
                        + f^{\IBL}_m(\tfrac{t}{\Eps^2}, 0, \mu) \,, \qquad \mu < 0 \,,
\\
   \hat f^{\inte}_0 &|_{x=a} = \phi_a(t, \mu) \,, \hspace{5.2cm} \mu > 0 \,,
\\
   \hat f^{\inte}_0 &|_{t=0} = \phi_0 \,, \hspace{6.1cm} x \in (a, x_m) \,,
\end{aligned}
\end{equation}
where $[\CalR(f_{M, +}(t, x_m, -\mu))](t, \mu)$ is the same Albedo operator as in \cite{GJL03} which is defined as follow: let $f_m$ be the solution to the half-space equation
\begin{align*}
   \mu \del_y f_m& + \CalL f_m = 0  \,,
\\
   f_m = &f_{M, +} (t, \mu) \,, \qquad \mu > 0 \,,
\\
   f_m &\to h_m \,, \hspace{1.3cm} \text{as $y \to \infty$}
\end{align*}
for some constant $h_m \in \R$. Then 
\begin{align} \label{def: Albedo}
   [\CalR(f_{M, +}(t, x_m, -\mu))](t, \mu) = f_m(t, 0, \mu) \,,
\qquad \mu < 0  \,.
\end{align}
We assume that the additional term $f^I_0(\tfrac{t}{\Eps^2}, 0, \mu)+ f^{\IBL}_m(\tfrac{t}{\Eps^2}, 0, \mu)$ in~\eqref{eq:hat-f-int-0} provides only a small perturbation of a certain order of $\Eps$ to the following system:
\begin{equation} \label{eq:f-int-0-kinetic}
\begin{aligned}
    \Eps \, \del_t  f^{\inte}_0 + \mu \del_x &f^{\inte}_0
    + \CalL f^{\inte}_0 = 0 \,,
\\
   f^{\inte}_0|_{x=x_m} &= \CalR(f_{M, +})  \,, \qquad \mu < 0 \,,
\\
    f^{\inte}_0|_{x=a} &= \phi_a(t, \mu) \,, \hspace{1cm} \mu > 0 \,,
\\
    f^{\inte}_0|_{t=0} &= \phi_0 \,, \hspace{1.8cm} x \in (a, x_m) \,,
\end{aligned}
\end{equation}
Therefore, instead of $\hat f^{\inte}_0$, we will use $f^{\inte}_0$ as the leading-order approximation in the kinetic region. In this way the kinetic equation~\eqref{eq:f-int-0-kinetic} for $x \in (a, x_m)$ is decoupled from the fluid region for $x \in (x_m, b)$ and can be solved independently. In the fluid region, we use the solution $\theta$ as the interior approximation which satisfies
\begin{equation} \label{eq:theta-fluid}
\begin{aligned}
    \del_t \theta - \vint{\mu \CalL^{-1} (\mu)} \, &\del_{xx} \theta = 0 \,,
\\
    \theta|_{x=x_m} = \theta_m \,, 
\qquad 
    &\theta|_{x=b} = \theta_b \,, 
\\
    \theta|_{t=0} =  \vint{\phi_0 &}(x) \,, \qquad x \in (x_m ,b) \,,
\end{aligned}
\end{equation}
where $\theta_m, \theta_b$ are defined in~\eqref{eq:f-b-0-m} and~\eqref{eq:f-b-0-R} respectively. Note that $\theta$ depends on the kinetic solution $f^{\inte}_0$ via $\theta_m$. As a summary, we use the coupled system~\eqref{eq:f-int-0-kinetic} and~\eqref{eq:theta-fluid} as the approximate system to the original kinetic equation~\eqref{eq:kinetic-coupling}.

\section{Algorithm} \label{section:algorithm} 
An efficiently numerical solver for kinetic equation is to apply fluid approximation whenever possible. For the region where $\epsilon \ll 1$, we compute the heat equation instead of directly computing the kinetic equation. The computational saving is significant: first, discretization does not depend on $\epsilon$, and in addition, the heat equation has one less dimension ($\mu$ variable) to resolve. When both regimes co-exist, we apply standard algorithms for the heat equation and kinetic equation in fluid and kinetic regimes respectively, and seek for boundary condition that connects them by applying the half space solver developed in~\cite{LiLuSun}. In what follows, we will briefly outline the numerical method for the half-space problem and treat the pure fluid and fluid-kinetic coupling scenarios separately.

\subsection{Numerical method for half-space problems}\label{sec:albedo}

In this subsection, we focus on the numerical method for the
half-space problems. This is a special case of the algorithm we
proposed in the previous work \cite{LiLuSun}. Here for completeness we
recall the procedures of the algorithm in the current simplified
setting. We refer the readers to \cite{LiLuSun} for more details and
analysis of the algorithm.

Consider the half space problem 
\begin{equation}\label{eqn:half_space}
\begin{cases}
\mu \partial_x f +  \mathcal{L} f = 0 \,, \\
f(0, \mu) = f_0(\mu ),\quad \mu >0 \,, \\
f(x,\mu) \to \theta_{\infty} \quad \text{as } x \to \infty \,.
\end{cases}
\end{equation}
We are interested in finding the limit at infinity $\theta_{\infty}$ and the back flow  at the boundary
\begin{equation}
f(0,\mu) = (\mc{R} f_0 )(\mu) \quad \text{for } \mu <0,
\end{equation}
where $\mc{R}$ is the Albedo operator defined in~\eqref{def: Albedo}.

To solve the infinite domain problem~\eqref{eqn:half_space}, we use a
semi-discrete method with a spectral discretization for the $\mu$-variable. In general, the solution might exhibit singularity like
jumps at $\mu = 0$. Hence, we use an even-odd decomposition of the
distribution function to avoid the Gibbs phenomena and ensure accuracy. Define
\begin{equation}
  f^E(x,\mu) = \frac{f(x,\mu)+f(x,-\mu)}{2} \,,
\qquad
  f^O(x,\mu) = \frac{f(x,\mu)-f(x,-\mu)}{2} 
\end{equation}
so that $f = f^E+f^O$. Due to the symmetry, it suffices to discretize
the functions $f^E$ and $f^O$ for $\mu \in (0, 1]$ and then extend the
functions to the whole interval $\mu \in [-1, 1]$. More precisely, we use the half-space Legendre
polynomials as basis functions, which are given by 
\begin{equation}
  \phi^E_m(\mu) = \begin{cases} \phi_m(\mu)& \mu>0 \,,\\ 
    \phi_m(-\mu)&   \mu<0 \,, \end{cases}
  \qquad\phi^O_m(\mu) = \begin{cases} \phi_m(\mu)& \mu>0 \,, \\ 
    -\phi_m(-\mu)& \mu<0 \,.\end{cases}
\end{equation}
Here $\phi_m$ are standard orthogonal Legendre polynomials on $[0,1]$ satisfying 
\begin{equation*}
  \int_0^1\phi_m(\mu)\phi_n(\mu)\ud\mu = \delta_{mn}. 
\end{equation*}
Note that the even and odd extensions $\phi^E_m$ and $\phi^O_m$ are no longer polynomials. It is easy to verify the orthogonality 
\begin{equation*}
\langle\phi^E_m,\phi^O_n\rangle = 0,\quad\langle\phi^E_m,\phi^E_n\rangle = \delta_{mn},\quad \langle\phi^O_m,\phi^O_n\rangle = \delta_{mn}, 
\end{equation*}
where we have used the notation
\begin{equation}
\langle f,g\rangle = \frac{1}{2}\int^1_{-1} f(\mu)g(\mu)\ud\mu.
\end{equation}

For the stability of the numerical scheme, we first solve a damped
version of the equation \eqref{eqn:half_space} and then recover the
solution to the original equation. The damped equation is given by
\begin{equation}\label{eqn:half_space_damped}
  \begin{cases}
    \mu \partial_x \wt{f} +  \mathcal{L} \wt{f} + \alpha \mu \average{\mu, \wt{f}} 
    + \alpha \mu (\mc{L}^{-1} \mu) \average{\mu (\mc{L}^{-1} \mu), \wt{f}} = 0;\\
    \wt{f}(0, \mu) = f_0(\mu ),\quad \mu >0;\\
    \wt{f}(x,\mu) \to 0 \quad \text{as } x \to \infty
  \end{cases}
\end{equation}
where $0 < \alpha \ll 1$ is a damping parameter. For the ease of notation, we denote the damped collision operator as
\begin{equation}
  \mc{L}^d \wt{f} = \mathcal{L} \wt{f} + \alpha \mu \average{\mu, \wt{f}} 
    + \alpha \mu (\mc{L}^{-1} \mu) \average{\mu (\mc{L}^{-1} \mu), \wt{f}}, 
\end{equation}
where $\mc{L}^{-1} \mu$ is well defined since $\mu \perp \NullL$.  Using
the half-space Legendre polynomials, we 
approximate the even and odd parts of the distribution functions by
\begin{equation}
  \wt{f}^E(x, \mu) = \sum_{k=1}^N c^E_k(x)\phi^E_k(\mu) \,,
\qquad 
  \wt{f}^O(x, \mu) = \sum_{k=1}^{N+1} c^O_k(x)\phi^O_k(\mu).
\end{equation}
Note that we have taken one more basis functions for the odd part than
for the even part, this is to make sure that the semi-discrete system
satisfies the inf-sup condition which guarantees well-posedness and
convergence of the scheme, as further discussed below. More details
can be found in \cite{LiLuSun}*{Corollary 3.4}. Substituting the approximation into equation \eqref{eqn:half_space_damped} and applying Galerkin method, we obtain the equation for the coefficients which reads
\begin{equation}\label{eqn:half_space_c}
  \mathsf{A}\frac{\rd}{\rd z}\vec{c}=\mathsf{B}\vec{c},
\end{equation}
where
\begin{equation}
  \mathsf{A}=
  \begin{pmatrix} 
    \mathsf{A}^\mu & 0\\
    0 & \mathsf{A}^\mu
  \end{pmatrix}, \quad 
  \mathsf{B} = 
  \begin{pmatrix} 
    0 & \mathsf{B}^E \\
    \mathsf{B}^O & 0
  \end{pmatrix}, \quad 
  \vec{c} = 
  \begin{pmatrix}
    c^O \\ c^E
  \end{pmatrix}
\end{equation}
with the matrices $\mathsf{A}^\mu$, $\mathsf{B}^E$ and $\mathsf{B}^O$ defined as
\begin{equation}\label{eqn:generalized_eigen}
  \mathsf{A}^\mu_{ij} = 2\int_0^1 \mu \phi_i(\mu) \phi_j(\mu) \ud \mu,\qquad \mathsf{B}^E_{ij} = 
  \average{\mc{L}^d \phi^E_j,\phi^E_i} \,, 
\qquad \mathsf{B}^O_{ij} = 
  \average{\mc{L}^d \phi^O_j,\phi^O_i}.
\end{equation}
Here we have used the fact that 
\begin{equation*}
\average{\mu\phi_i^E,\phi^E_j} = 0,\quad \average{\mu\phi^E_i,\phi^O_j} = 2 \int_0^1 \mu \phi_i(\mu) \phi_j(\mu) \ud \mu, \quad\forall i,j.
\end{equation*}
due to symmetry and the definition of $\phi^E$ and $\phi^O$.

To solve the equation~\eqref{eqn:half_space_c}, we need $2N+1$ boundary conditions to determine $\vec{c}$ at $z=0$. These boundary conditions are of two kinds. The first is given by the incoming boundary condition (cf.~\cite{LiLuSun}*{Eq.~(3.11)})
\begin{equation}\label{eqn:constraint_boundary_data}
  \sum_{i=1}^{N+1} \average{\mu\phi^O_i, \phi^E_j} c^O_i(0) + 
  \sum_{i=1}^{N} \average{\abs{\mu} \phi^E_i, \phi^E_j} c^E_i(0) 
  = 2 \int_0^1 \mu f_0 \phi^E_j \ud \mu,
\end{equation}
for $j = 1, \ldots, N$, which gives us $N$ conditions for $\vec{c}(0)$. The remaining conditions come from the requirement that $\wt{f} \to 0$ as $x \to \infty$. Hence, $\vec{c}(0)$ can only consist of the decaying modes. More precisely, consider  the generalized eigenvalue problem: 
\begin{equation}\label{eqn:GeneralizedEigen}
  \lambda_m\mathsf{A} \vec{v}_m=\mathsf{B}^T \vec{v}_m,
\end{equation}
with $\vec{v}_m$ being the eigenvector associated with eigenvalue $\lambda_m$. Then $e_m = \vec{v}^T_m \mathsf{A} \vec{c}$ satisfies 
\begin{equation}
  \frac{\rd}{\rd z}e_m = \lambda_m{e}_m \,.
\end{equation}
To satisfy the
condition imposed at infinity, no components associated with
non-negative eigenvalues are allowed, namely, $\vec{c}$ should satisfy
$e_m(x = 0) = 0$ for those $\lambda_m \geq 0$,
\begin{equation}\label{eqn:constraint_positive}
  e_m(0) = \vec{v}_m^T \mathsf{A} \vec{c}(0) = 0.
\end{equation}
It is proved (see \cite{LiLuSun}*{Corollary 3.7}) that there are
exactly $N$ positive eigenvalues and $1$ zero eigenvalue of the
generalized eigenvalue problem \eqref{eqn:GeneralizedEigen}. This
gives us $N+1$ conditions for $\vec{c}(0)$.  In sum, the two types of
constraints \eqref{eqn:constraint_boundary_data} and
\eqref{eqn:constraint_positive} together determine $\vec{c}$ at $x=0$
(note that we have $2N+1$ constraints and $2N+1$ degrees of freedom in
total). Then \eqref{eqn:half_space_c} fully determines $c$, which
gives us the approximate solution to the damped equation
\eqref{eqn:half_space_damped}.

The solution to the original boundary layer equation
\eqref{eqn:half_space} can be then easily recovered by
\begin{equation}
  f = \wt{f} - \theta_{\infty} ( g_0 - 1) 
\end{equation}
where $g_0$ solves the damped equation \eqref{eqn:half_space_damped}
with boundary data given by the constant $g_0(0, \mu) = 1, \forall\,
\mu > 0$ and $\theta_{\infty}$ is given by (it can be proved that
$\average{\mu, g_0(0, \mu)} \neq 0$, see \cite{LiLuSun}*{Proposition
  3.8})
\begin{equation}\label{eq:thetainf}
  \theta_{\infty} = \average{\mu, g_0(0, \mu)}^{-1} \average{\mu, \wt{f}(0, \mu)}. 
\end{equation}
By substituting into the equation, it is easy to verify that $f$
defined above solves \eqref{eqn:half_space} with boundary data given
by $f_0$. Note that as $\wt{f}$ and $g_0$ decay to $0$ as $x \to
\infty$, $\theta_{\infty}$ given by \eqref{eq:thetainf} is exactly the
limit of the solution at infinity. The action of the Albedo operator
on $f_0$ is obtained from $f(0, \mu)$ for $\mu < 0$. 
\begin{rmk}
The discretization of the velocity space here is similar to the double-Pn method developed in the literature, see e.g.,~\cite{thompson_theory_1963}. However, the double-Pn method fully separates out the positive and negative part, with each piece expanded by its own set of basis functions. In our formulation, we perform even-odd parity decomposition. Upon this decomposition we have rigorous well-posedness and error analysis provided in~\cite{LiLuSun}.
\end{rmk}
\subsection{Pure fluid system}\label{sec:pure_fluid}
We now consider the pure fluid system, i.e.~$\sigma = 1$ on the entire
domain $[a,b]$, and we approximate~\eqref{eq:kinetic} by the interior
heat equation:
\begin{equation}\label{eqn:pure_fluid_heat}
\begin{cases}
\partial_t\theta = \lambda\partial_{xx}\theta \,, \\
\theta(t,a) = \theta_a(t), \quad \theta(t,b) = \theta_b(t),
\end{cases}
\end{equation}
where $\theta_a$ and $\theta_b$ are the steady state of the half space problem associated with the two boundary layers, given by \eqref{eq:half-space-a} and \eqref{eq:half-space-b}. Numerically, they are obtained using the method in Section~\ref{sec:albedo}. 

To numerically solve \eqref{eqn:pure_fluid_heat}, we use a standard
finite difference scheme.  Take $N_x$ equally distanced grid points
with mesh size $\Delta{x} = (b - a)/N_x$ such that
$a+\frac{\Delta{x}}{2} = x_1 < x_2 <\cdots x_{N_x} =
b-\frac{\Delta{x}}{2}$. As the boundary layer of the original kinetic
equation has width $\epsilon \ll 1$, the grid point $x_1$
(resp. $x_{N_x}$ lies well outside of the layer when $\epsilon$ is
small. Hence its value is approximately given by $\theta_a$
(resp. $\theta_b$), the infinite limit of the half-space problem.

We approximate the solution by $\theta^n_i \approx \theta(t^n, x_i)$
for $i=1,\cdots,N_x$, and use a standard
implicit central finite difference scheme to update the solution
\begin{equation}\label{scheme:heat}
\theta^{n+1}_i-\theta^n_i= \frac{\lambda\Delta t}{\Delta{x}^2}\left[\theta^{n+1}_{i+1}-2\theta^{n+1}_{i}+\theta^{n+1}_{i-1}\right],
\end{equation}
with two Dirichlet boundary conditions $\theta^{n+1}_1 = \theta_a(t^{n+1})$ and $\theta^{n+1}_{N_x} = \theta_a(t^{n+1})$ given by the half space problem solver in Section~\ref{sec:albedo}.

\subsection{Coupled system}\label{sec:coupled}
For the coupled system we approximate~\eqref{eq:kinetic} by a coupled kinetic-heat system
\begin{subequations}
\begin{align}\label{scheme:f_left}
  &\epsilon\partial_tf + \mu \partial_xf + \frac{\sigma}{\epsilon}\mathcal{L}f = 0, & x \in (a, x_m),\, \mu \in [-1, 1] \\
  &f(t,a,\mu) = \phi_a(t,\mu),& \mu>0\\
  &f(t,x_m,\mu) = \mathcal{R}\bigl(f(t,x_m, \cdot) \vert_{\mu > 0}\bigr)(\mu),& \mu<0
\end{align}
in the kinetic region $x\in (a, x_m)$ and
\begin{align}\label{scheme:theta_right}
  &\partial_t\theta = \lambda\partial_{xx}\theta, & x \in (x_m, b)\\
  &\theta(t,x_m) = \theta_m(t), \quad\theta(t,b) = \theta_b(t),
\end{align}
in the fluid region $x \in (x_m, b)$.
\end{subequations}

Here $\theta_m$ is defined by the half-space problem that couples together the kinetic and heat equations \eqref{eq:f-b-0-m}. 
Note that the boundary data for the original equation are imposed at
the two ends $x=a$ and $x=b$ for the distribution function $f$. On
the right boundary, we need to convert the boundary condition to
$\theta_b$ for the heat equation. At $x = x_m$,  the heat and kinetic equations are coupled together through
the Albedo operator $\mc{R}f\vert_{x_m}$ and the end-state
$\theta_m(t)$ of the solution to the half-space equation.

To numerically solve the coupled system, note that the kinetic
equation in the kinetic regime~\eqref{scheme:f_left} does not depend on the
fluid regime and is self-contained. We use a standard time-splitting
finite volume scheme for the kinetic equation, where at $x = x_m$, the
boundary condition is given by the Albedo operator, which is
determined numerically using the half-space problem solver in
Section~\ref{sec:albedo}. After the kinetic solution is determined,
the solution of the fluid part is similar as
Section~\ref{sec:pure_fluid}. Hence, we will only focus on the kinetic
regime in the following.

For the kinetic equation, in the $x$ direction, we use a standard
finite volume mesh: we evenly divide the domain $(a, x_m)$ into $N_x$
cells, and denote $x_i,\quad i = 1,\cdots,N_x$, the centers of cells
as the grid points. For the $\mu$ variable, we use the Legendre
quadrature points to discretize the domain. In each time step $t^n$,
assume we have the data $f^n_{i,j}$ for $i = 1,\cdots N_x, j =
1,\cdots N_{\mu}$, the next time step solution is computed as follows:
\begin{itemize}
\item Apply the standard kinetic solver for the kinetic equation~\eqref{scheme:f_left} in the kinetic region with the following boundary conditions:
  \begin{itemize}
  \item $x=a$: $\phi_a(t^{n},\mu_j>0)$,
  \item $x=x_m$: $f^n_{N_x,j}$ with $\mu_j<0$.
  \end{itemize}
  Note that the data $f^n_{N_x,j}$ is already computed in the previous
  step.  The output would be:
  \begin{equation}
    \begin{cases}
      f^{n+1}_{i,j}& i = 2,\cdots N_x-1, \forall j\\
      f^{n+1}_{1,j}& \mu_j<0\\
      f^{n+1}_{N_x,j}& \mu_j>0
    \end{cases}.
  \end{equation}
  We will employ a time splitting scheme: First solve the pure advection part
  \begin{equation}
    \epsilon \, \partial_tf + \mu\partial_xf = 0
  \end{equation}
  with the standard finite volume method: we use Lax-Wendroff scheme
  with van Leer limiters, and ghost cells with transparent boundary
  data is used for the boundary conditions. This step is followed by
  the collision part:
  \begin{equation}
    \epsilon \, \partial_tf + \frac{\sigma}{\epsilon} \mc{L}f = 0,
  \end{equation}
  where Gaussian quadrature is used to numerically calculate the
  integral. Note that in the kinetic regime, $\sigma =
  \mc{O}(\epsilon)$ in the above equation. 

  Note that since the standard finite volume method is used, at the left end $x=a$, only $f_{1,j}$ for $\mu_j<0$ is updated in the scheme
  (corresponding to the left-going waves). Similarly at the right end,
  $i = N_x$, only $f_{N_x,j}$ for $\mu_j>0$ is updated (corresponding
  to the right-going waves). $f^{n+1}_{1,j}$ for $\mu_j>0$ will be
  given by the Dirichlet boundary condition at time $t^{n+1}$:
  \begin{equation}
    f^{n+1}_{1,j} = \phi_a(t^{n+1},\mu_j),\quad \mu_j>0.
  \end{equation}
  and $f^{n+1}_{N_x,j}$ for $\mu_j<0$ comes from the Albedo operator
  $\mc{R}$, this will be obtained in the next step.
\item Use $f^{n+1}_{N_x,j}$ for $\mu_j>0$ as the incoming boundary
  data for the half-space problem numerically solved the half-space
  problem solver in Section~\ref{sec:albedo}, and obtain two outputs,
  they serve to update
  \begin{itemize}
  \item $\theta^{n+1}_{m}$. This provides the Dirichlet boundary condition at the interface at time $t^{n+1}$ for the fluid part and will be used in the heat equation solver.
  \item $f^{n+1}_{N_x,j}$ for $\mu_j<0$. This updates $f$ at the coupling point for negative velocities.
  \end{itemize}
\end{itemize}

\section{Numerical Example} \label{section:numerical-example} 

In this section we show various numerical tests based on the asymptotics and the algorithm in Section~\ref{section:asymptotics} and~\ref{section:algorithm}. Without loss of generality, we will fix $a = -1$ and $b = 1$. For coupled kinetic-fluid system, we set $x_m = 0$. Results for the pure fluid systems and the couple systems are presented in two separate subsections. For the entire section we choose collision kernel to be: $\kappa(\mu,\mu') = 1/2 + \mu\mu'/4$.

\subsection{Pure heat equation}
In this subsection we treat the pure diffusion case where $\sigma = 1$ throughout the entire domain. We consider examples with different
combination of possible initial, boundary, and initial-boundary
layers. We also show examples where the data for the diffusion
equation are compatible or incompatible. Our main interest is in the
modeling error, that is, to check the convergence rate in terms of
$\epsilon$. To this end, we have refined the mesh size and time
stepping small enough to make sure that the error coming from
numerical approximation is negligible compared to the asymptotic
error.

To compute the heat equation with Dirichlet boundary
condition~\eqref{eqn:pure_fluid_heat}, we take $\Delta x = 10^{-3}$
and $\Delta t = 2.5\times 10^{-4}$. A small mesh size is chosen to
guarantee that the numerical error is negligible, while in practice,
larger mesh size can be used for the heat equation. To compute the
reference solution to the kinetic equation~\eqref{eq:kinetic}, we use
$\Delta x = \min\{5\times 10^{-4},\frac{\epsilon}{25}\}$ to resolve
the boundary layer and $N_{\mu} = 32$ (32 quadrature points for the
$\mu$ variable). For time step, we choose $\Delta t =
\min\{\omega\epsilon\Delta x,\epsilon^2\}$ with the CFL number $\omega
= 0.5$. The solution is computed up to time $T = 0.03$ with various
choice of $\epsilon$.

The difference between the numerical solution of the heat equation and the reference solution to the kinetic equation is measured in $L^2$-norm. Denote $\theta$ and $f$ the solution to the heat and kinetic equations respectively, we consider the four type of error measurements:
\begin{align*}
  & E_{\theta} = \norm{\theta-\vint{f}}_{L^2_x(-1, 1)}; \quad && E_f = \norm{\theta - f}_{L^2_{x\mu}((-1, 1) \times [-1, 1])}; \\
  & E_{\theta, \text{inner}} = \norm{\theta-\vint{f}}_{L^2_x[-0.9, 0.9]}; \quad  && E_{f, \text{inner}} = \norm{\theta - f}_{L^2_{x\mu}([-0.9, 0.9] \times [-1, 1])}.
\end{align*}
It is clear that $E_{\theta} \leq E_f$ and $E_{\theta, \text{inner}} \leq E_{f, \text{inner}}$ by Minkowski inequality. The inner error is calculated away from the boundary to avoid the impact from the boundary layer that the heat equation does not resolve (though the limit behavior of the boundary layer is captured through solving the half-space problem). The errors are functions of $\Eps$, we will use $1/\Eps = 32,64,128,256$ in the examples below to study the dependence of the asymptotic error on $\epsilon$. For each numerical test below, we plot four figures
\begin{itemize}
\item top left: profile of the solutions with different $\Eps$'s; 
\item top right: profile of the solutions zoomed in near the left
  boundary $x = -1$;
\item bottom left: convergence rate of $E_\theta$ and $E_f$;
\item bottom right: convergence rate of $E_{\theta,\text{inner}}$ and $E_{f,\text{inner}}$.
\end{itemize}

For readers' convenience, let us recall that the kinetic equation and
the approximating heat equation are
\begin{align*}
   \Eps \, \del_t f + \mu &\del_x f + \frac{1}{\Eps} \, \CalL f = 0 \,, \\
   f|_{x=-1} &= \phi_{-1} (t, \mu) \,,  \qquad \mu > 0 \,, \\
   f|_{x=1} &= \phi_1 (t, \mu) \,,  \qquad \mu < 0 \,, \\
   f|_{t=0} &= \phi_0 (x, \mu) \,,
\end{align*}
and
\begin{equation*}
\begin{aligned}
    \del_t \theta - \vint{\mu \CalL^{-1} (\mu)} \, &\del_{xx} \theta = 0 \,, \\
    \theta|_{x=-1} = \theta_{-1} \,, \qquad  &\theta |_{x=1} = \theta_1 \,, \\
    \theta|_{t=0} = \theta_i(x) = \vint{\phi_0&}(x) \,.
\end{aligned}
\end{equation*}
We also recall the definition of the end-state: let $g$ be
the solution to the half-space equation with incoming condition $\phi
= \mu$, and the end-state $\eta$ associated with $\CalL$ satisfies
\begin{align*}
    \mu \del_y g + &\CalL g = 0 \,, \\
    g \big|_{y=0} &= \mu \,, \qquad \mu > 0 \,, \\
    g \to &\, \eta \,, \hspace{1.2cm} \text{as $y \to \infty$} \,.
\end{align*}

We have tested our numerical schemes for the pure diffusive scaling
system for six cases. The first four cases have compatible data for the heat equation (see \eqref{cond:compatibility1}), while in the last two cases the data for the heat equation is incompatible. 

\noindent 
\textbf{Test 1: No Layer.}
In the first test, we study an example with no layer present at the leading order. The data for the kinetic equation are given as
\begin{equation}
\begin{cases}
   \text{boundaries:}\quad &\phi_{-1} = 0,\quad \phi_{1}= 0\,, 
\\[2pt]
   \text{initial:} \quad &\phi_0(x,\mu) = \sin(\pi x) \,.
\end{cases}
\end{equation} 
The resulting data for the heat equation are
\begin{equation}
\begin{cases}
   \text{boundaries:}\quad &\theta_{-1}(t) = 0,\quad \theta_1(t)= 0\,, 
\\[2pt]
   \text{initial:} \quad &\theta_i(x) = \sin(\pi x) \,.
\end{cases}
\end{equation} 

\noindent 
\textbf{Test 2: Initial Layer Only.}
The second test treat the case which only has an initial layer. The data for the kinetic equation are
\begin{equation}
\begin{cases}
   \text{boundaries:}\quad &\phi_{-1} = 0,\quad \phi_{1}= 0\,, 
\\[2pt]
   \text{initial:} \quad &\phi_0(x,\mu) = \sin(\pi x)(1+0.5|\mu|) \,.
\end{cases}
\end{equation} 
The resulting data for the heat equation are
\begin{equation}
\begin{cases}
   \text{boundaries:}\quad &\theta_{-1}(t) = 0,\quad \theta_1(t)= 0\,, 
\\[2pt]
   \text{initial:} \quad &\theta_i(x) = \frac{5}{4}\sin(\pi x) \,.
\end{cases}
\end{equation} 

\noindent 
\textbf{Test 3: Boundary Layer Only.}
The third example is to check a case with only the boundary layer. No initial layer or initial-boundary layer is present. The data for the kinetic equation are
\begin{equation}
\begin{cases}
   \text{boundaries:}\quad &\phi_{-1} = 1.5+100t|\mu|,\quad \phi_{1}= 1.5+100t|\mu|\,, 
\\[2pt]
   \text{initial:} \quad &\phi_0(x,\mu) = \sin(\pi x) +1.5\,.
\end{cases}
\end{equation} 
The resulting data for the heat equation are
\begin{equation}
\begin{cases}
   \text{boundaries:}\quad &\theta_{-1}(t) = 1.5+100t\eta,\quad \theta_1(t)= 1.5+100t\eta\,, 
\\[2pt]
   \text{initial:} \quad &\theta_i(x) = \sin(\pi x) +1.5\,.
\end{cases}
\end{equation} 

\noindent 
\textbf{Test 4: All Layers.}
In the fourth test, we study an example with the initial layer, boundary layer, and initial-boundary layer all present. The data for the kinetic equation are
\begin{equation}
\begin{cases}
   \text{boundaries:}\quad &\phi_{-1} = |\mu|(1+100t),\quad \phi_{1}= |\mu|(1+100t)\,, 
\\[2pt]
   \text{initial:} \quad &\phi_0(x,\mu) = \eta|\mu|+\frac{\eta}{2}\,.
\end{cases}
\end{equation} 
The resulting data for the heat equation are
\begin{equation}
\begin{cases}
   \text{boundaries:}\quad &\theta_{-1}(t) = \eta(1+100t),\quad \theta_1(t)= \eta(1+100t)\,, 
\\[2pt]
   \text{initial:} \quad &\theta_i(x) = \eta\,.
\end{cases}
\end{equation} 
where $\eta$ is the end-state.

In the first four test, the data for the heat equation is compatible, meaning that
\begin{align} \label{cond:compatibility1}
    \theta_{-1}(t=0) = \theta_i (x=-1) \,,
\qquad
    \theta_1(t=0) = \theta_i (x=1) \,.
\end{align}
In the rest of the tests, we check the cases where the above compatibility does not hold. 

\noindent 
\textbf{Test 5: No Layer.}
In the fifth test, we show a result which has no layer at the leading order but the Dirichlet boundary condition for the heat equation is not compatible in the sense of~\eqref{cond:compatibility1}. In this case, we choose the data for the kinetic equation as
\begin{equation}
\begin{cases}
   \text{boundaries:}\quad &\phi_{-1} = 1,\quad \phi_{1}= 1\,, 
\\[2pt]
   \text{initial:} \quad &\phi_0(x,\mu) = 0\,.
\end{cases}
\end{equation} 
The resulting data for the heat equation are
\begin{equation}
\begin{cases}
   \text{boundaries:}\quad &\theta_{-1}(t) = 1,\quad \theta_1(t)= 1\,, 
\\[2pt]
   \text{initial:} \quad &\theta_i(x) = 0\,.
\end{cases}
\end{equation} 

\noindent 
\textbf{Test 6: All Layers.}
In the sixth test we study the case where all the layers are present and the data for the heat equation is not compatible in the sense of~\eqref{cond:compatibility}. The data for the kinetic equation are
\begin{equation}
\begin{cases}
   \text{boundaries:}\quad &\phi_{-1} = |\mu|,\quad \phi_{1}= |\mu|\,, 
\\[2pt]
   \text{initial:} \quad &\phi_0(x,\mu) = |\mu|\,.
\end{cases}
\end{equation} 
The resulting data for the heat equation are
\begin{equation}
\begin{cases}
   \text{boundaries:}\quad &\theta_{-1}(t) = \eta,\quad \theta_1(t)= \eta\,, 
\\[2pt]
   \text{initial:} \quad &\theta_i(x) = \frac{1}{2} \,.
\end{cases}
\end{equation} 

The numerical results are presented in
Figures~\ref{fig:test1}--\ref{fig:test6}. In all the examples,
good agreement of the solution to the heat equation and the
kinetic equation is observed. We note that in the cases with the
presence of boundary layers (test cases 3, 4, and 6), while the heat
equation does not capture the boundary layer, it captures well the
asymptotic behavior away from the layer, as can be seen from the
zoom-in profile of the solutions near the boundary. More
quantitatively, we observe that the convergence rate of the asymptotic
error matches well the error analysis in the Appendix. Also note that
for test 6, due to the incompatibility of the data for the heat
equation, we observe a slower convergence rate for the error on the
whole domain, which is also the case with the asymptotic analysis.

\begin{figure}[ht]
\includegraphics[height = 1.8in]{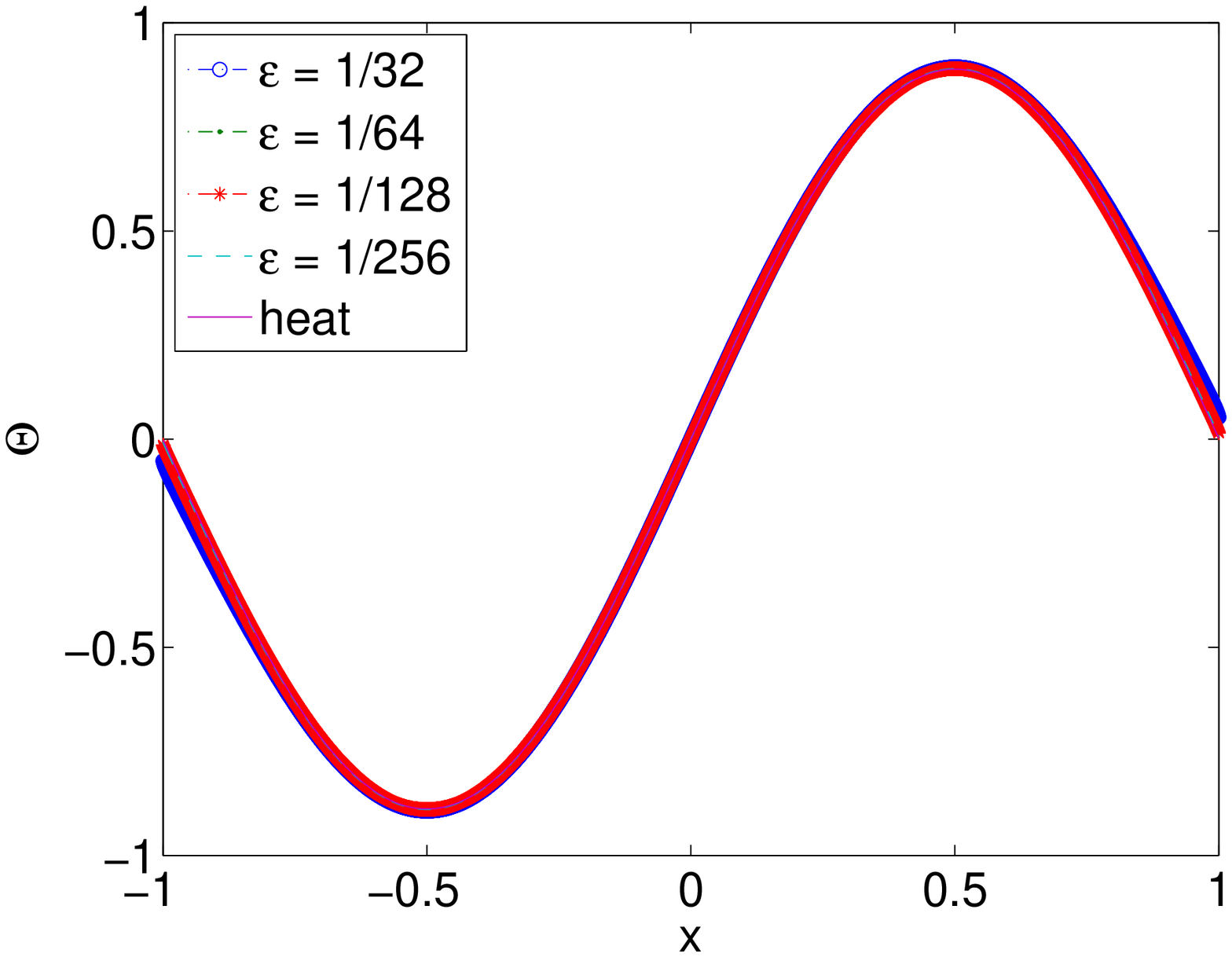}
\includegraphics[height = 1.8in]{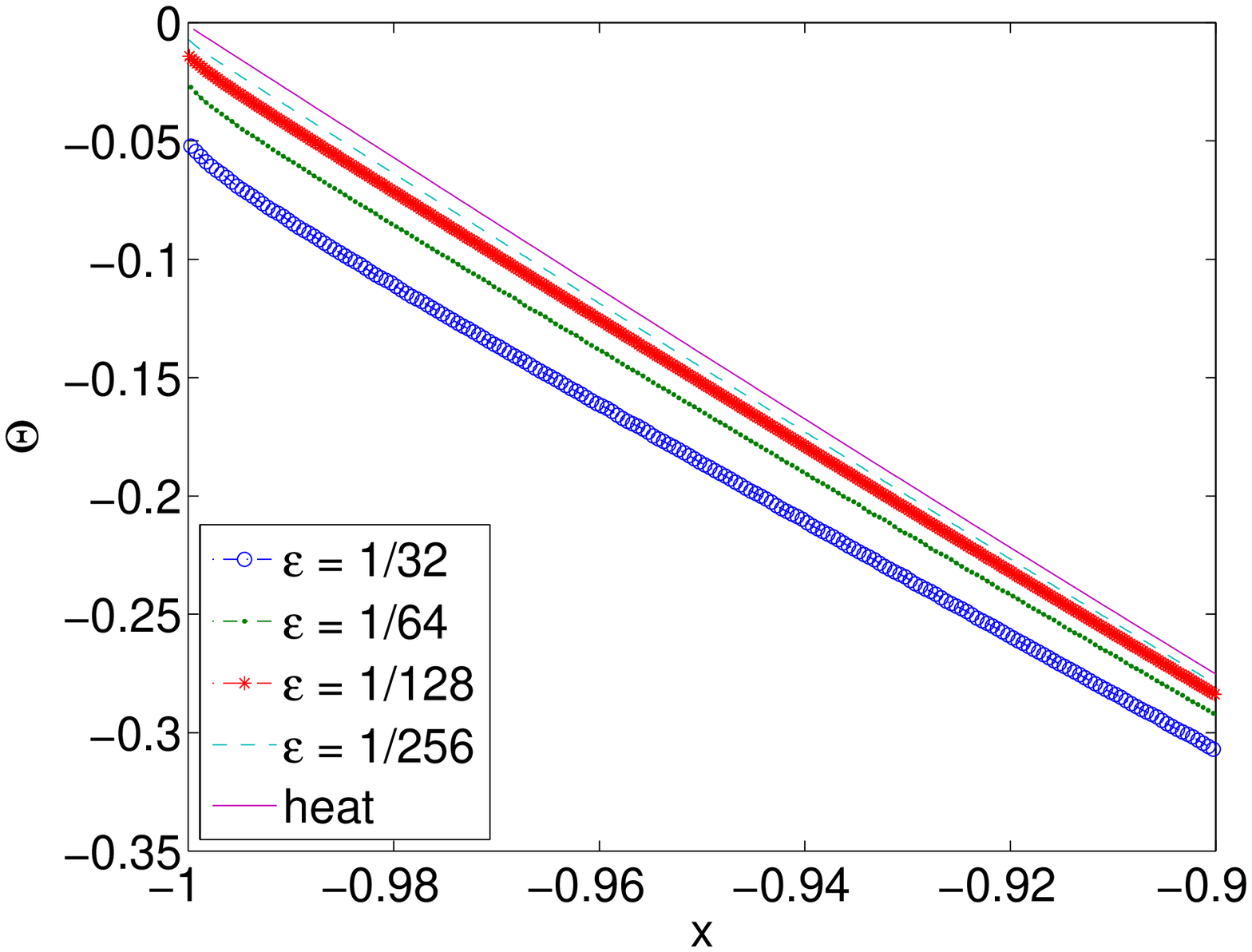} \\
\includegraphics[height = 1.8in]{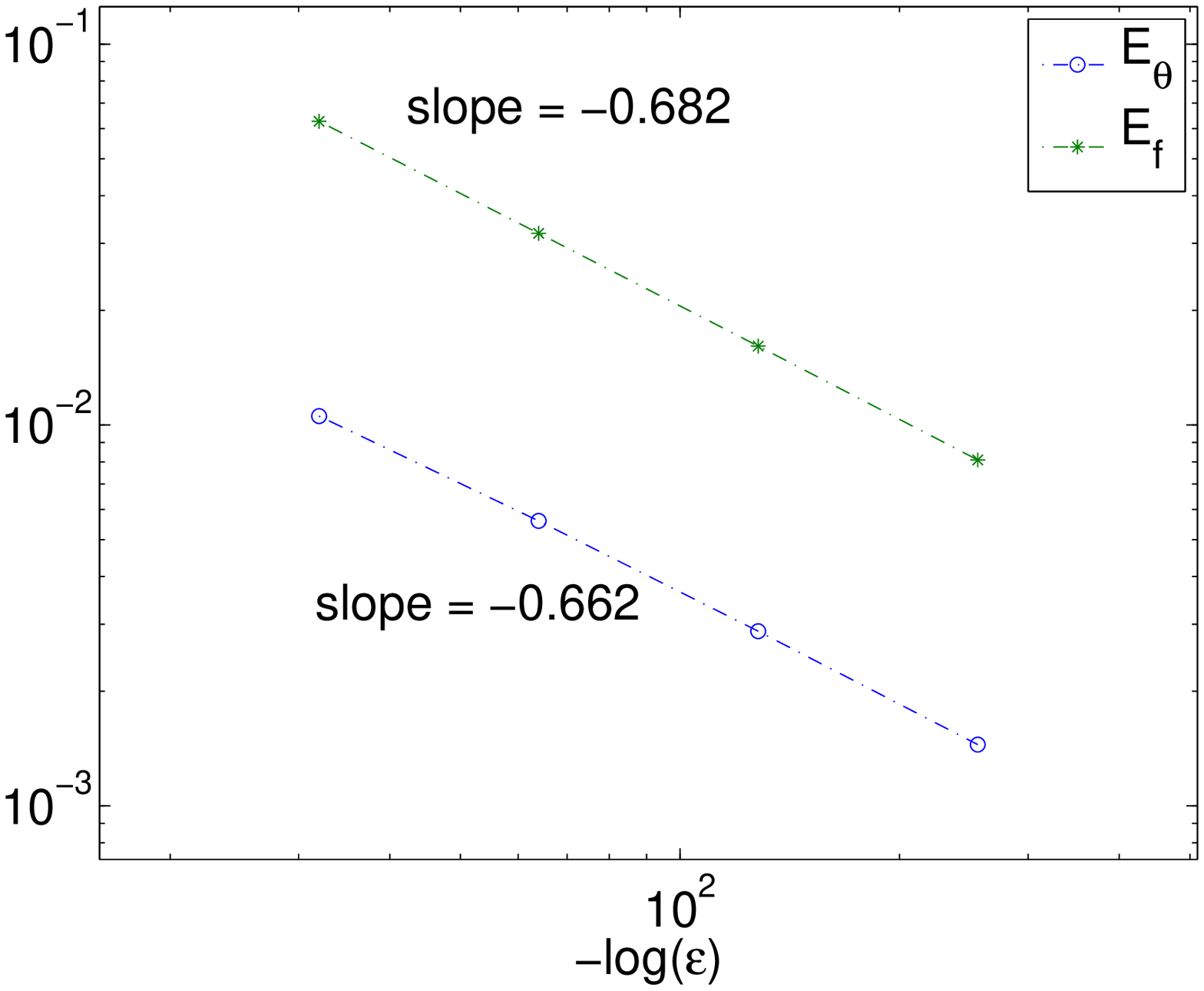}
\includegraphics[height = 1.8in]{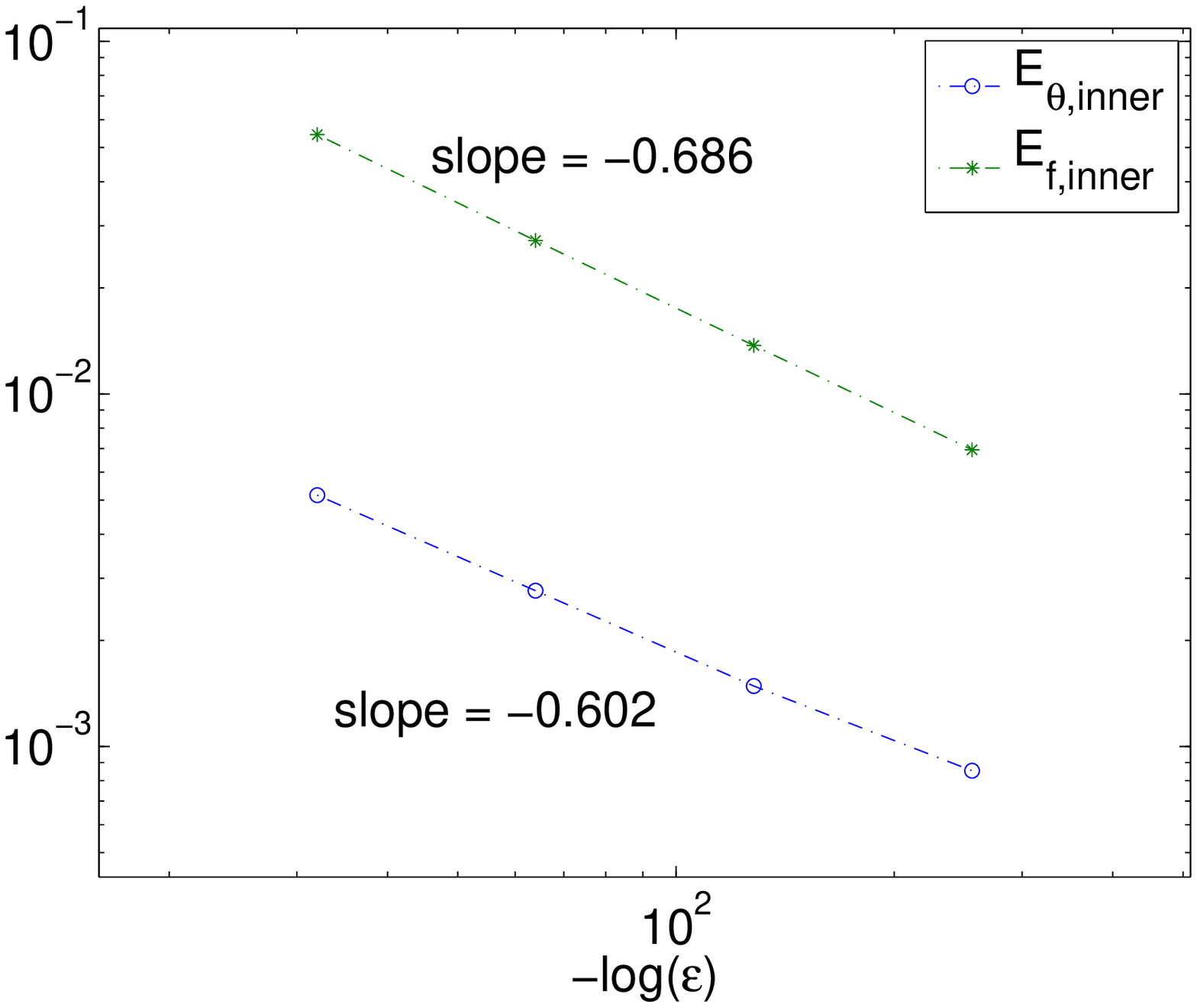}
\caption{Pure heat system, test 1: compatible heat boundary condition. No layer presents.}\label{fig:test1}
\end{figure}

\begin{figure}[ht]
\includegraphics[height = 1.8in]{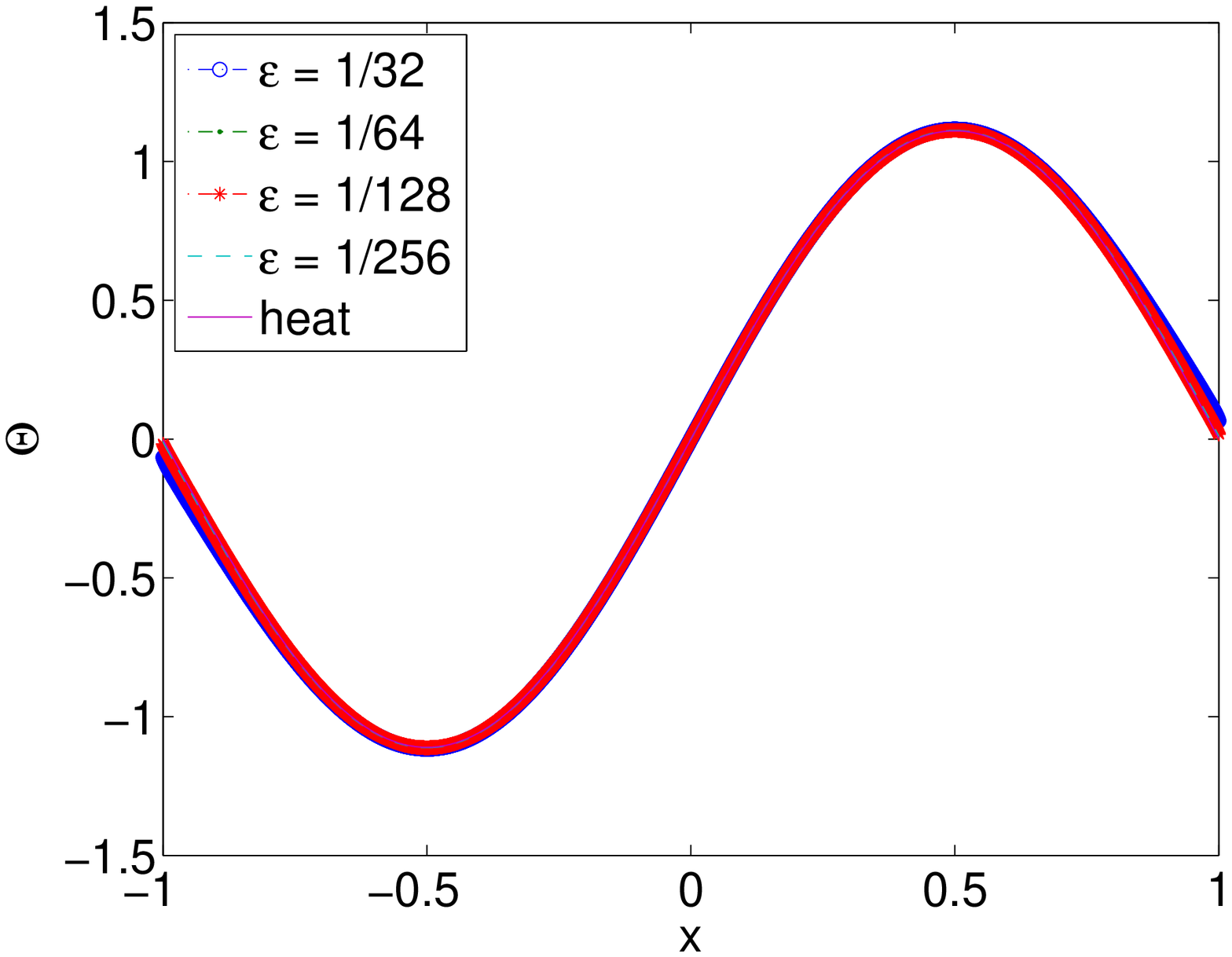}
\includegraphics[height = 1.8in]{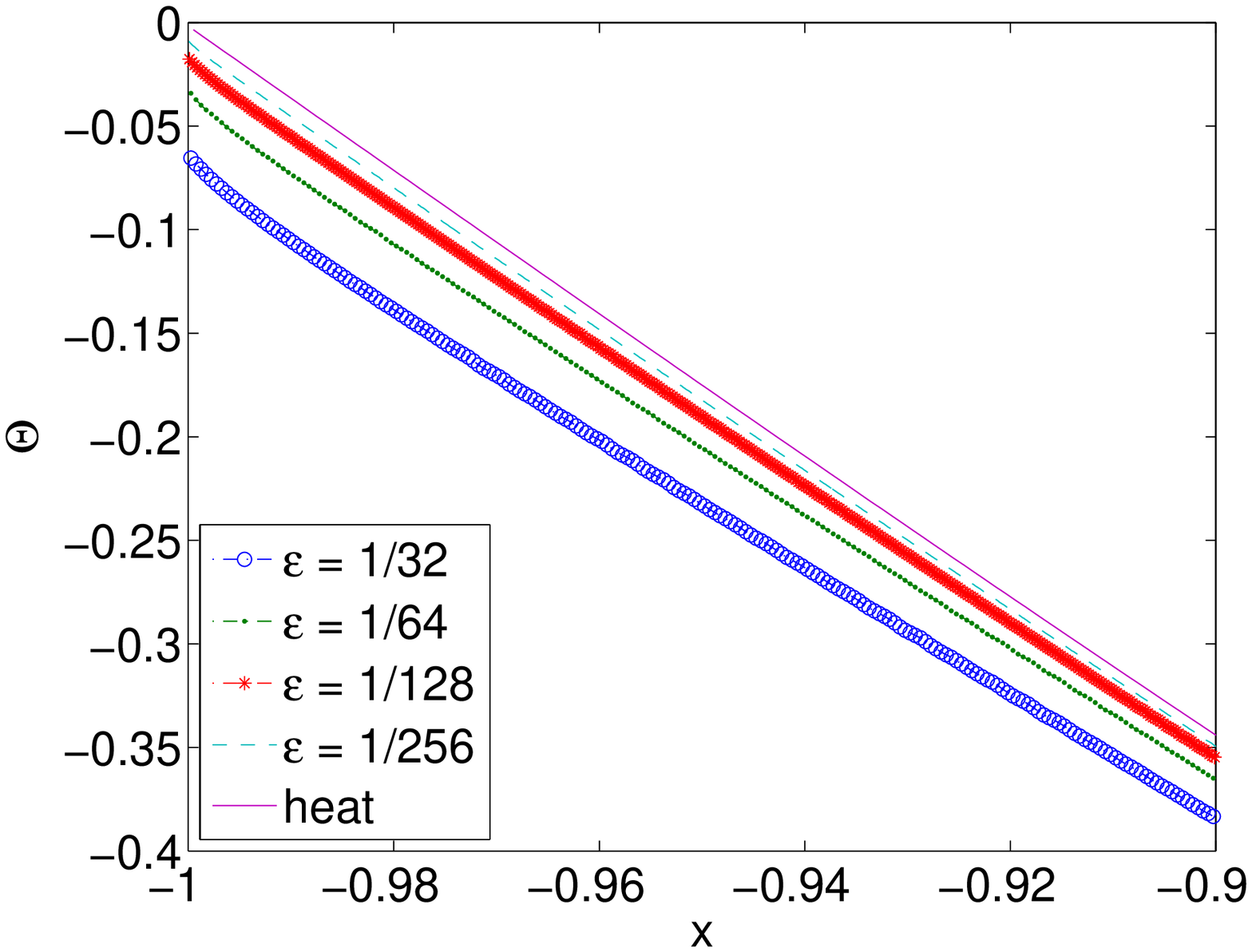}\\
\includegraphics[height = 1.8in]{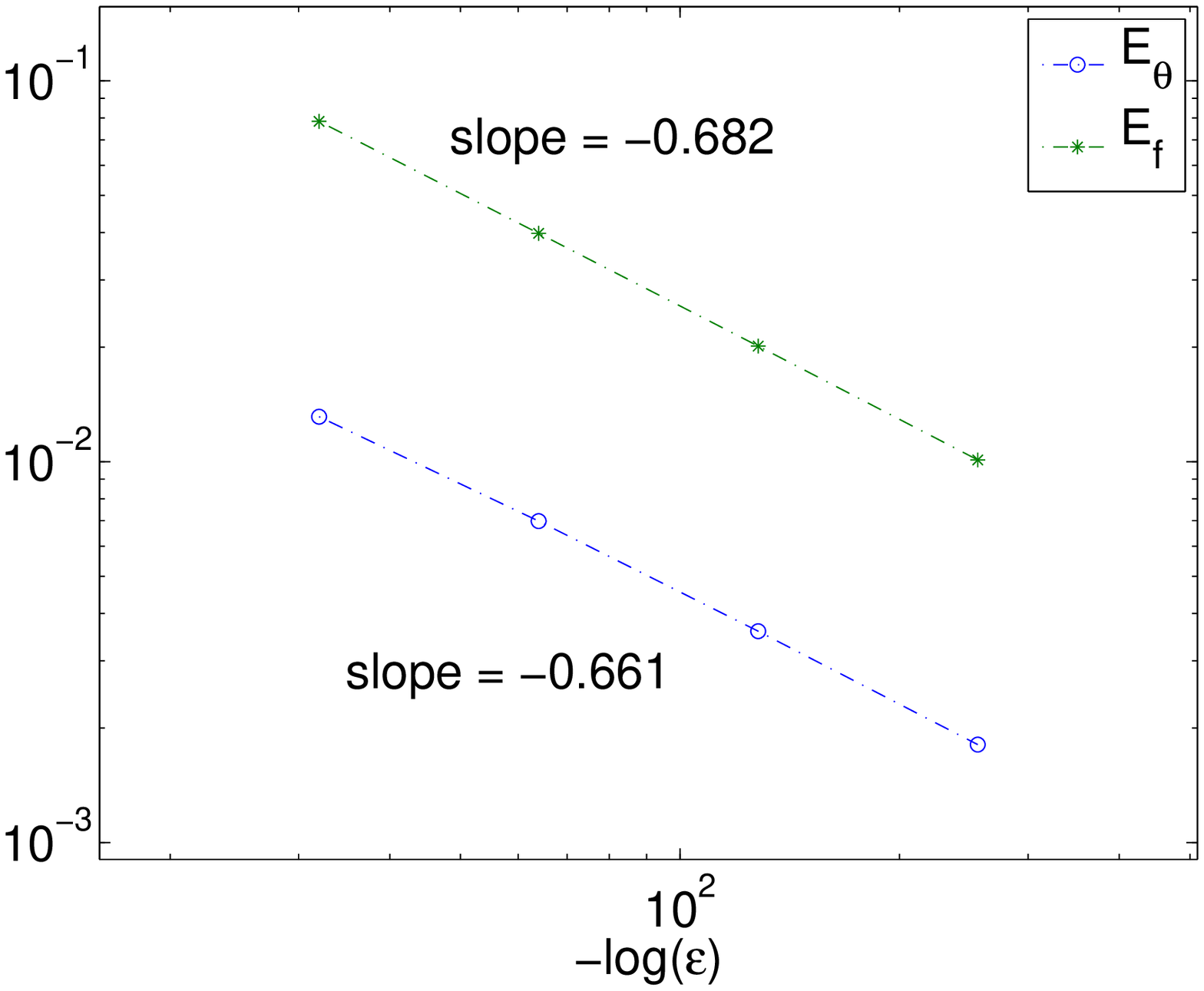}
\includegraphics[height = 1.8in]{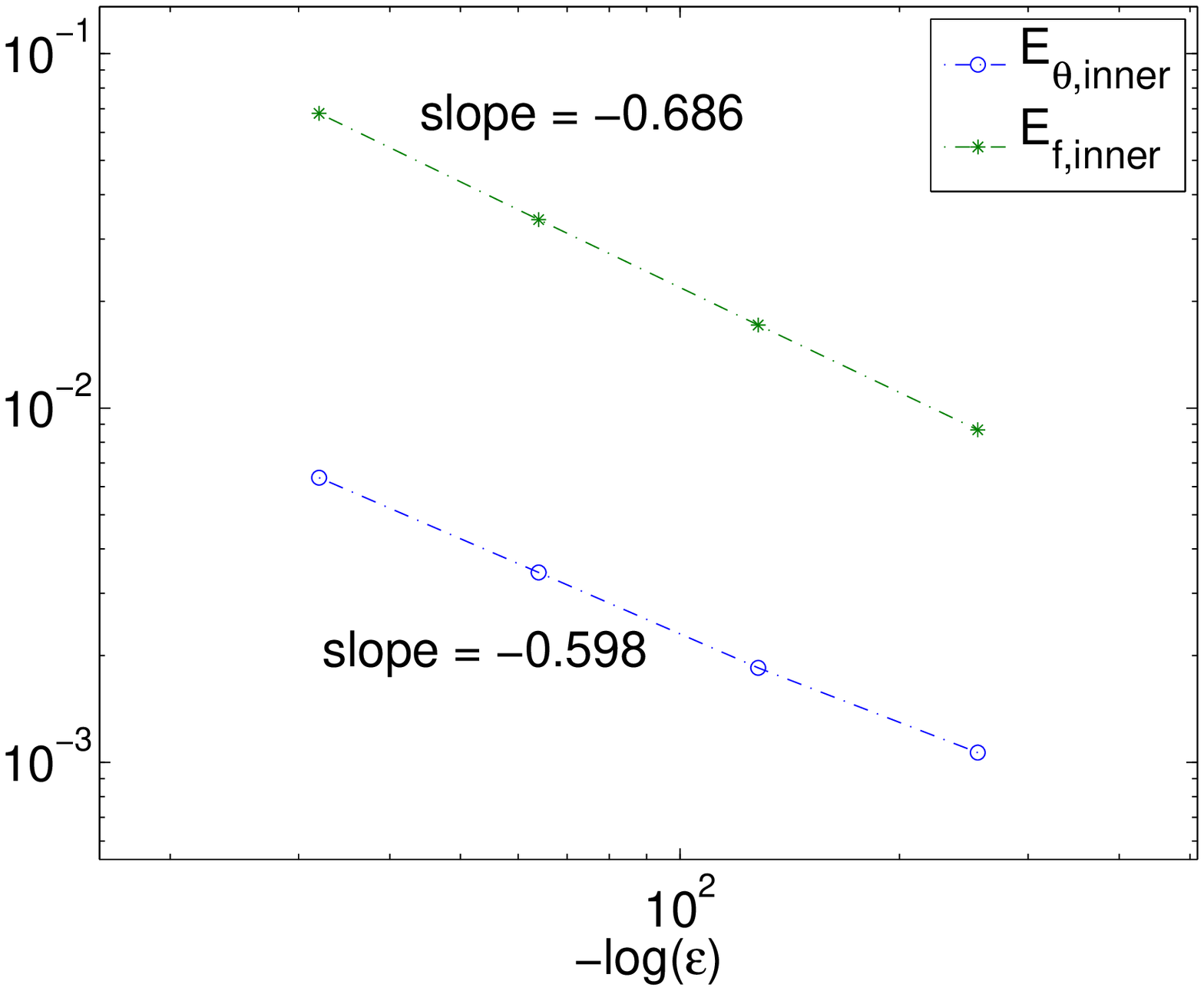}
\caption{Pure heat system, test 2: compatible heat boundary condition. Only initial layer presents.}
\end{figure}

\begin{figure}[ht]
\includegraphics[height = 1.8in]{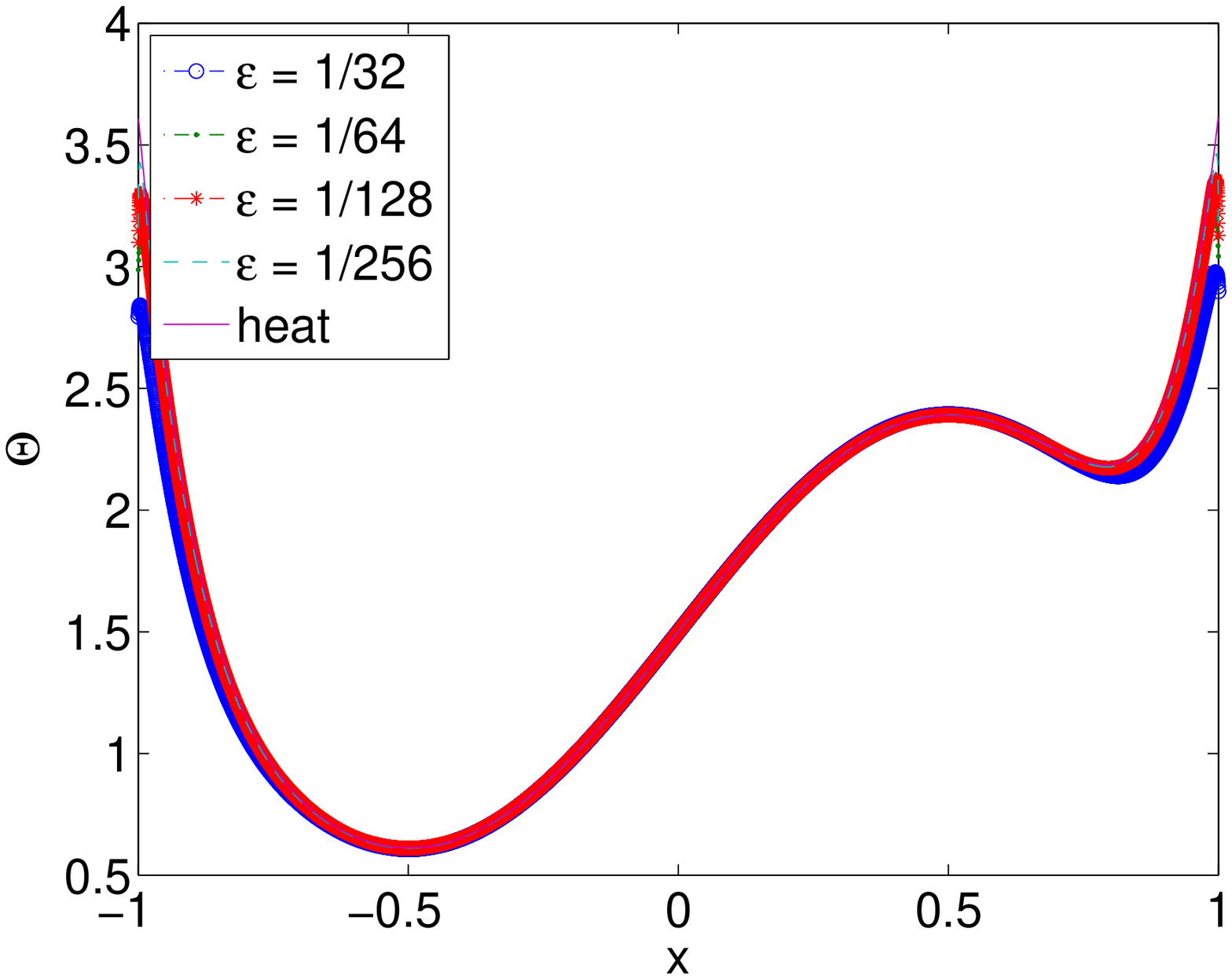}
\includegraphics[height = 1.8in]{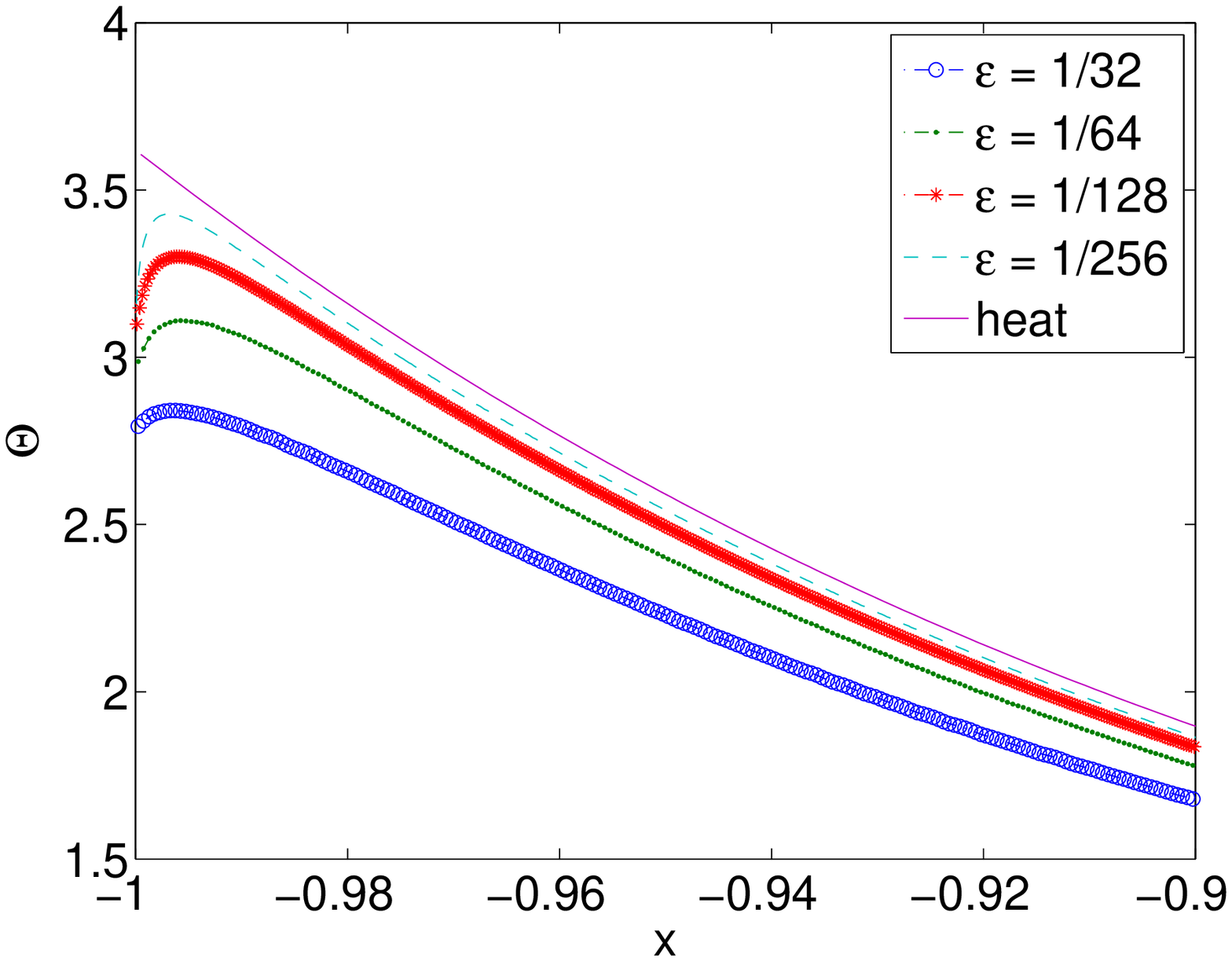}\\
\includegraphics[height = 1.8in]{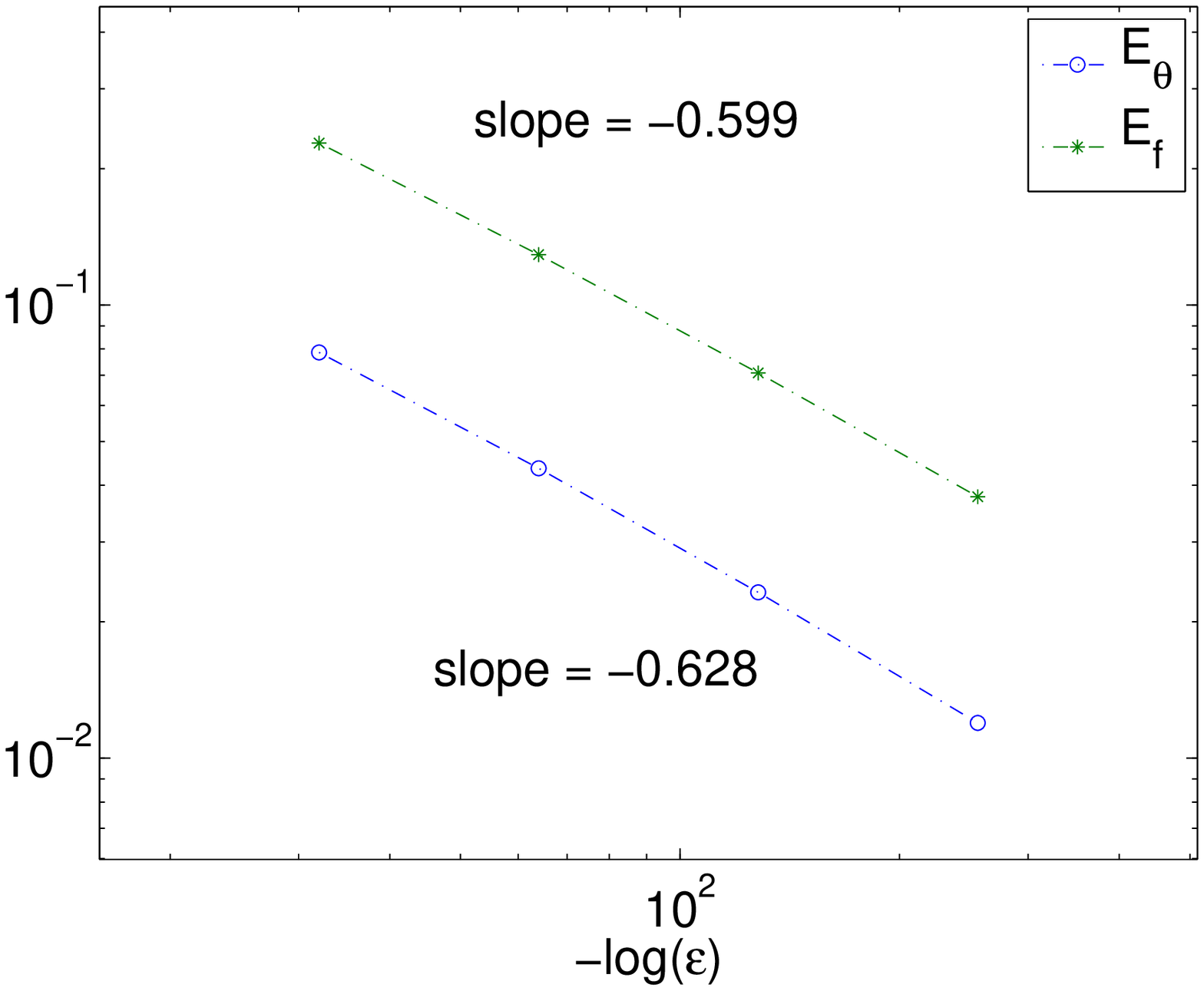}
\includegraphics[height = 1.8in]{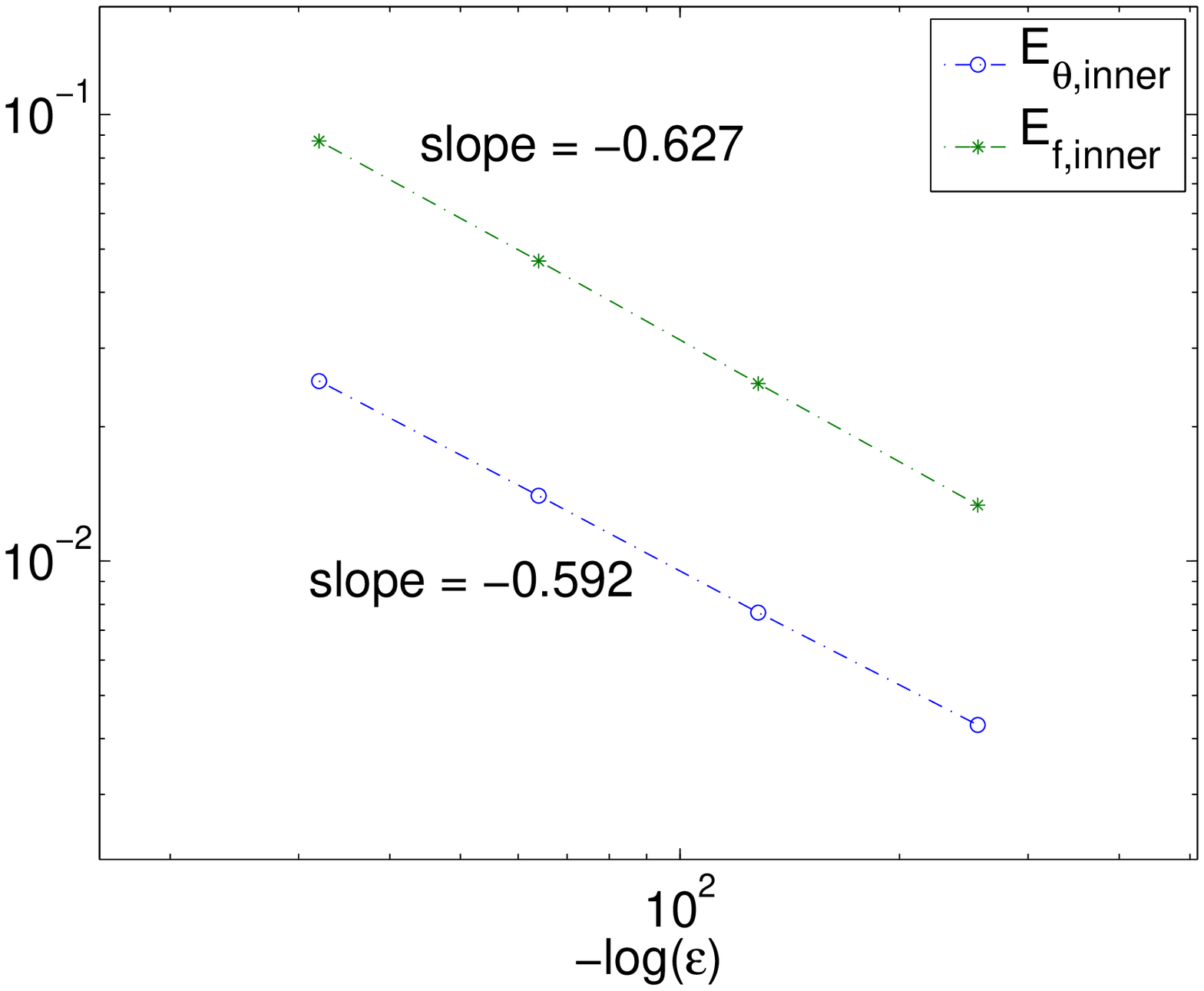}
\caption{Pure heat system, test 3: compatible heat boundary condition. Only boundary layer presents.}
\end{figure}

\begin{figure}[ht]
\includegraphics[height = 1.8in]{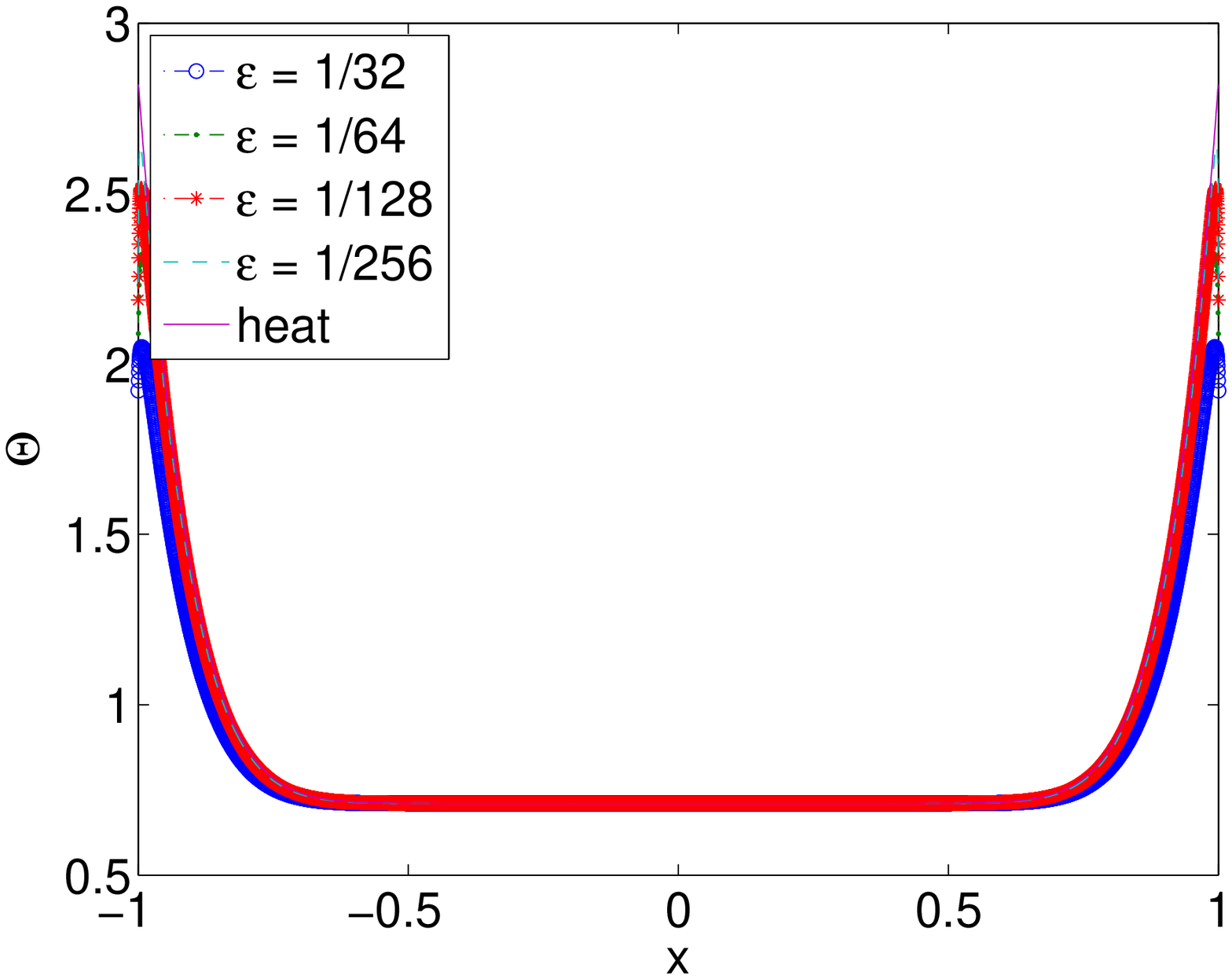}
\includegraphics[height = 1.8in]{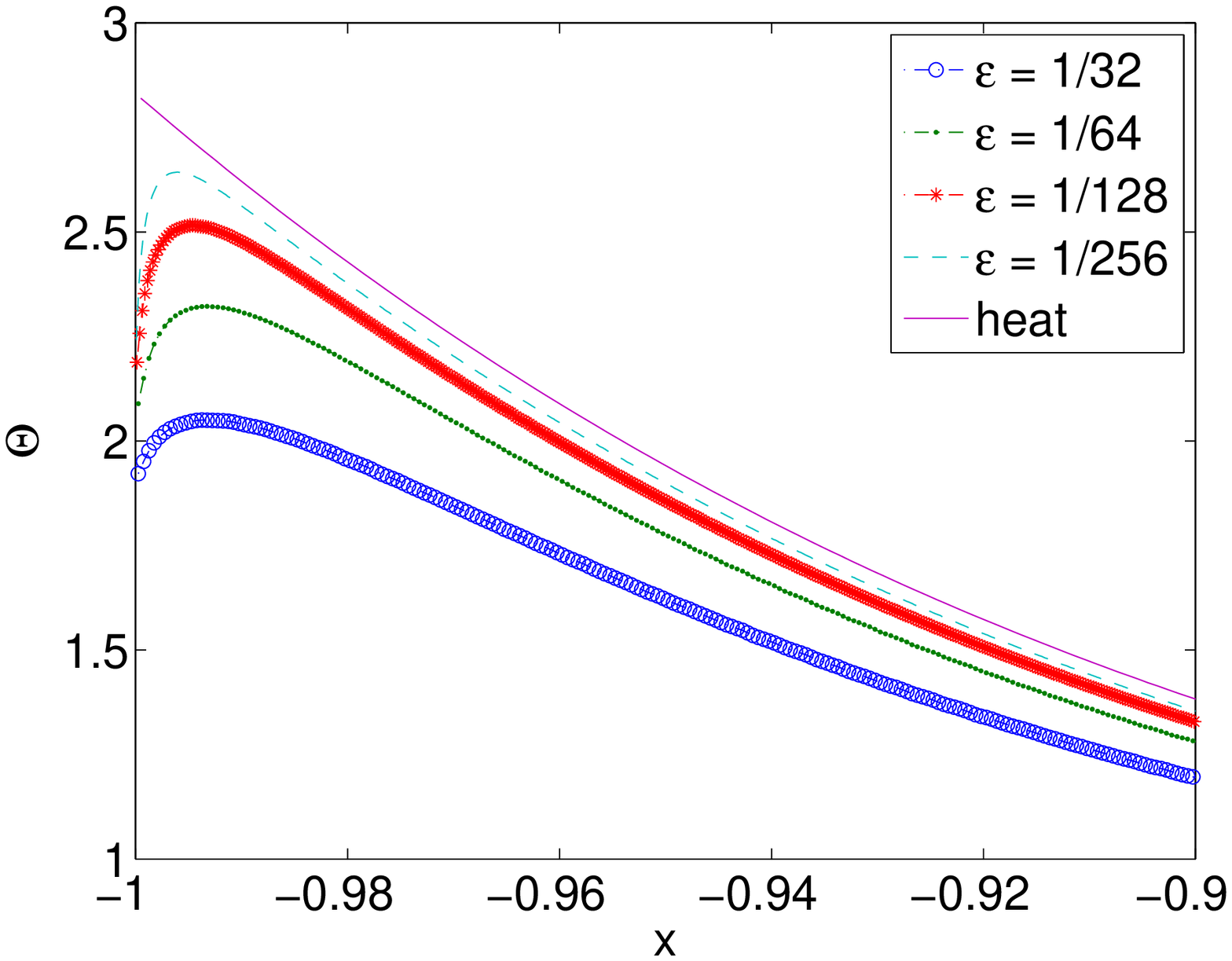}\\
\includegraphics[height = 1.8in]{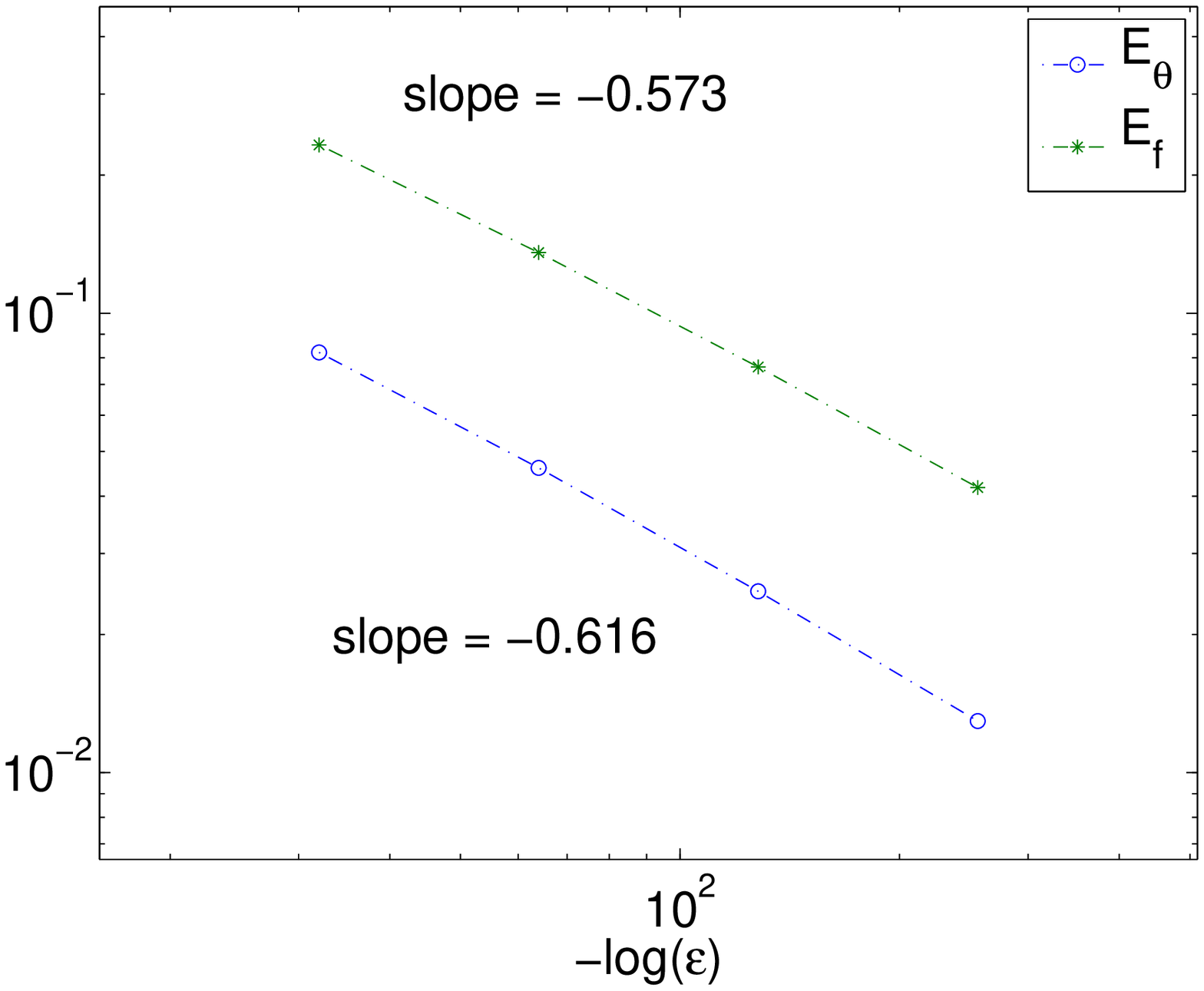}
\includegraphics[height = 1.8in]{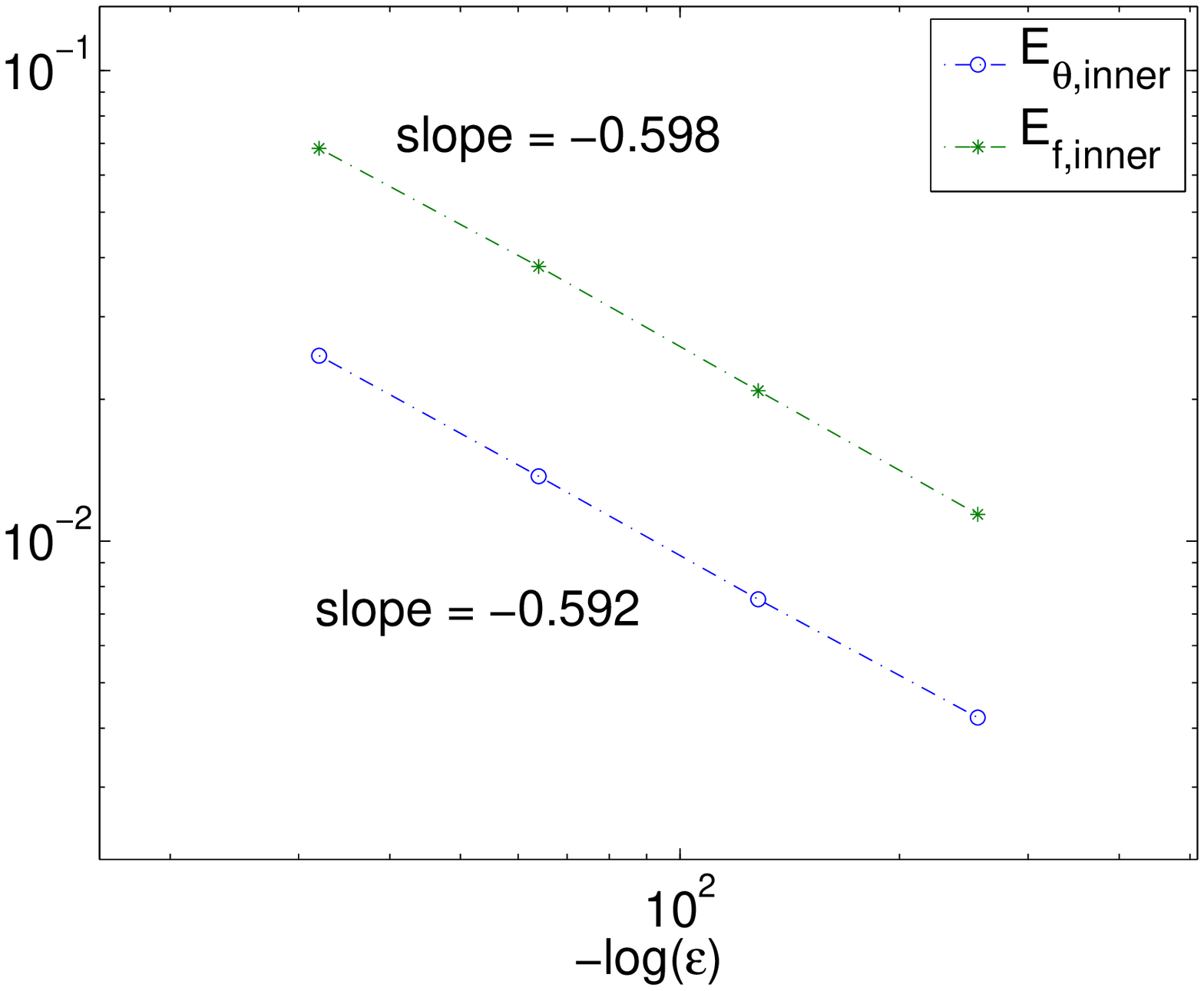}
\caption{Pure heat system, test 4: compatible heat boundary condition. All three types of layers: boundary layer, initial layer and the initial boundary layer are included.}
\end{figure}

\begin{figure}[ht]
\includegraphics[height = 1.8in]{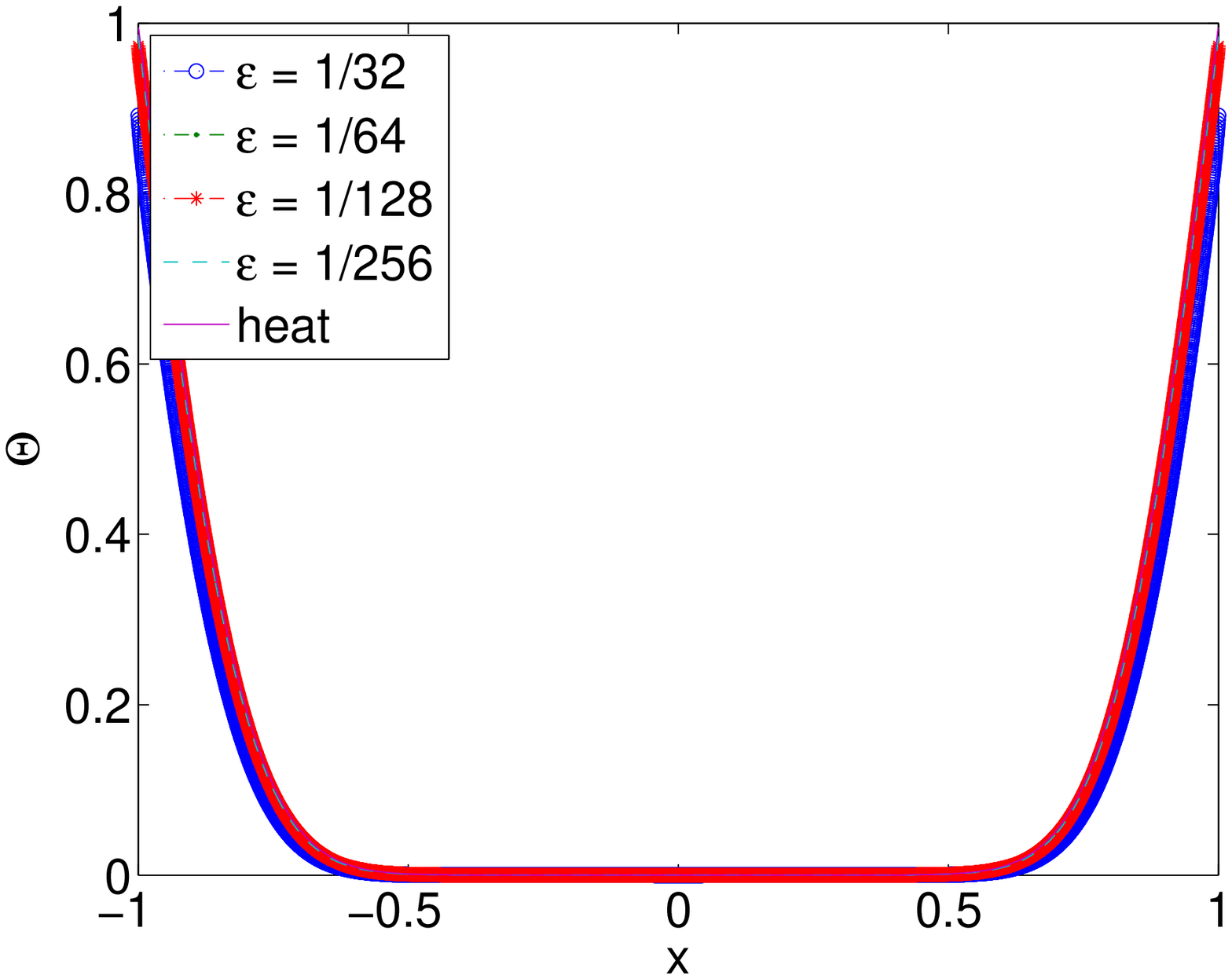}
\includegraphics[height = 1.8in]{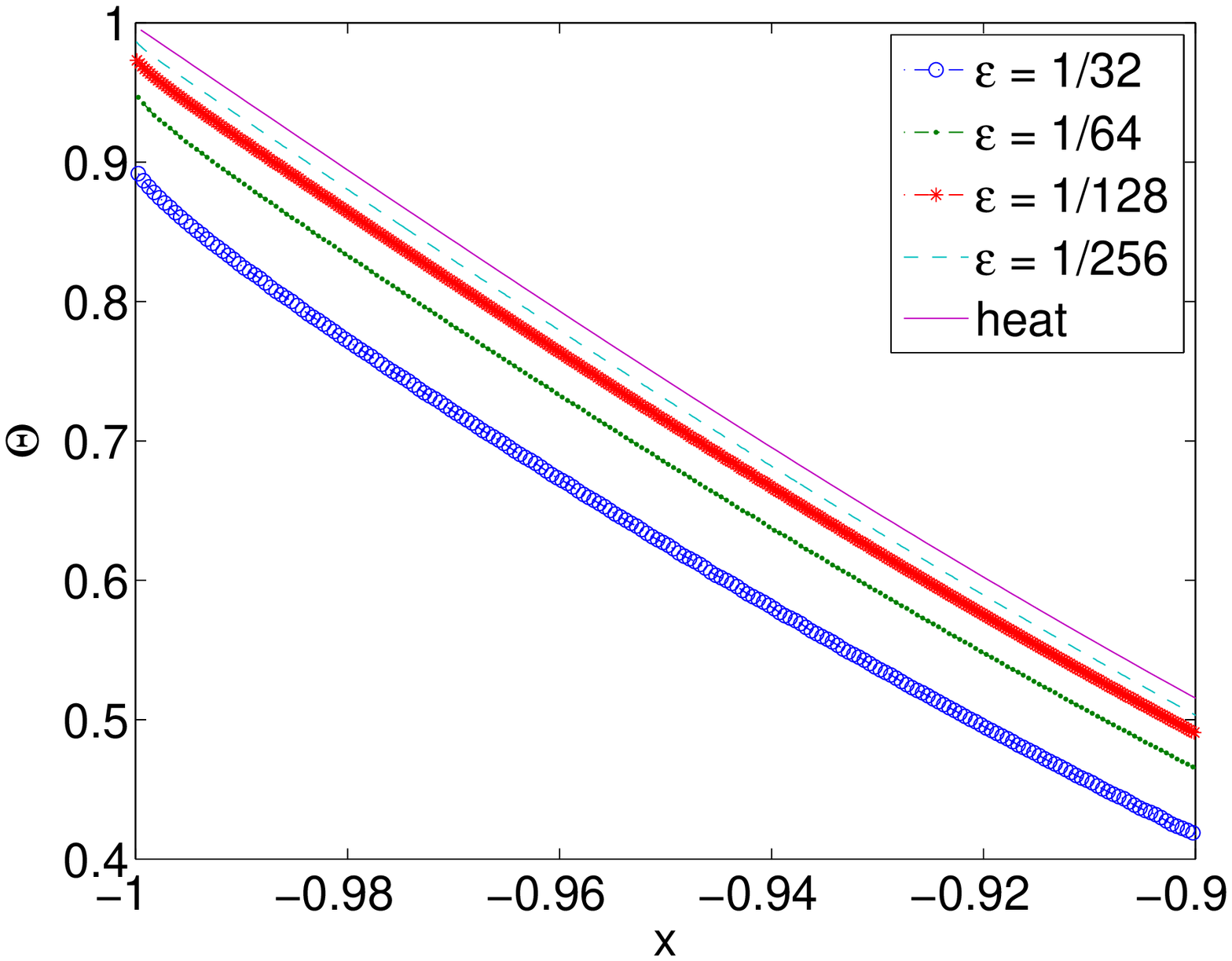}\\
\includegraphics[height = 1.8in]{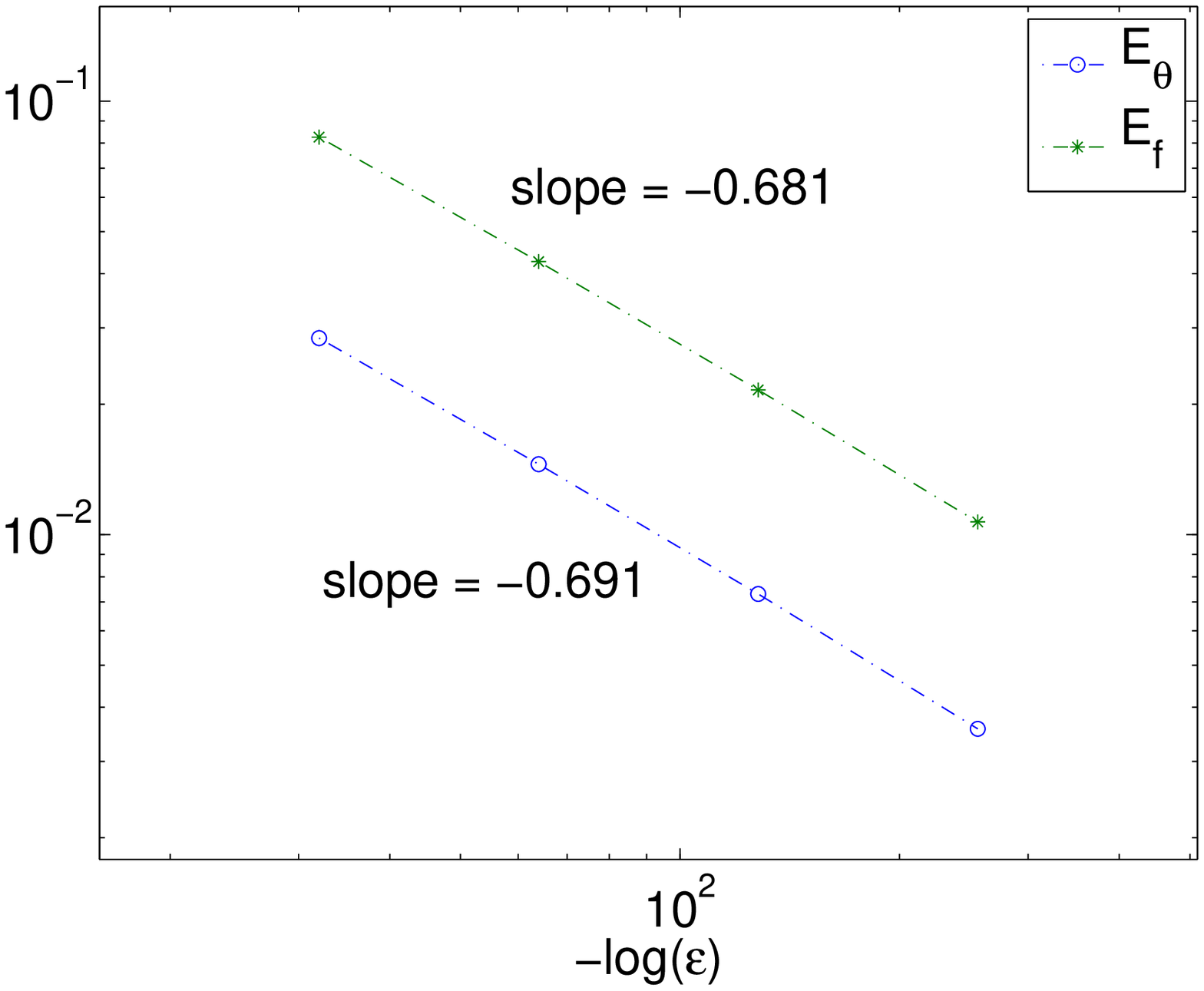}
\includegraphics[height = 1.8in]{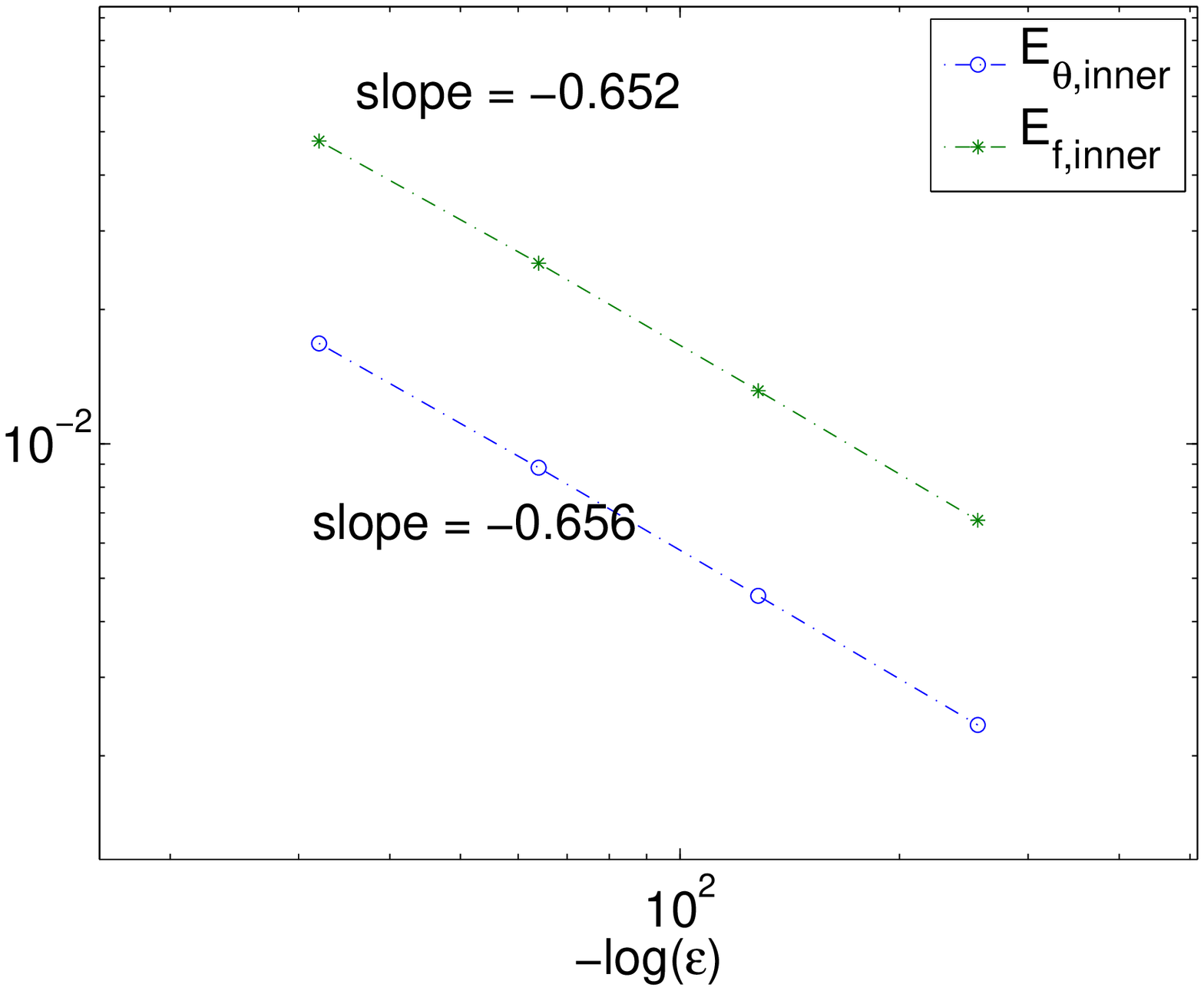}
\caption{Pure heat system, test 5: initially the boundary condition for the heat equation is not compatible. No layer presents.}
\end{figure}

\begin{figure}[ht]
\includegraphics[height = 1.8in]{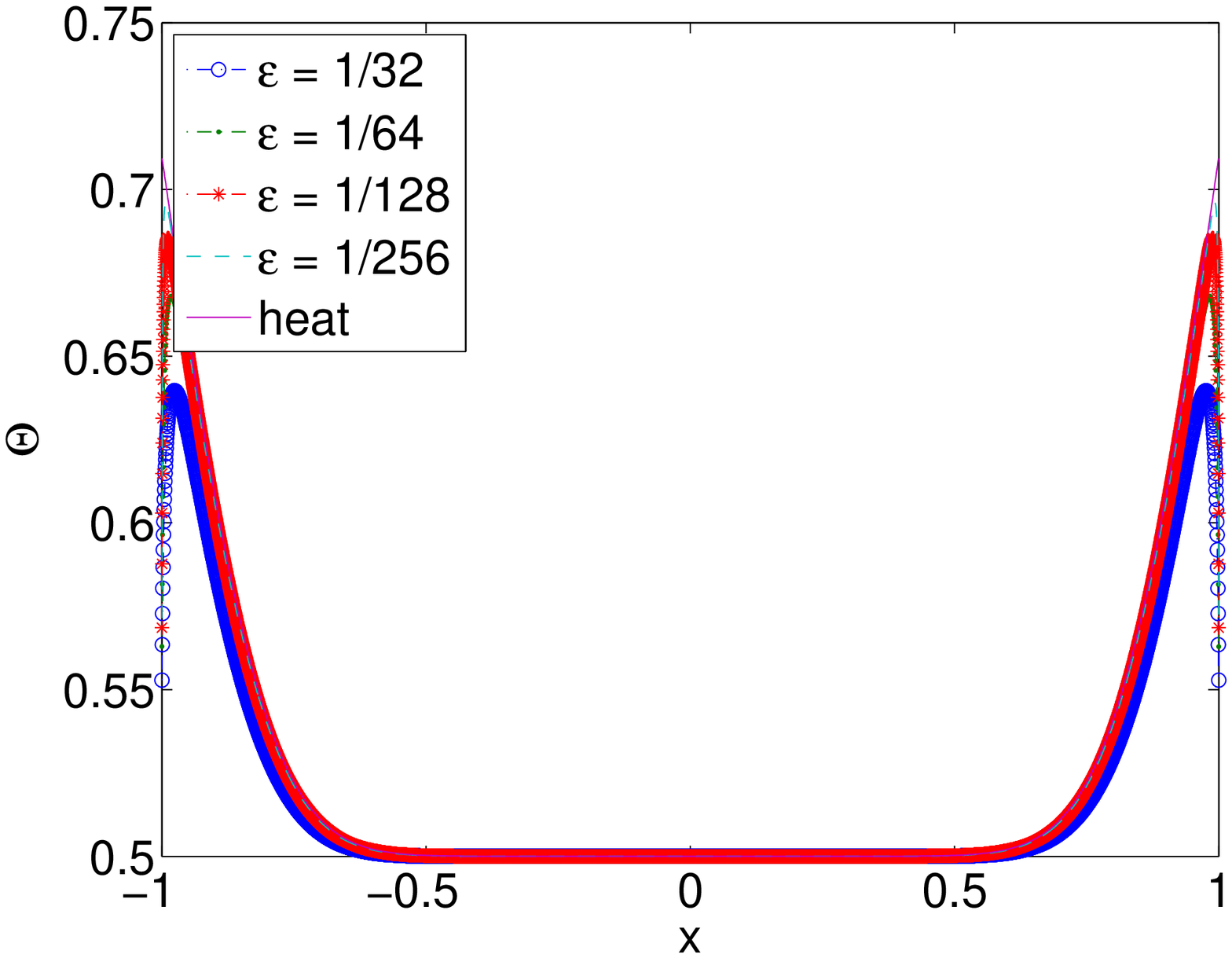}
\includegraphics[height = 1.8in]{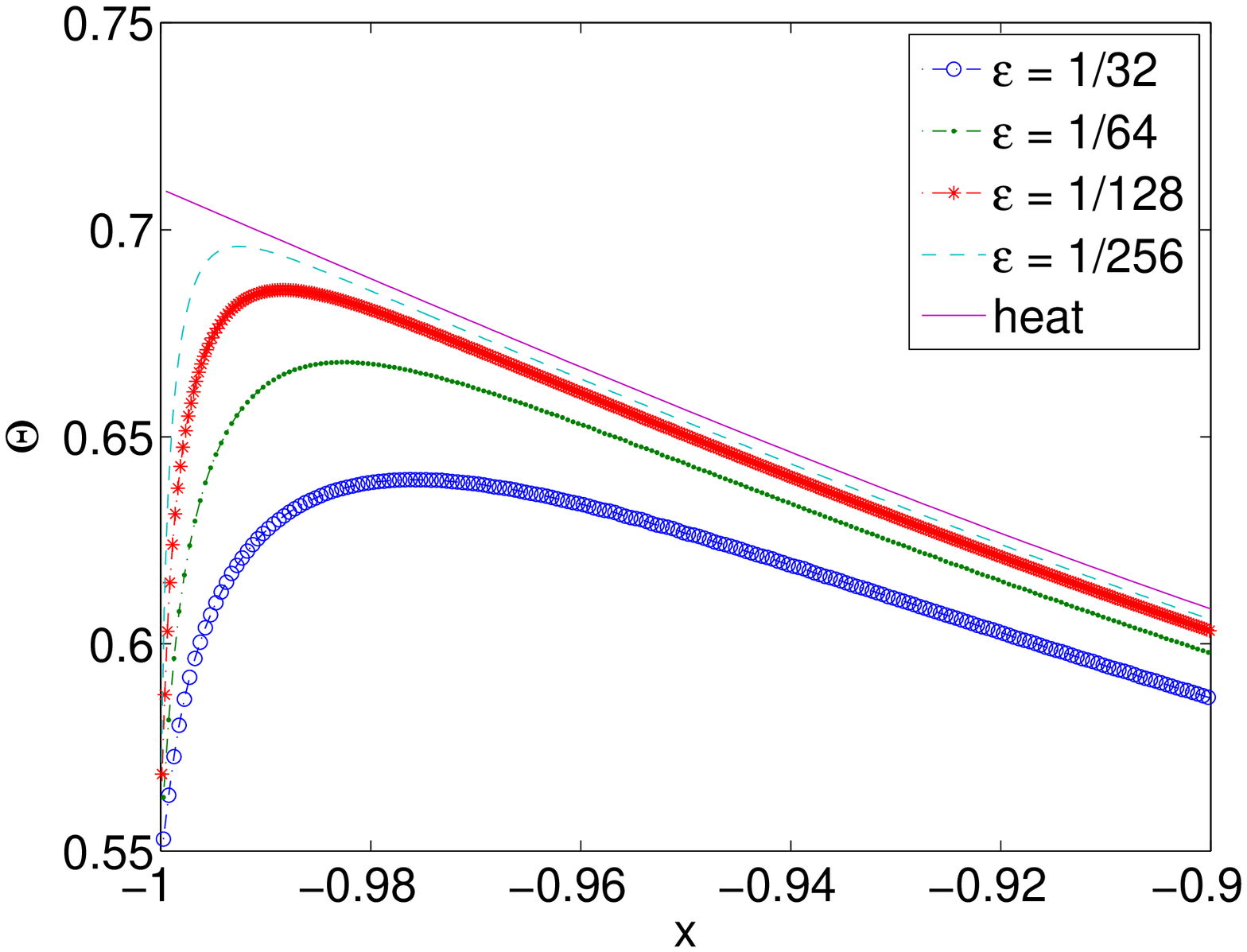}\\
\includegraphics[height = 1.8in]{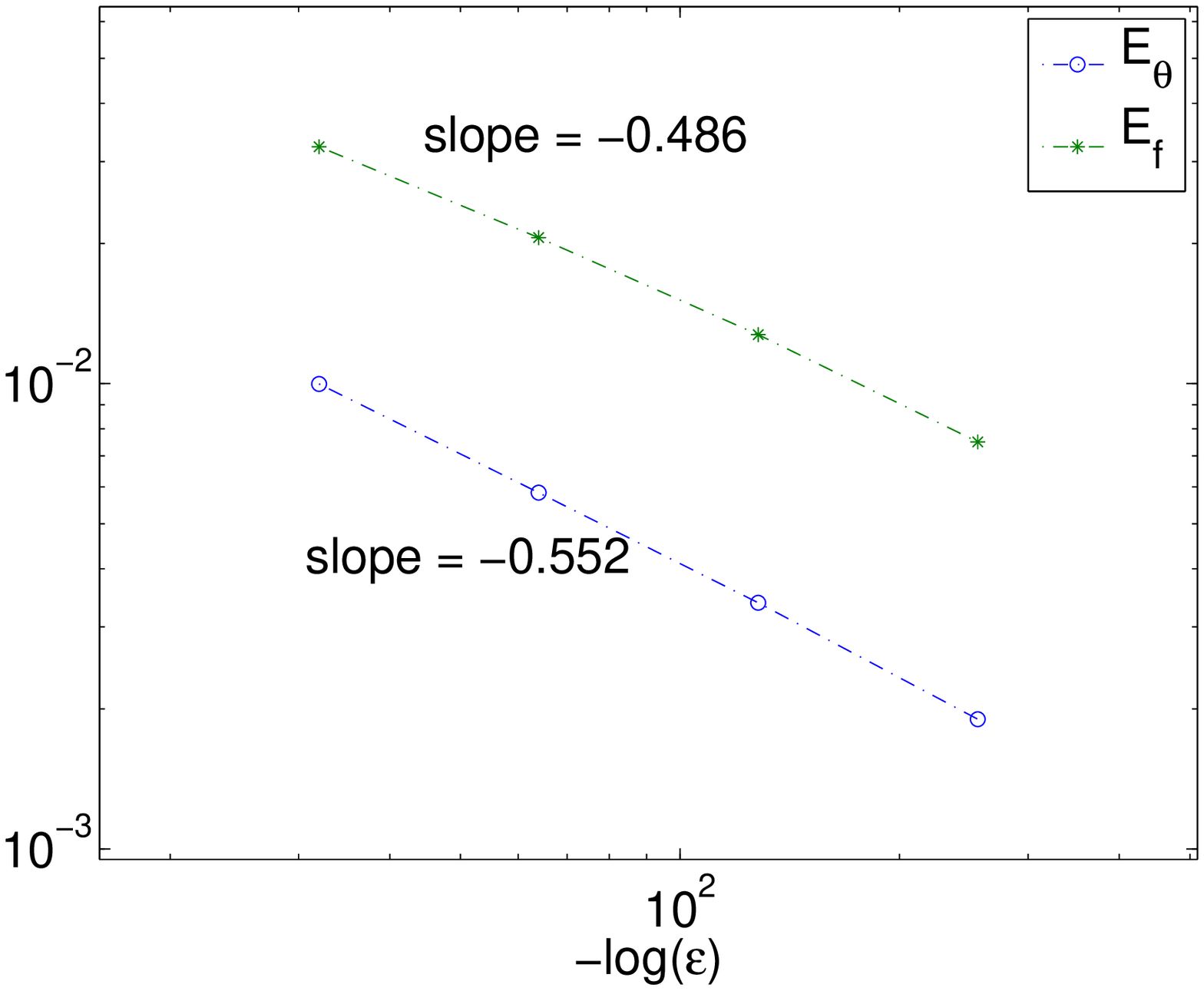}
\includegraphics[height = 1.8in]{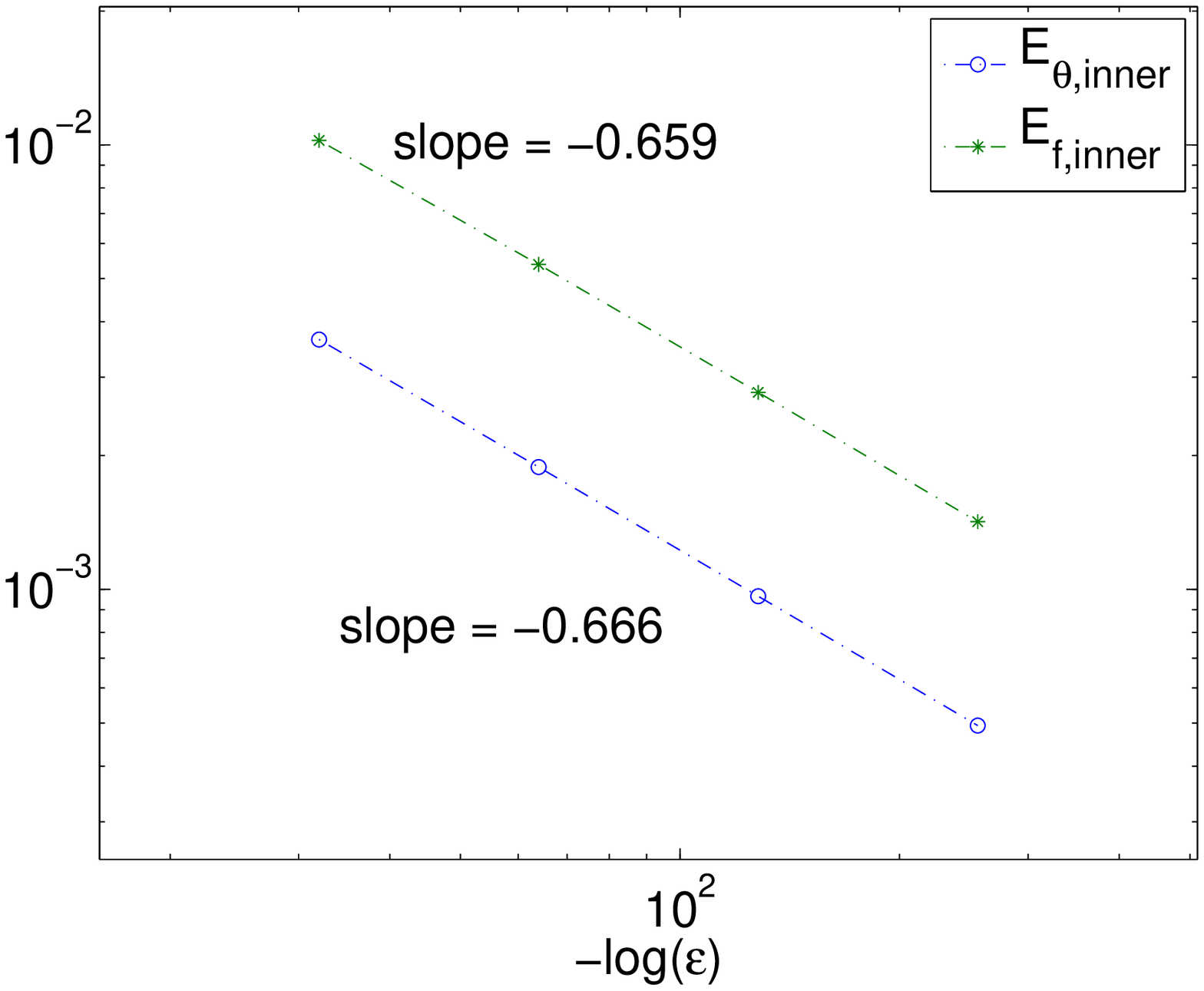}
\caption{Pure heat system, test 6: initially the boundary condition for the heat equation is not compatible. This examples includes all three types of layers.}\label{fig:test6}
\end{figure}

\subsection{Coupled system}
In this subsection we collect all numerical results for coupled systems. The setting is the same as presented in section~\ref{sec:coupled}, in which we set $\sigma = \epsilon$ in the left part of the domain but $1$ on the right, so that the kinetic and diffusive scaling coexist in the system. 

We apply the algorithm developed in Section~\ref{sec:coupled} to approximate the coupled system and compare the result with the reference solution by resolving the full kinetic equation~\eqref{eq:kinetic} on the whole domain. 
For computing the reference solution, we use $32$ quadrature points
for the $\mu$ variable, $\Delta x = 5\times 10^{-3}$, and $\Delta t =
\min\{\epsilon^2,\omega\epsilon\Delta x\}$ where the restrictions on
$\Delta t$ come from the scaling and the CFL condition. We set the CFL
number $\omega$ to be $0.5$. Recall the coupled system we compute:
\begin{equation*}
\begin{aligned}
    \Eps \, \del_t  f + \mu \del_x &f + \CalL f = 0 \,,\\
   f|_{x=0} &= \CalR(f_{0, +})  \,, \qquad \mu < 0 \,, \\
    f|_{x=-1} &= \phi_{-1}(t, \mu) \,, \hspace{0.8cm} \mu > 0 \,, \\
    f|_{t=0} &= \phi_0 \,, \hspace{1.8cm} x \in (-1, 0) \,,
\end{aligned}
\end{equation*}
on the left side of the domain (kinetic regime), and this is joint with the heat equation on the right:
\begin{equation*}
\begin{aligned}
    \del_t \theta - \vint{\mu \CalL^{-1} (\mu)} \, &\del_{xx} \theta = 0 \,, \\
    \theta|_{x=0} = \theta_m \,, \qquad &\theta|_{x=1} = \theta_1 \,,  \\
    \theta|_{t=0}(x) = \theta_i(x) = \vint{\phi_0 &}(x) \,, \qquad x \in (0,1) \,,
\end{aligned}
\end{equation*}
where $\theta_m, \theta_1$ are defined as the end state of the associated half-space problem~\eqref{eq:f-b-0-m} and~\eqref{eq:f-b-0-R} respectively. For computing this coupled system, we set $\Delta x = 5\times 10^{-3}$ and $32$ quadrature points for the $\mu$ variable in the kinetic region. Since the implicit method is used for the heat equation which relaxes the parabolic scale, the time step $\Delta t$ only needs to satisfies the CFL condition. We set the CFL number $\omega = 0.5$. We use $E_\theta$ to measure the error. For each test, we present five figures, they show the convergence of $E_\theta$ using $\epsilon^{-1} = 32,64,128,256$, the profile of the solution, together with the bottom three showing the zoom-in at $x=-1,0,1$ respectively. For all the profile figures, solid lines are given by the solution of coupling approximation and the dotted line are refined solution to the kinetic equation. Red, green blue and black lines are for $\epsilon = 32,64,128,256$ respectively.

\noindent \textbf{Test 1: Initial layer.}
In the first test, we investigate an example with only initial layer. 
Note that here the data for the heat equation in the kinetic region is not well-prepared (in the sense of \cite{GJL03}*{Equation (5.2)}), and hence an initial layer presents. 
\begin{equation}
\begin{cases}
   \text{left boundary:}& f(t, x,\mu) \big|_{x=-1} = 0,\quad \mu>0 \,,
\\[2pt]
   \text{right boundary:}&f(t, x,\mu) \big|_{x=1} = 0,\quad \mu<0 \,,
\\[2pt]
\text{initial:} & f(t, x, \mu) \big|_{t=0} = \phi_i(x,\mu) = |\mu|\sin{\pi x} \,, \quad x \in [-1, 1] \,.
\end{cases}.
\end{equation}
The derived data for the heat equation for $x \in [0, 1]$ are
\begin{equation} \nn
\begin{cases}
   \text{left boundary:} \qquad \theta(t,x) \big|_{x=0} = \theta_m \,, 
\\[2pt]
   \text{right boundary:} \quad \, \theta(t,x) \big|_{x=1} =  \theta_1 = 0  \,, 
\\[2pt]
   \text{initial:}\hspace{2cm} \theta(t,x) \big|_{t=0} = \theta_i(x) = \frac{1}{2} \sin(\pi x)  \,.
\end{cases}
\end{equation}
We calculate the solution up to time $T = 0.1$.

\noindent \textbf{Test 2: Boundary layer.}
In the second test, the initial data for the kinetic region is well-prepared in the sense of \cite{GJL03}. We allow a boundary layer generated at $x=1$ in the fluid region. The data for the full kinetic equation are
\begin{equation} \nn
\begin{cases}
   \text{left boundary:} & f(t, x, \mu) \big|_{x=-1} = |\mu|t+1,\quad \mu>0 \,,
\\[2pt]
   \text{right boundary:} & f(t, x, \mu) \big|_{x=1} = |\mu|t+0.5,\quad \mu<0 \,, 
\\[2pt]
\text{initial:} & f(t, x, \mu) \big|_{t=0} = 0.25\cos{(\pi x)}+0.75 \,, \quad x \in [0, 1] \,.
\end{cases}.
\end{equation}
The derived data for the heat equation for $x \in [0, 1]$ are
\begin{equation} \nn
\begin{cases}
   \text{left boundary:} \qquad \theta(t,x) \big|_{x=0} = \theta_m \,, 
\\[2pt]
   \text{right boundary:} \quad \, \theta(t,x) \big|_{x=1} =  \eta \, t + 0.5  \,, 
\\[2pt]
   \text{initial:}\hspace{2cm} \theta(t,x) \big|_{t=0} = 0.25\cos{(\pi x)}+0.75 \,.
\end{cases}
\end{equation}
The solution is calculated up to time $T = 0.5$.

\noindent \textbf{Test 3: incompatible initial data, all layers.}
The third example considers the most general one where all layers are present and the initial data for the kinetic region is not well-prepared. The data for the full kinetic equation are
\begin{equation}
\begin{cases}
   \text{left boundary:}& f(t, x, \mu) \big|_{x=-1} = |\mu|(t+1),\quad \mu>0 \,,
\\[2pt]
   \text{right boundary:}& f(t, x, \mu) \big|_{x=1} = |\mu|(t+1),\quad \mu<0 \,, 
\\[2pt]
   \text{initial:} & f(t, x, \mu) \big|_{t=0} = |\mu| \,, \qquad x \in [-1, 1] \,.
\end{cases}.
\end{equation}
The derived data for the heat equation for $x \in [0, 1]$ are
\begin{equation} \nn
\begin{cases}
   \text{left boundary:} \qquad \theta(t,x) \big|_{x=0} = \theta_m \,, 
\\
   \text{right boundary:} \quad \, \theta(t,x) \big|_{x=1} =  \eta \, (t + 1)  \,, 
\\
   \text{initial:}\hspace{2cm} \theta(t,x) \big|_{t=0} = \frac{1}{2} \,.
\end{cases}
\end{equation}
The computation is stopped at $T = 0.5$. For this example we also plot the profile of $\theta_{k}$ and $\theta_c$ ($\theta$ computed using kinetic model and the coupled approximation) at several time  together with the evolution of the $L_2$-norm of the differences.

The numerical results are presented in
Figures~\ref{fig:couple_test1}--\ref{fig:couple_test3}. We observe
nice agreement with the fully resolved solution to the kinetic
equation and the approximation by the coupled equation. The right
panel of Figure~\ref{fig:couple_test3} shows the evolution of the
error, note that the error decreases initially due to the decay of the
error from the initial layer, the error then accumulates as the
simulation proceeds, though the growing rate for the error is fairly
slow (empirically linear growth is observed).

\begin{figure}[h]
\includegraphics[height = 2in]{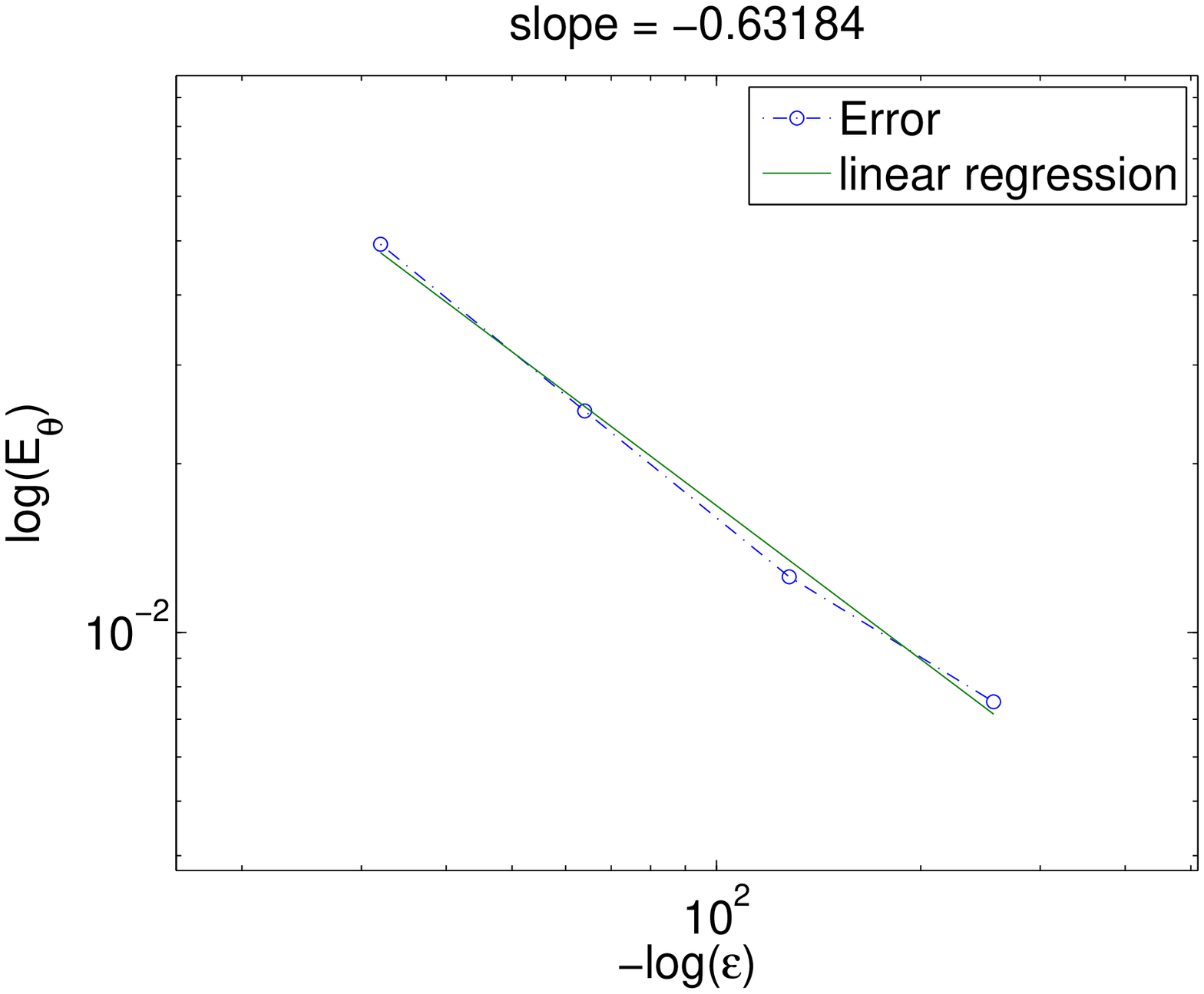}
\includegraphics[height = 2in]{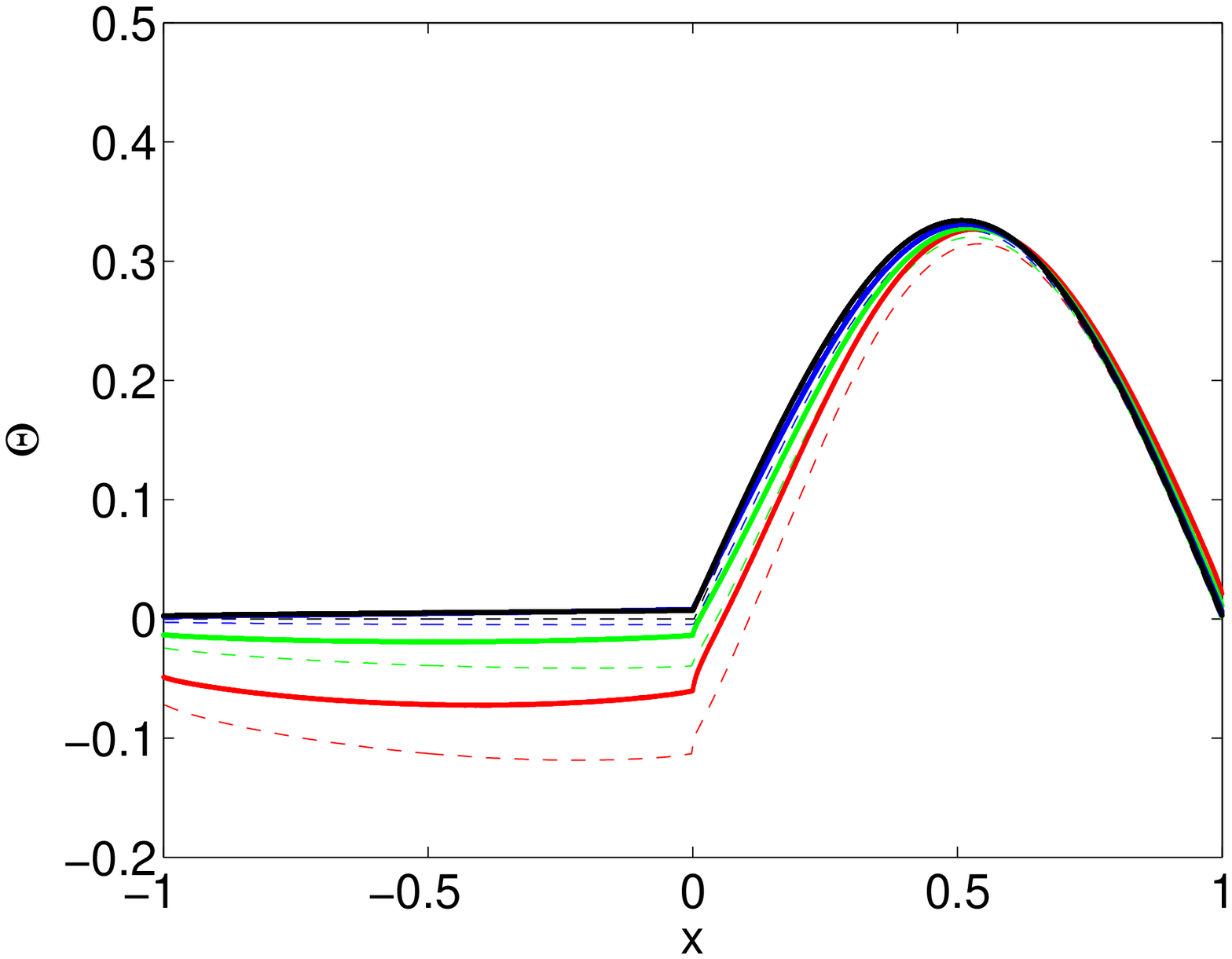}\\
\includegraphics[height = 2in,width = 1.8in]{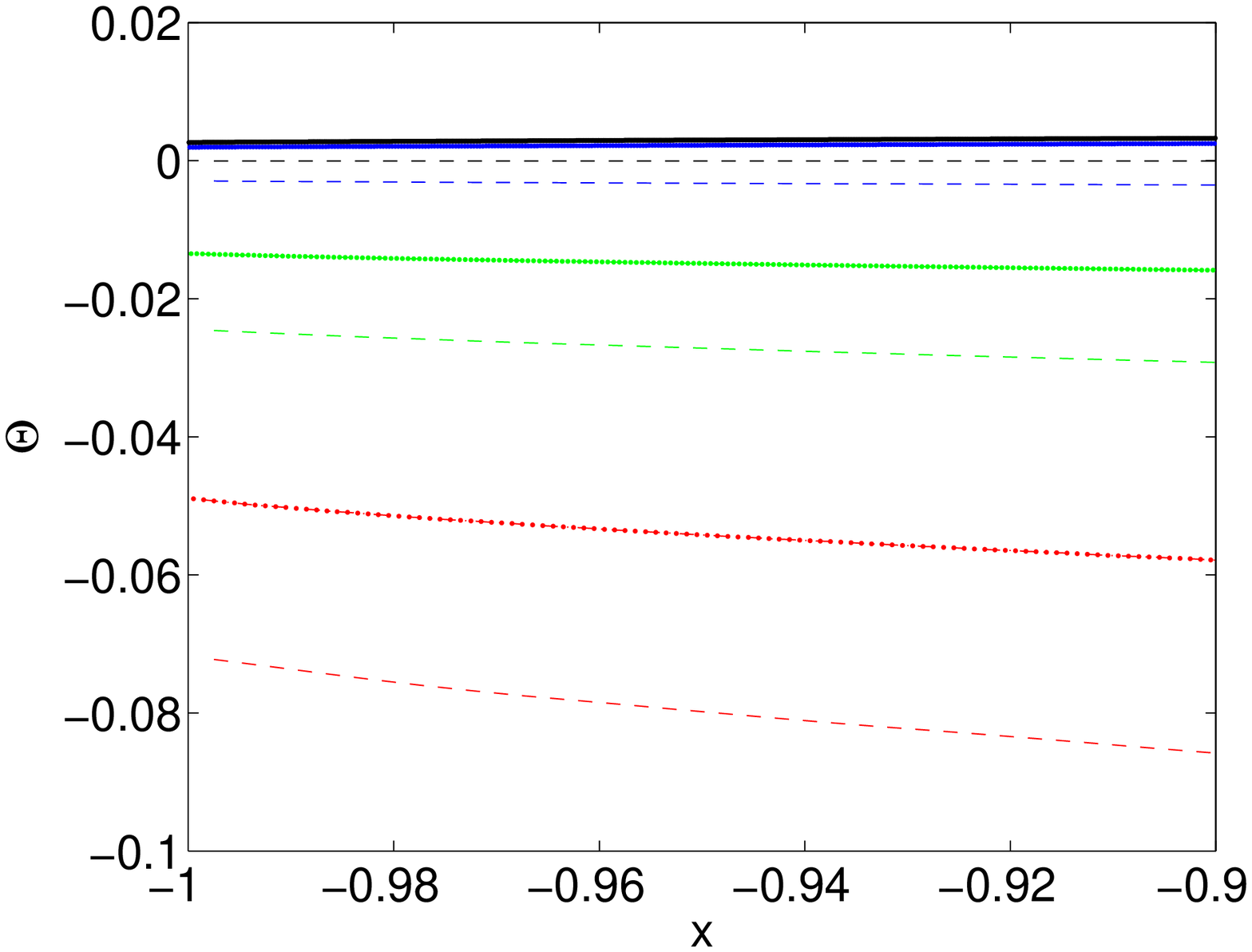}
\includegraphics[height = 2in,width = 1.8in]{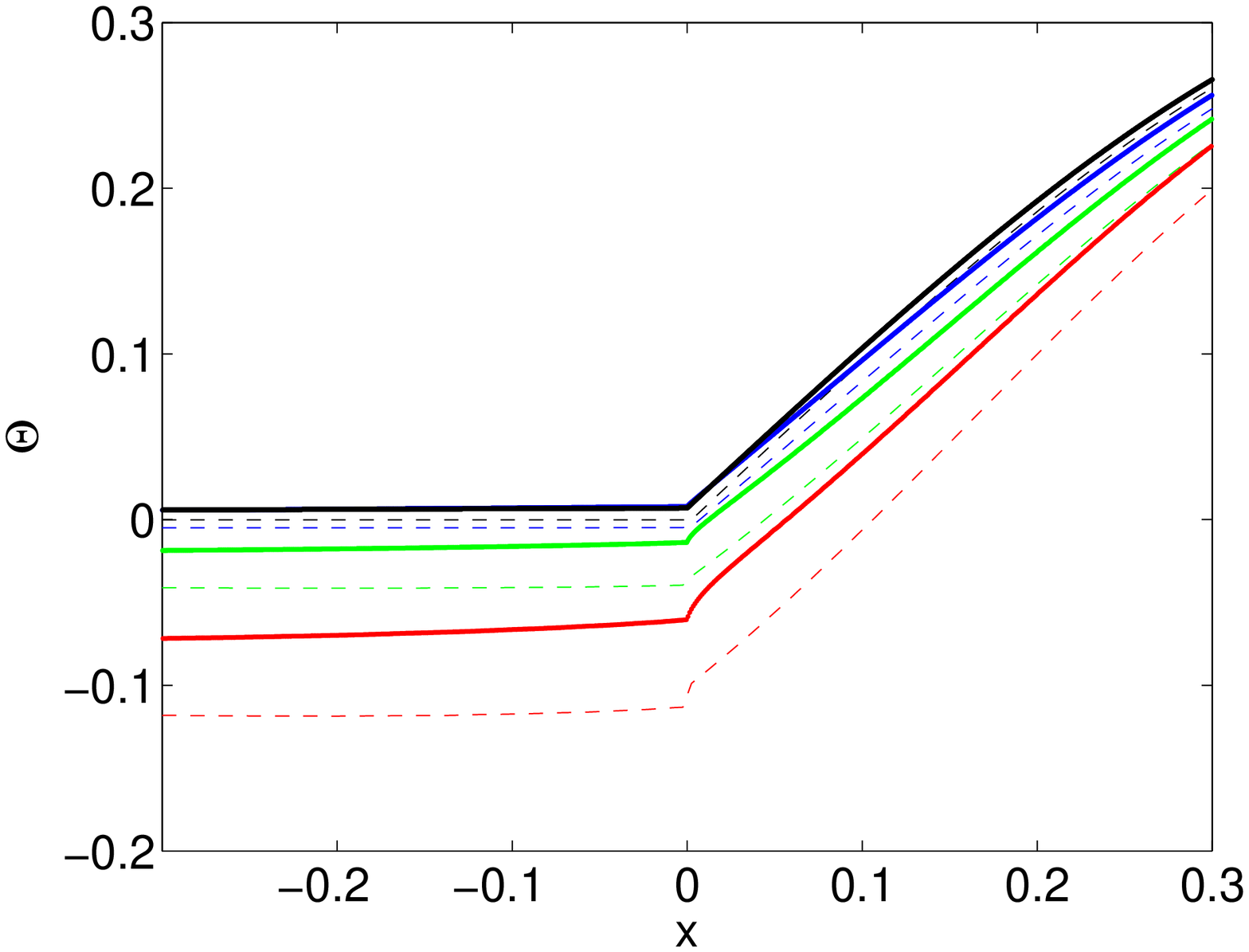}
\includegraphics[height = 2in,width = 1.8in]{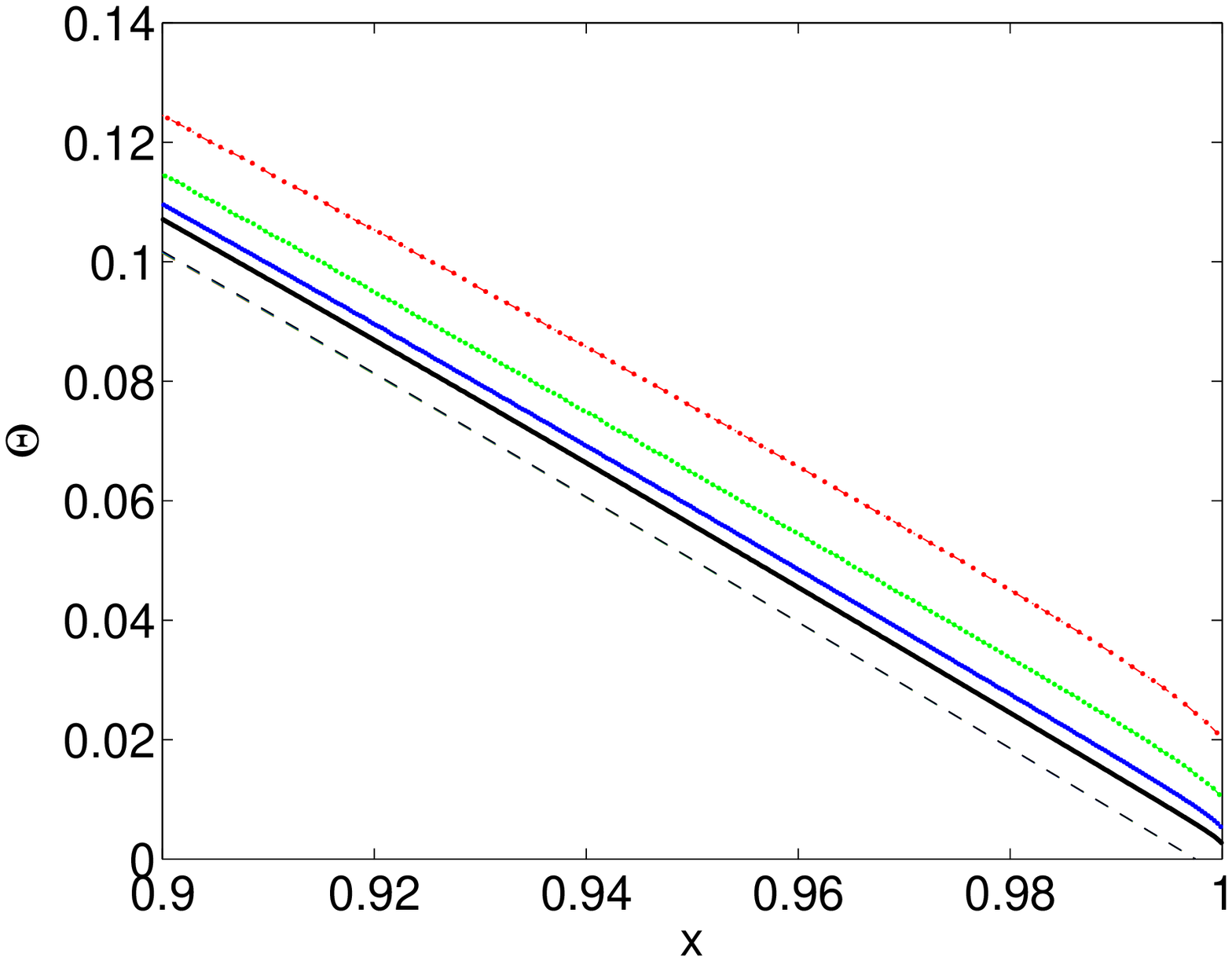}
\caption{Coupled system, test 1: $T = 0.1$. Initial layer presents.}
\label{fig:couple_test1}
\end{figure}

\begin{figure}[h]
\includegraphics[height = 2in]{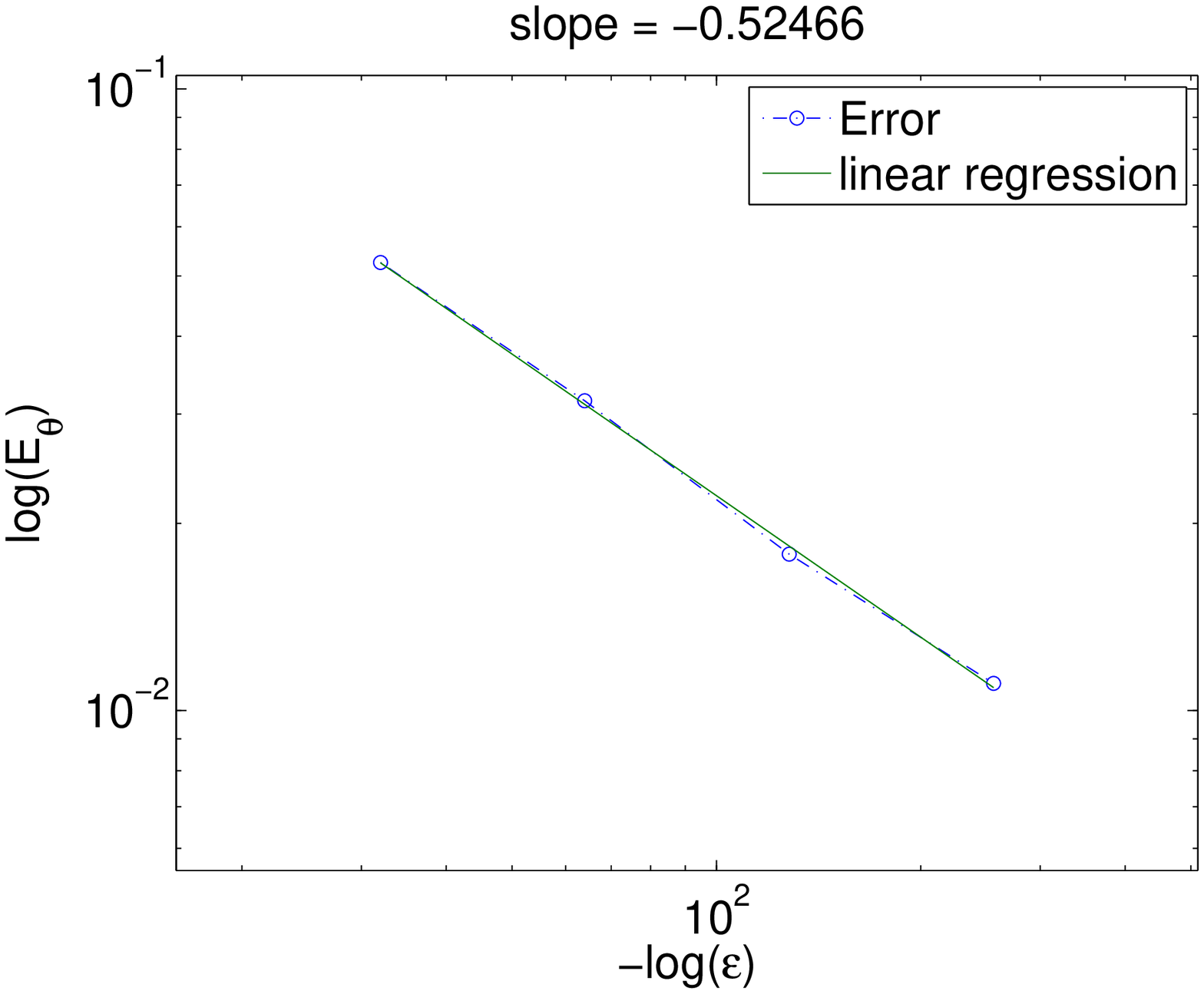}
\includegraphics[height = 2in]{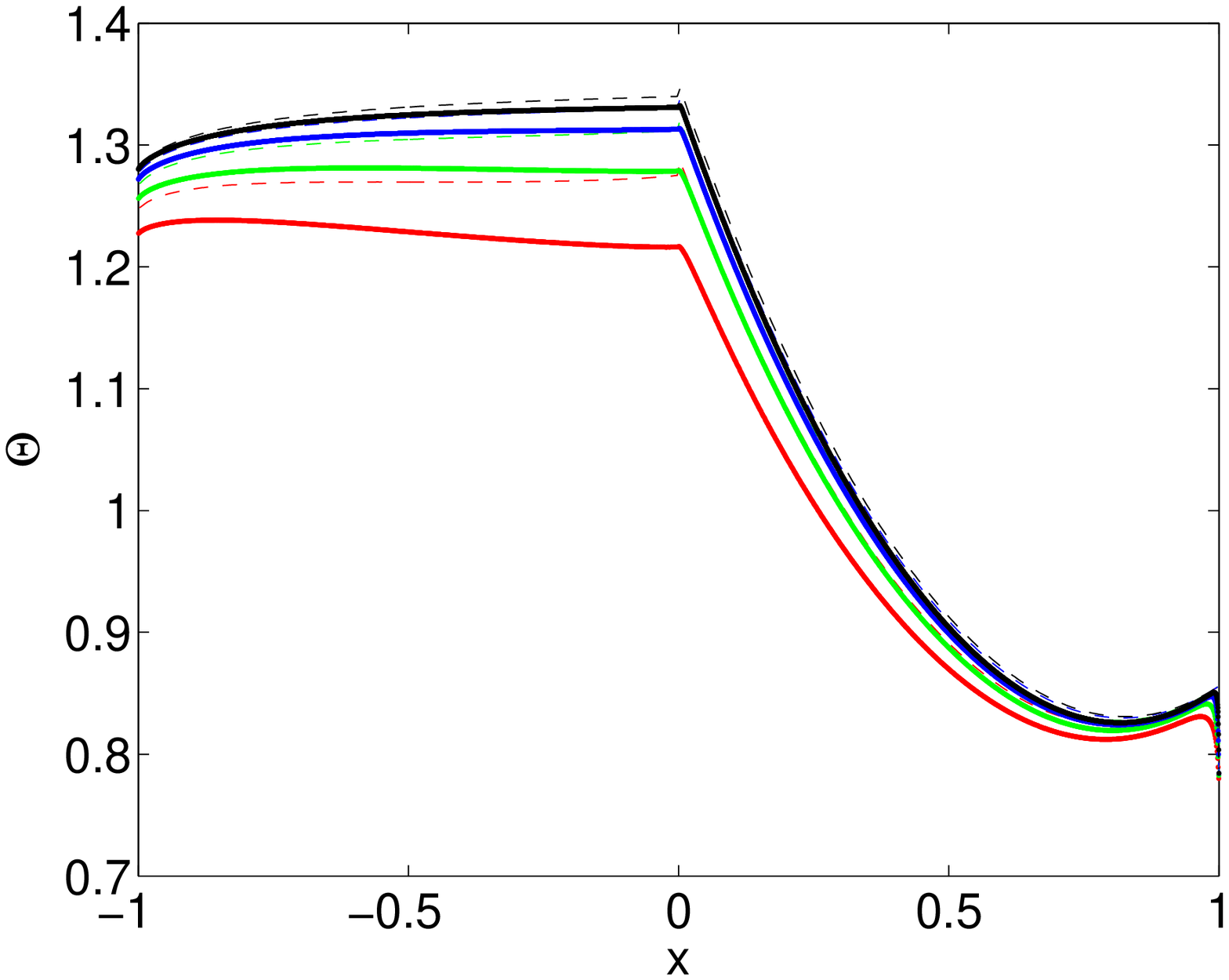}\\
\includegraphics[height = 2in,width = 1.8in]{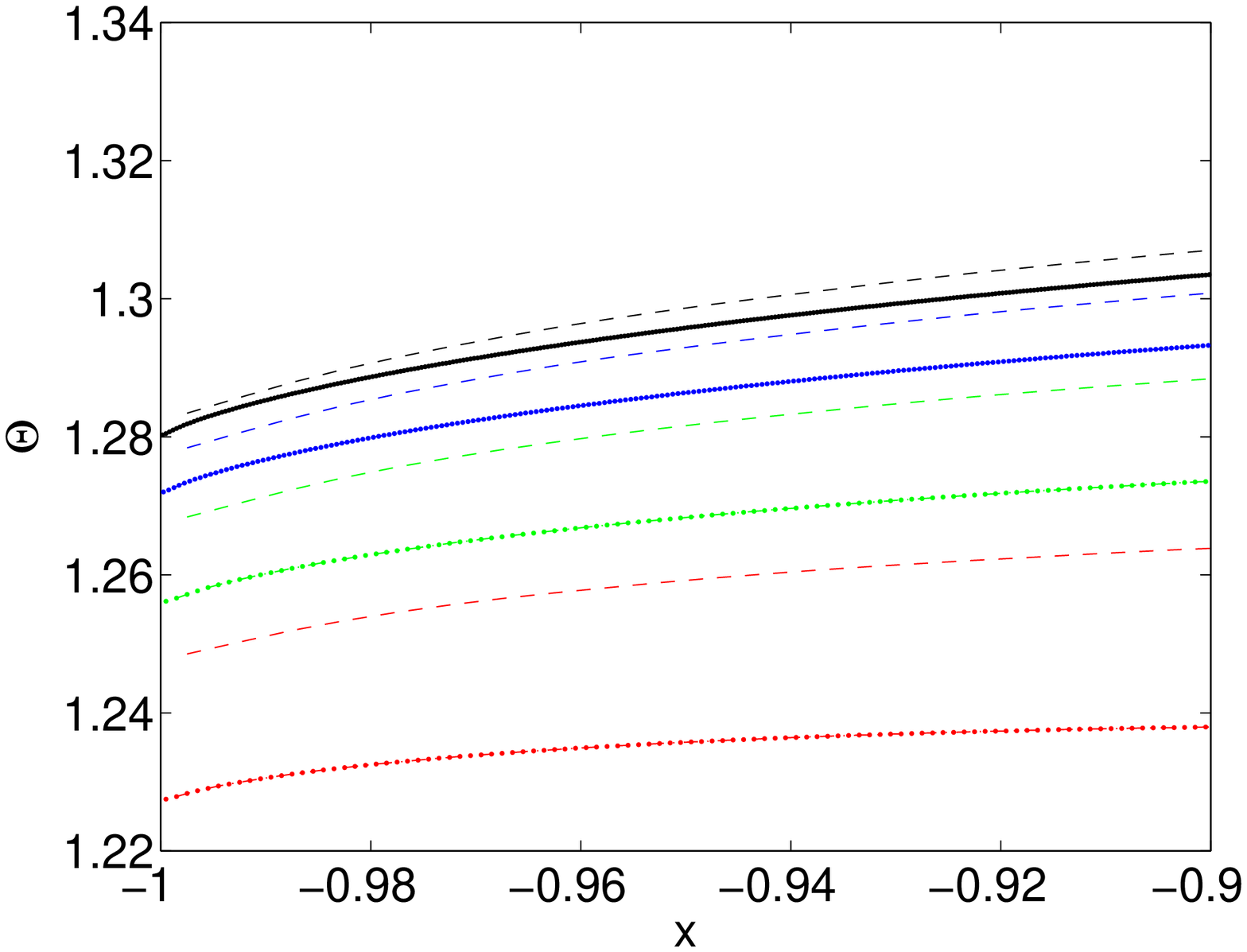}
\includegraphics[height = 2in,width = 1.8in]{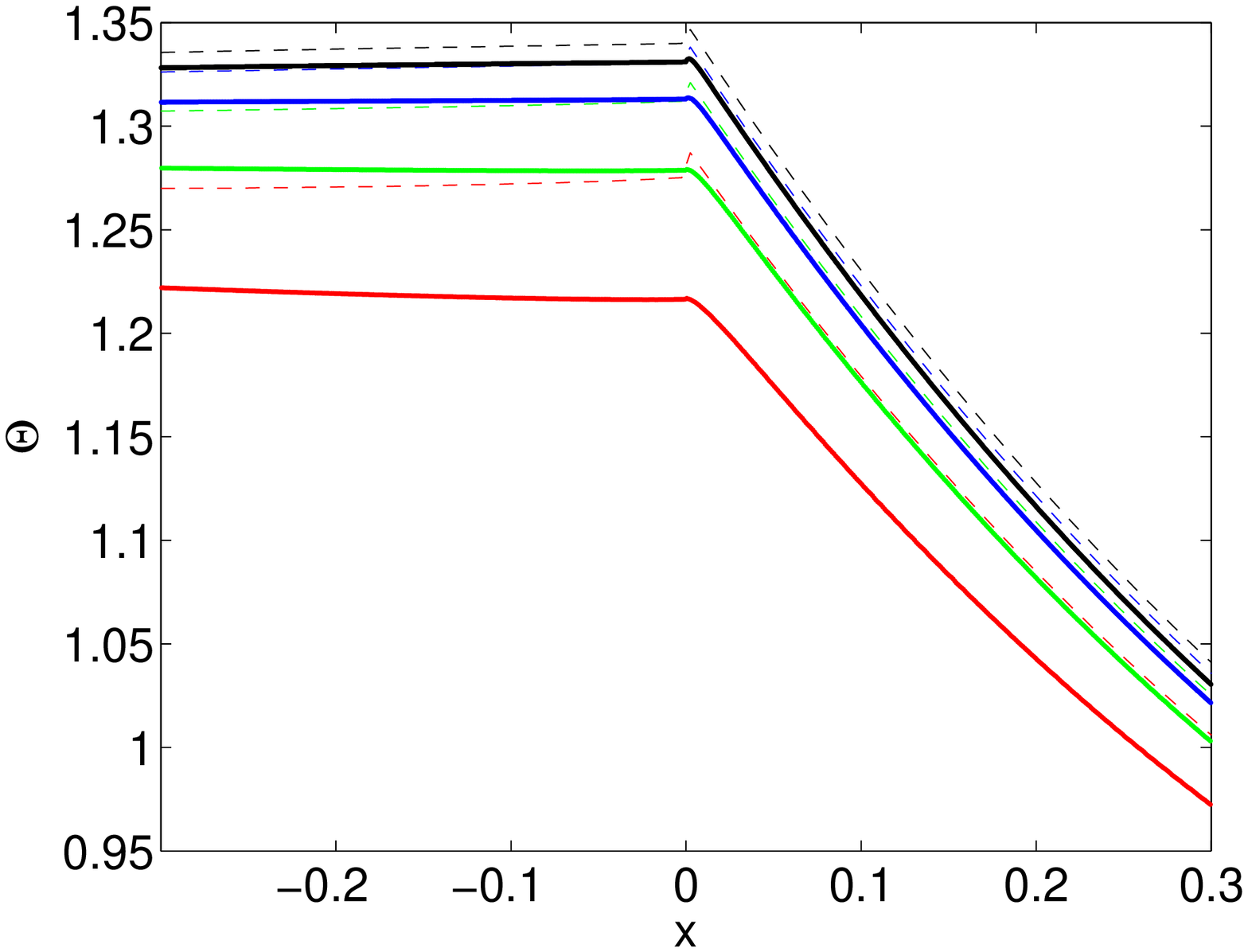}
\includegraphics[height = 2in,width = 1.8in]{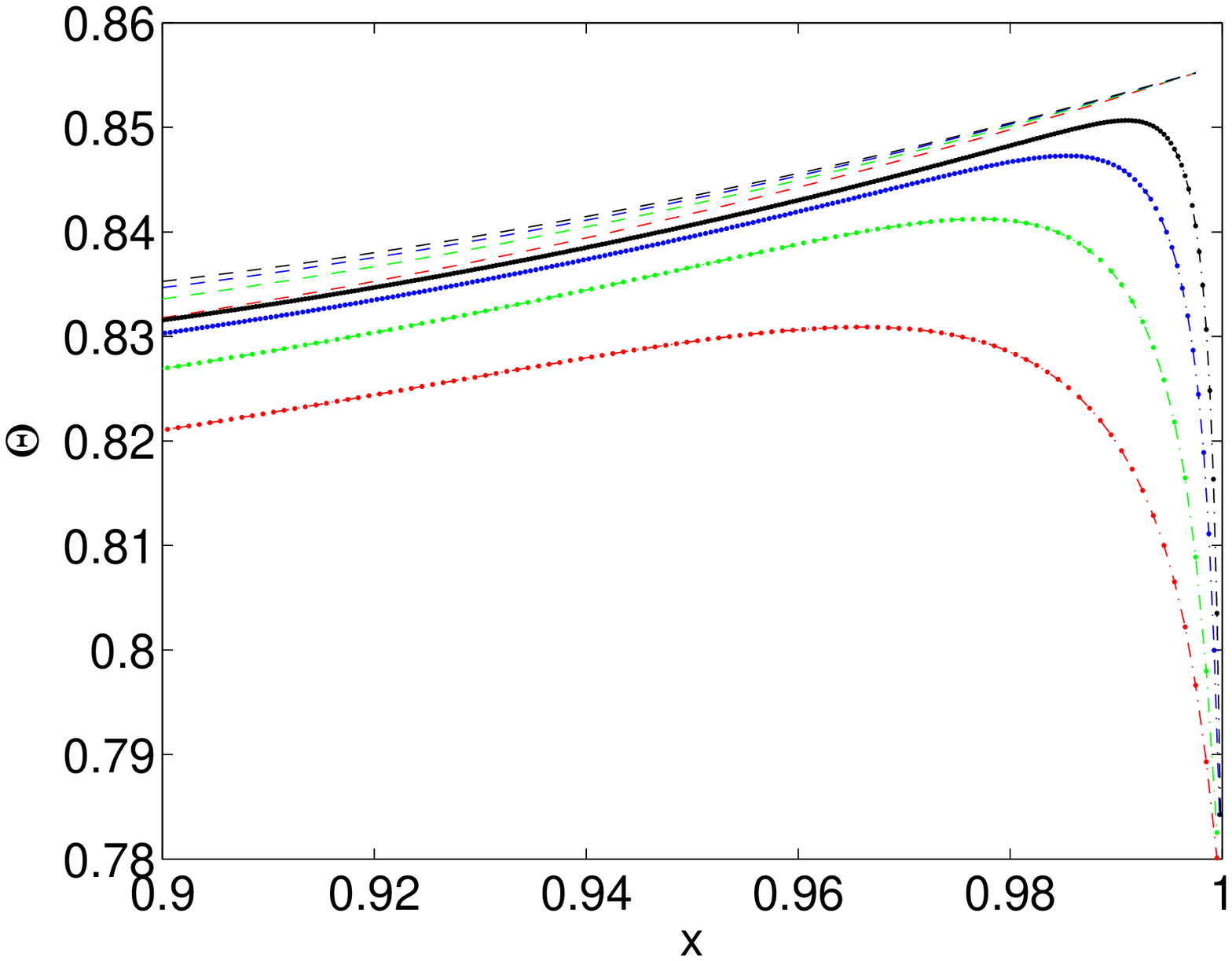}
\caption{Coupled system, test 2: $T = 0.5$. Boundary layer presents.}
\end{figure}

\begin{figure}[h]
\includegraphics[height = 2in]{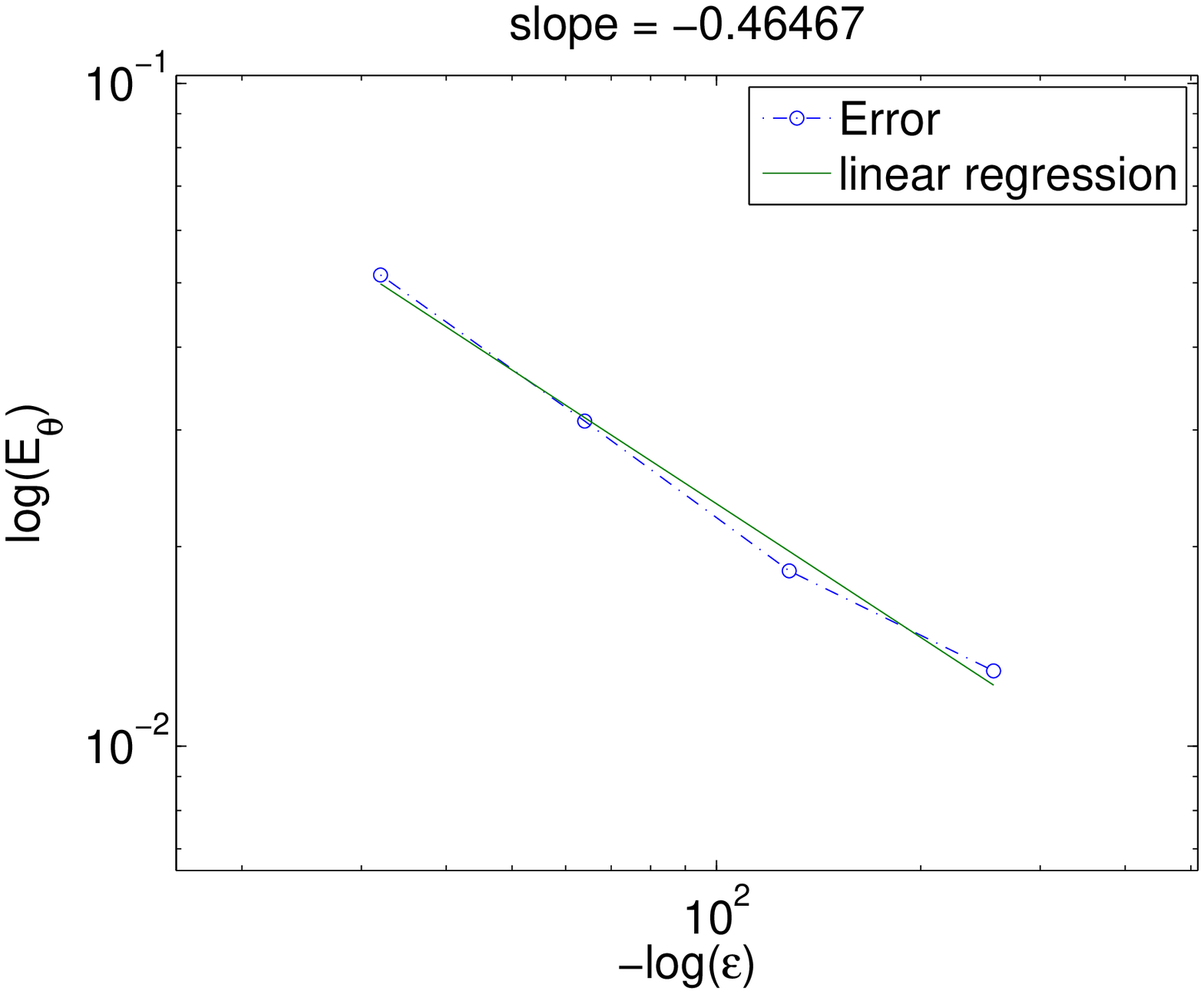}
\includegraphics[height = 2in]{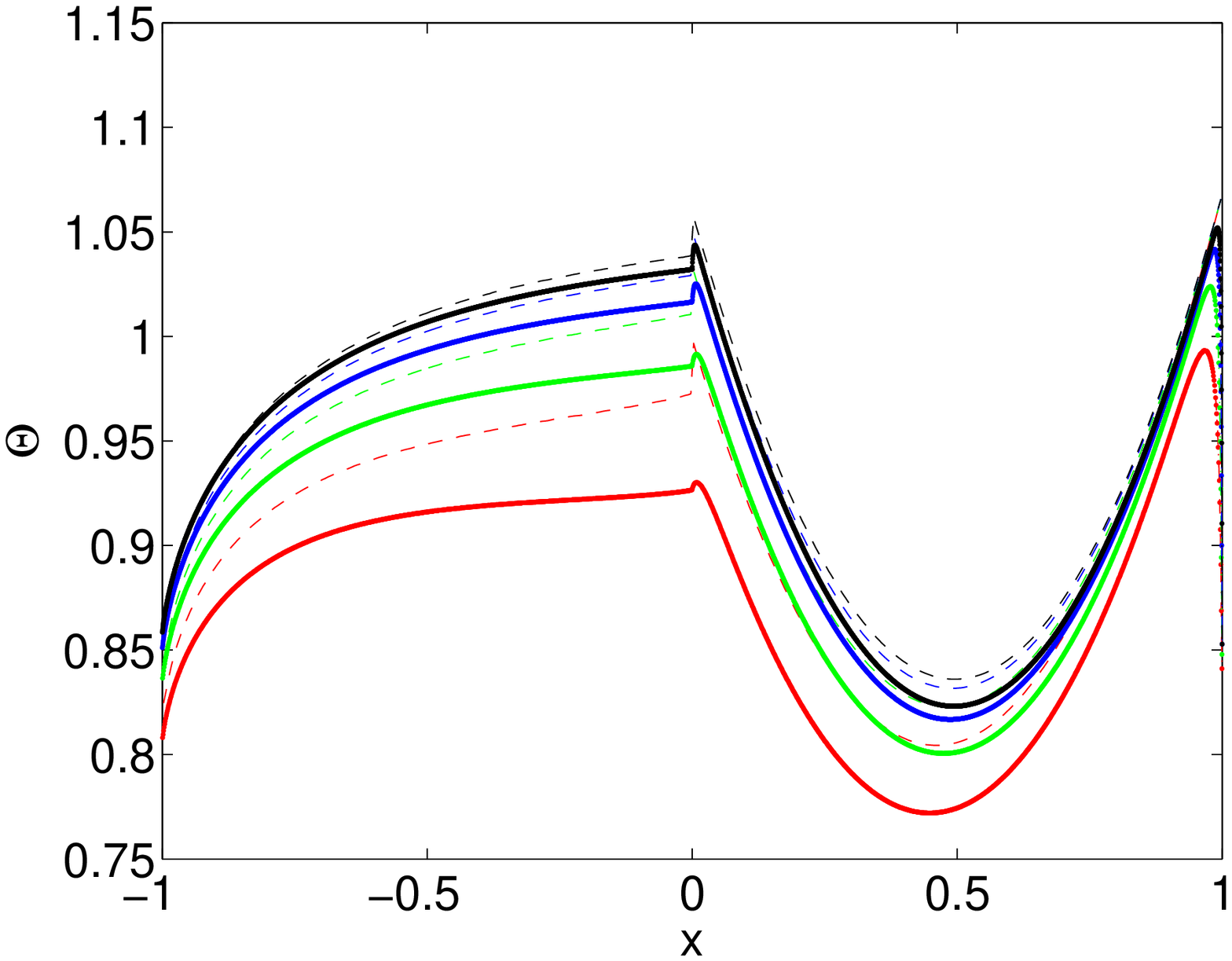}\\
\includegraphics[height = 2in,width = 1.8in]{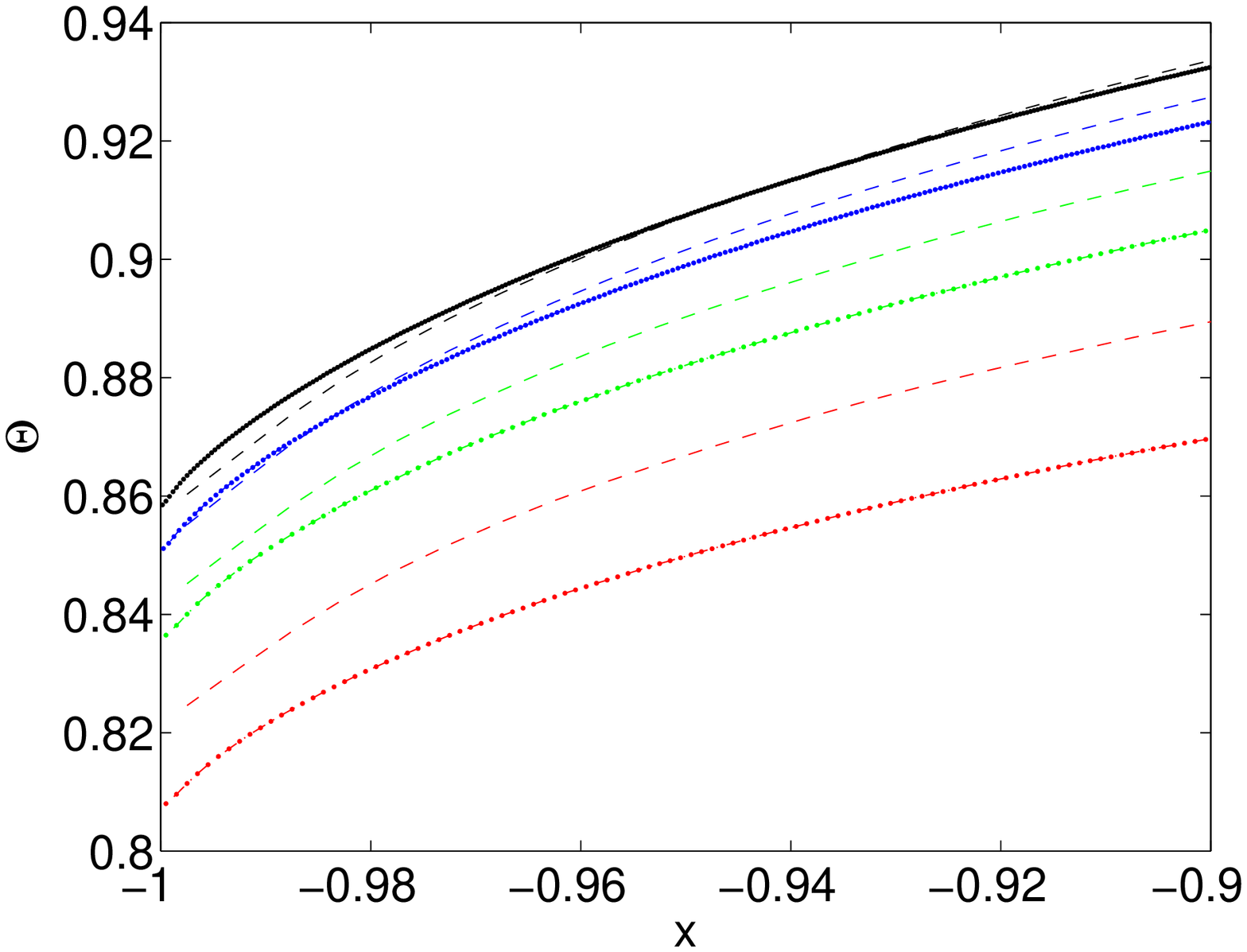}
\includegraphics[height = 2in,width = 1.8in]{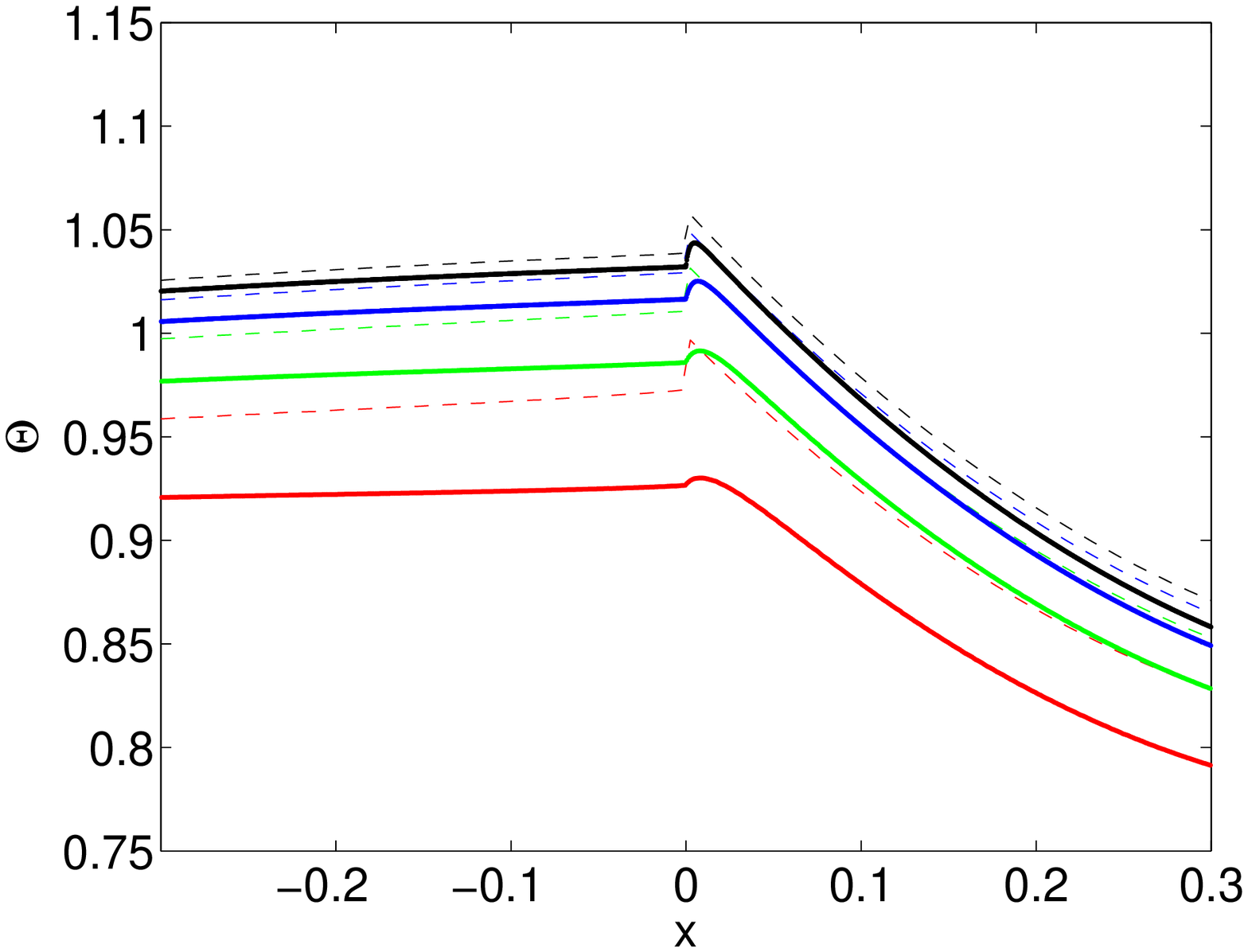}
\includegraphics[height = 2in,width = 1.8in]{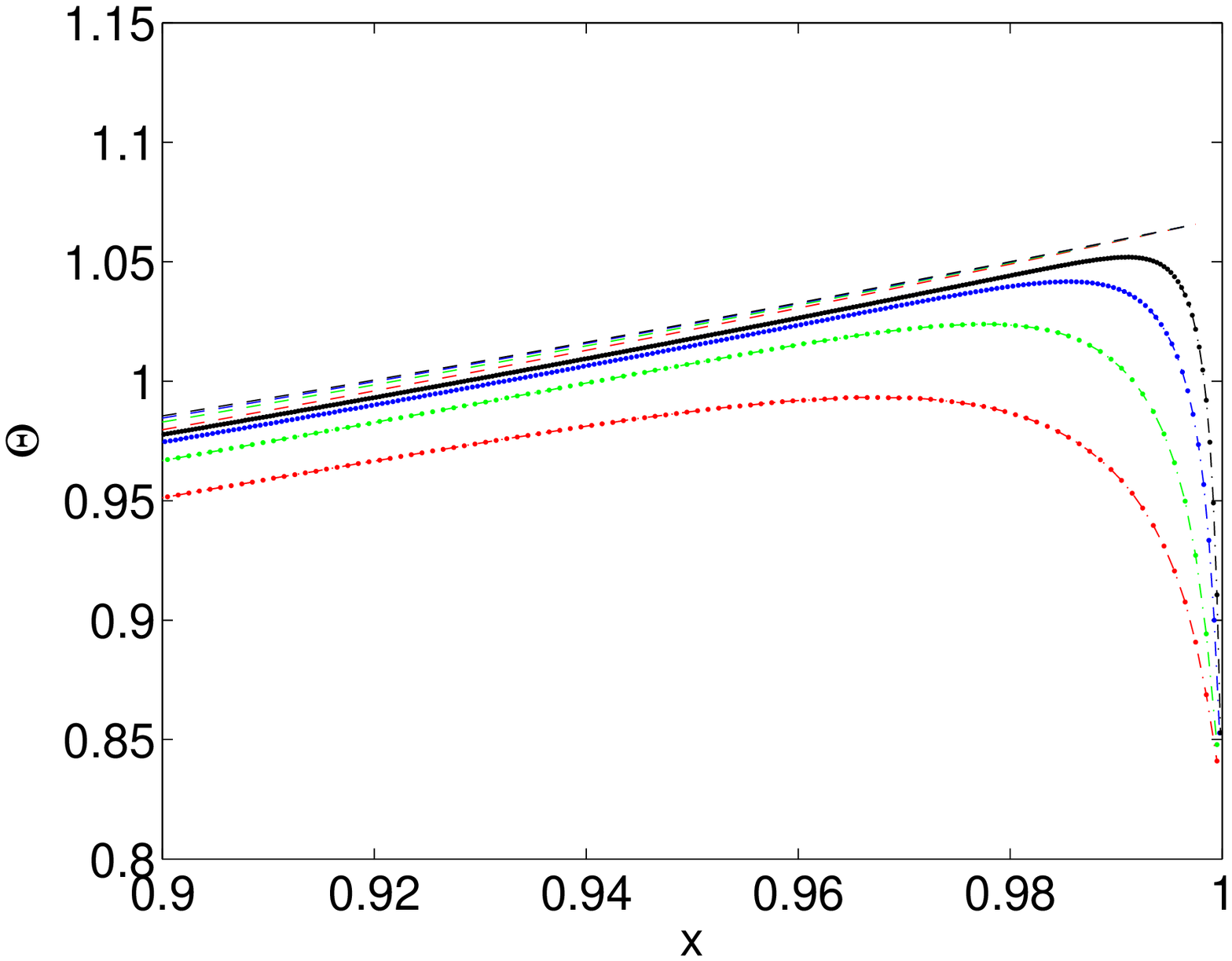}
\caption{Coupled system, test 3: general case. $T = 0.5$. All layer are included.}
\end{figure}

\begin{figure}[h]
\includegraphics[height = 2in,width = 1.8in]{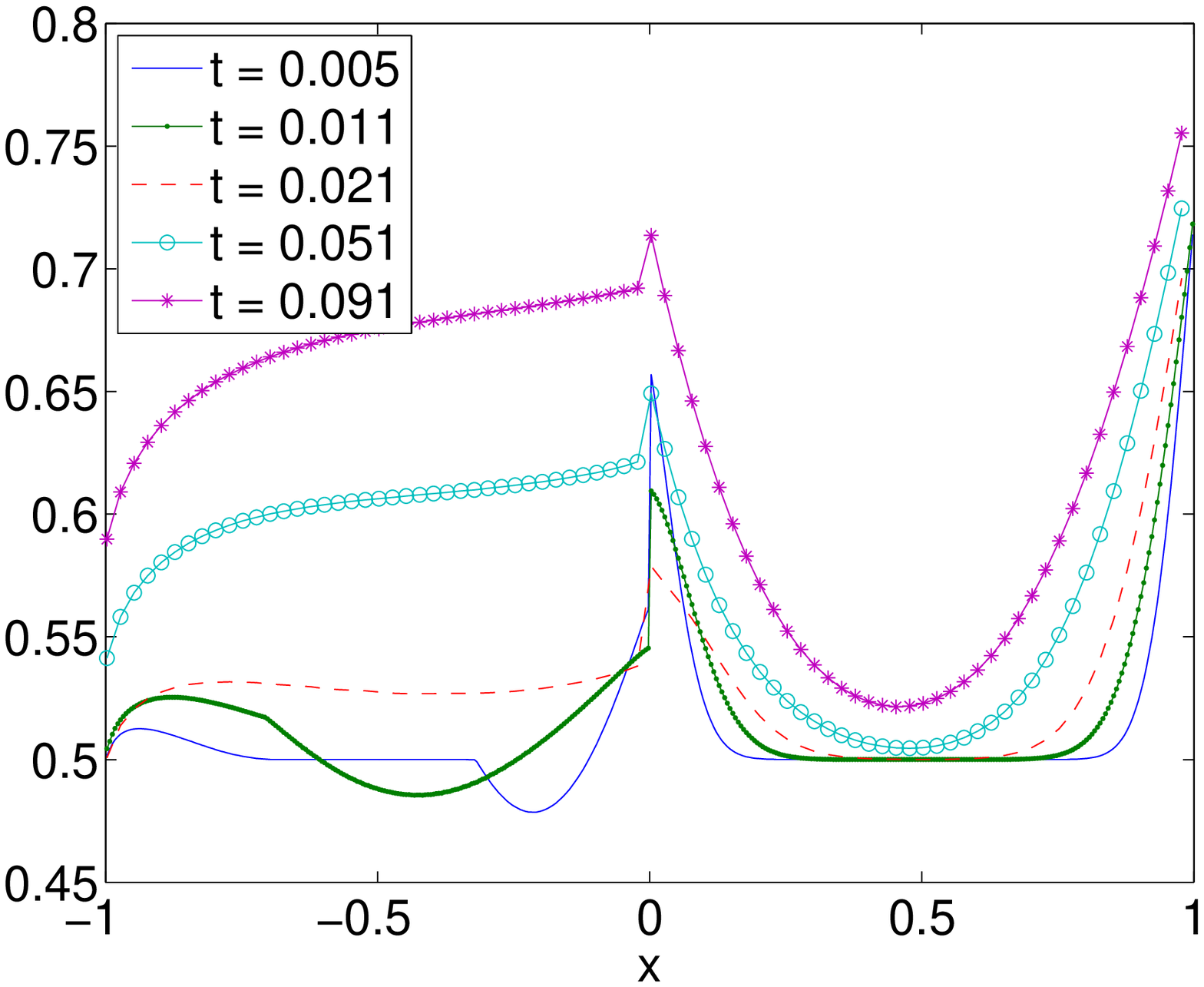}
\includegraphics[height = 2in,width = 1.8in]{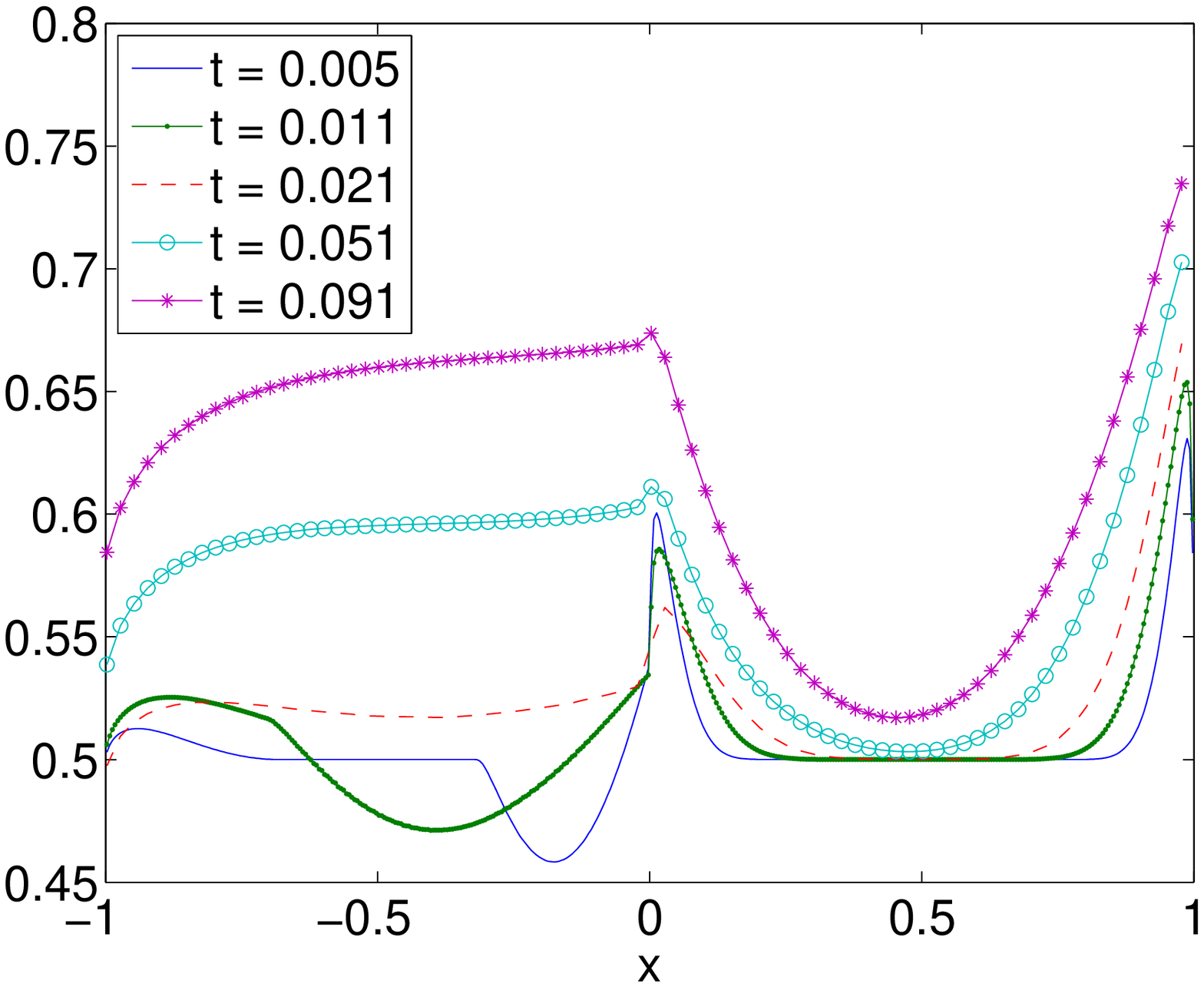}
\includegraphics[height = 2in,width = 1.8in]{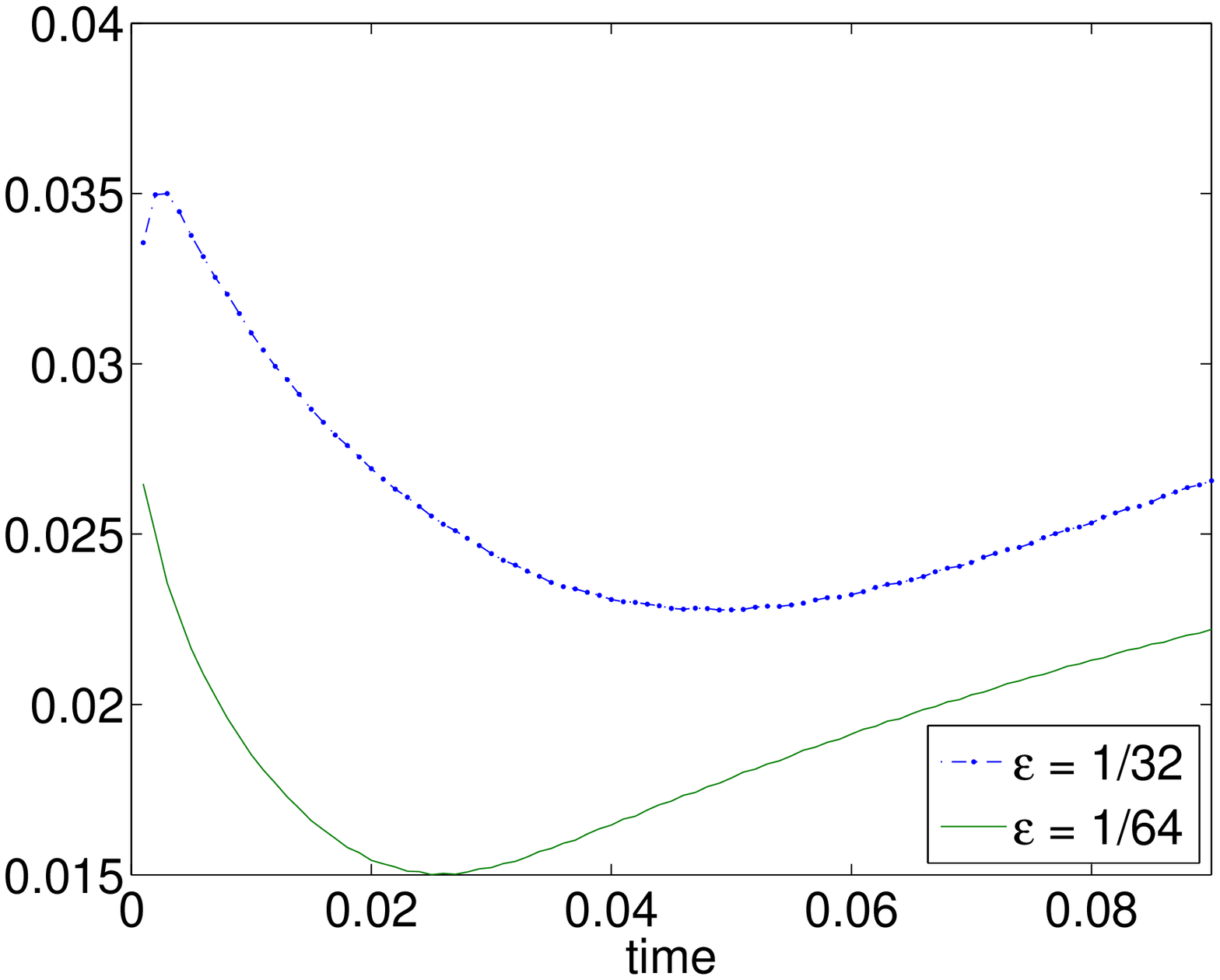}
\caption{Coupled system, test 3: general case. The two figures on the left show five snapshots of the the profile computed by coupled method and the standard kinetic method. The figure on the right demonstrates the evolution of $L_2$ norm of the error for $\epsilon = 1/32,1/64$.}
\label{fig:couple_test3}
\end{figure}

\vspace{-8pt}
\subsection{Stability test} \label{section:stability-test} This
subsection is to provide a numerical test of the error induced in the
kinetic part of the coupled system by the approximation to the
back-flow at the interface. Hence we study the kinetic equation with
zero initial and ``reflective'' boundary with 
perturbation that impacts the solution on a time scale of order $\CalO(\Eps^2)$. Specifically, we compute the equation
\begin{equation}
\begin{aligned}\label{eqn:perturb}
&\Eps\partial_tf + \mu\partial_xf + \mathcal{L}f = 0\,, \\
&f(t,x,\mu)\big|_{x = -1} = 0\,, \hspace{3cm} \mu>0 \,, \\
&f(t,x,\mu)\big|_{x=0} = \CalR \bigl(f(t,0, \cdot) \vert_{\mu > 0}\bigr) + p \vpran{t/\Eps^2} \,, \qquad \mu<0 \,,
\end{aligned}
\end{equation}
where the perturbation is chosen as 
\begin{equation}
p \vpran{t/\Eps^2} = \frac{1}{1 + \sqrt{t/\epsilon^2}} \,.
\end{equation}
In Figure~\ref{fig:stable_profile} we show the profile of the solution to~\eqref{eqn:perturb} at several times. The stopping time is $T = 0.1$. Observe that the error decays (in $L^2$ norm) as time proceeds and also the impact becomes smaller for a smaller $\Eps$. This justifies neglecting the feedback to the kinetic equation due to initial layer and initial-boundary layer in our coupled scheme. It will be interesting to rigorously show the decay of the perturbation and find its decay rate.

\begin{figure}[h]
\includegraphics[height = 2in]{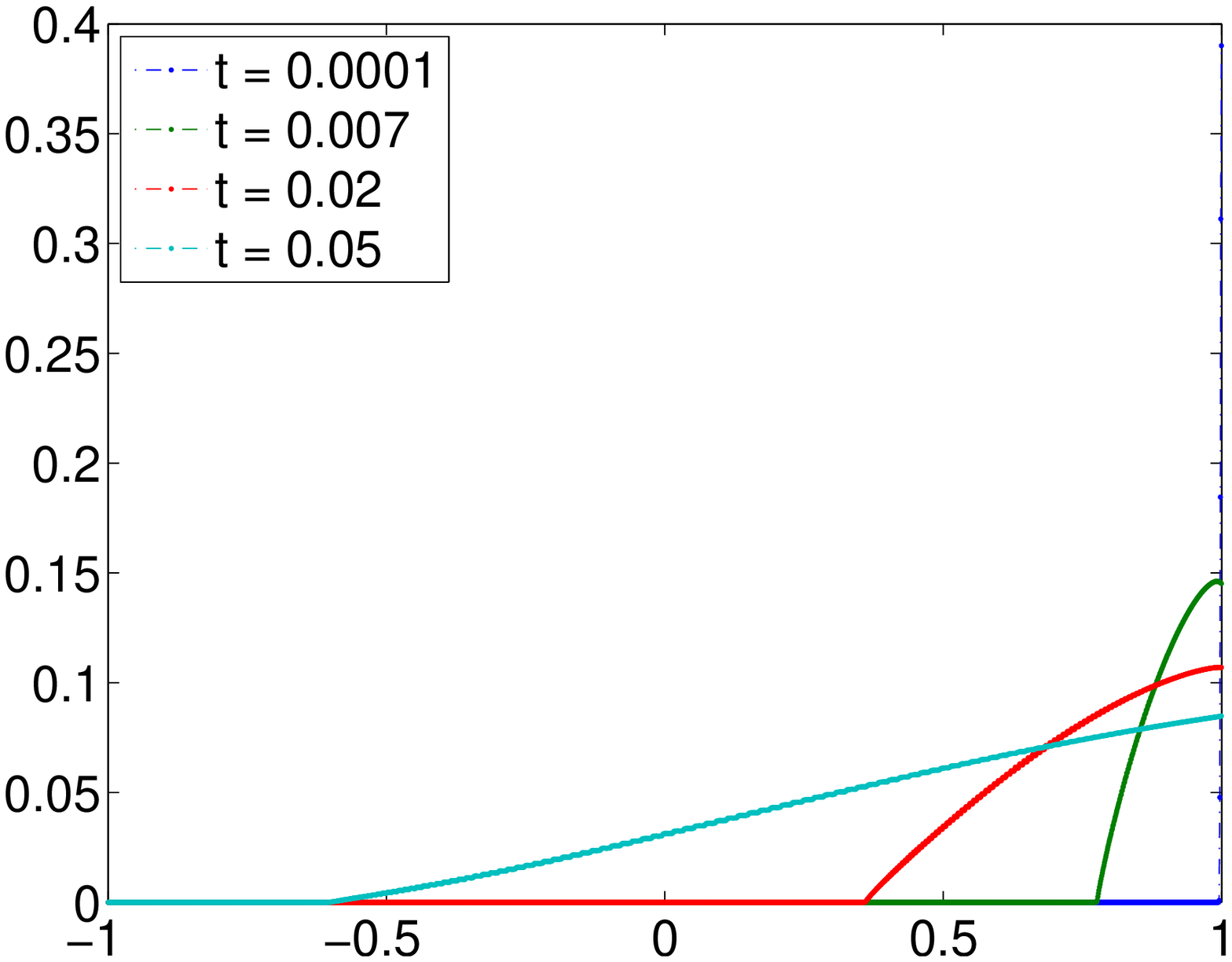}
\includegraphics[height = 2in]{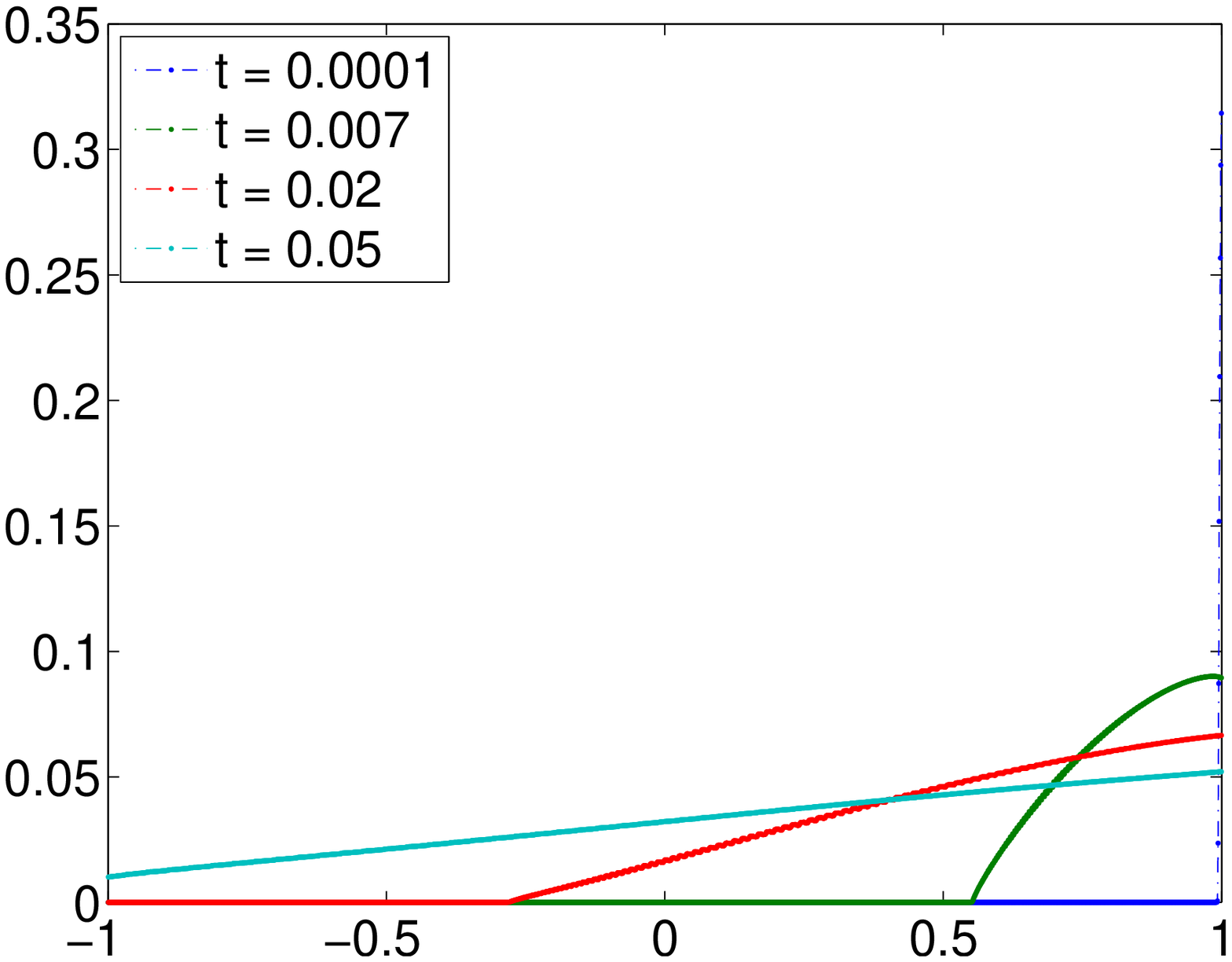}
\caption{Stability test: as time evolves, the perturbation propagates to the inner domain and gradually diminish. The figure on the left is given by setting $\epsilon = 1/32$ and the one on the right is for $\epsilon = 1/64$.}\label{fig:stable_profile}
\end{figure}

\appendix

\section{Error Estimates}

In this appendix, we give some formal estimates of the errors for the pure fluid approximations. The main assumption here is the existence and decay of the initial-boundary layer solution. Our analysis covers both cases where the diffusion equation has either compatible or incompatible data and we treat them separately. The latter case is slightly more involved since we have uniform bounds for derivatives of the heat solution if its data are compatible while this property ceases to hold for the incompatible data.

In the current work we restrict ourselves to the pure case. Error analysis for the coupling case requires studies for not only the initial-boundary layers but also the perturbation equation in the form of~\eqref{eqn:perturb}. Moreover, it is expected that careful spectral analysis needs to be done since we are considering the critical case. These analysis will be left for further investigation. 

\subsection{Compatible data}\label{est:compatible-data}

First we show some basic error estimates in $L^2$-spaces for the diffusion approximation with compatible data such that $\theta \in C^2([0, T]; C^4[a, b])$. Our analysis follows the classical idea of constructing approximation solutions (see for example \cite{BLP:79, GJL99, GJL03}). Define the approximate solution $f^A$ as
\begin{equation}
\begin{aligned} \label{def:f-A-approximation}
    f^A = & f^{\inte}(t, x, \mu) 
             + f^b_{L} (t, \tfrac{x-a}{\Eps}, \mu)
             + f^b_{R} (t, \tfrac{b-x}{\Eps}, \mu)
             + f^{I}(\tfrac{t}{\Eps^2}, x, \mu)
\\
           & + f^{\IBL}_L(\tfrac{t}{\Eps^2}, \tfrac{x-a}{\Eps}, \mu)
             + f^{\IBL}_R(\tfrac{t}{\Eps^2}, \tfrac{b-x}{\Eps}, \mu) \,,
\end{aligned}
\end{equation}
where the details for each term are explained below. 
First, we construct the interior approximation as 
\begin{align} \label{def:f-int-1}
     f^{\inte} (t, x , \mu) = \theta(t, x) - \Eps \, \CalL^{-1}(\mu) \, \del_x \theta(t, x)
                                   + \Eps^2 \vpran{\del_{xx} \theta(t, x)} 
                                      \CalL^{-1}(\mu \CalL^{-1}(\mu))  \,,
\end{align}
where $\theta$ satisfies the diffusion equation~\eqref{eq:interior-full}. Then $f^{\inte}$ satisfies that 
\begin{equation}
\begin{aligned} \label{eq:f-int}
     \Eps \, \del_t f^{\inte} + \mu \del_x f^{\inte} + \frac{1}{\Eps} \, \CalL f^{\inte} 
     &=  - \vpran{\Eps^2 \del_x^3 \theta} (\mu \CalL^{-1}(\mu \CalL^{-1}(\mu)))
\\
 & \quad \,
            - \vpran{\Eps^2 \del_{tx} \theta} \vpran{\CalL^{-1}(\mu)}
            + \vpran{\Eps^3 \del_{txx} \theta} \vpran{\CalL^{-1}(\mu \CalL^{-1}(\mu))}
\,,
\\
       f^{\inte} \big|_{x=a} 
       = \theta_a -  \Eps \, \CalL^{-1}(\mu) \, \del_x &\theta(t, a)
          + \Eps^2 \CalL^{-1}(\mu \CalL^{-1}(\mu))
           \vpran{\del_{xx} \theta(t, a)} \,,
\qquad \mu > 0,                              
\\
        f^{\inte} \big|_{x=b} 
        = \theta_b -  \Eps \, \CalL^{-1}(\mu) \, \del_x &\theta(t, b)
            + \Eps^2 \CalL^{-1}(\mu \CalL^{-1}(\mu))
             \vpran{\del_{xx} \theta(t, b)} \,,
\qquad \mu < 0,
\\
        f^{\inte} \big|_{t=0} 
        = \vint{\phi_0} 
           -  \Eps \, \CalL^{-1}(\mu)& \, \del_x \theta(0, x)
           + \Eps^2 \CalL^{-1}(\mu \CalL^{-1}(\mu))
            \vpran{\del_{xx} \theta(0, x)} \,.
\end{aligned}
\end{equation}
Second, the boundary layer approximation $f^b_L$ at $x=a$ is given by 
\begin{align*}
     f^b_L = f^b_{0, L} + \Eps^2 f^b_{1, L} \,,
\end{align*}
where $f^b_{0, L} \in L^\infty(\dy\dmu)$ satisfies the half space equation
\begin{align*}
    \mu \del_y f^b_{0, L} + &\CalL f^b_{0, L} = 0 \,,
\\
    f^b_{0, L}|_{y=0} = &\phi_a(t, \mu) - \theta_a(t) \,, \qquad \mu > 0 \,,
\\
    f^b_{0, L} &\to 0 \hspace{2.7cm} \text{as $y \to \infty$.} 
\end{align*}
By the theory for kinetic half-space equations \cite{BardosSantosSentis:84, CoronGolseSulem:88}, $\theta_a(t)$ is uniquely determined by the incoming data $\phi_a(t, \mu)$  and the solution $f^b_{0, L}$ decays exponentially in $y$. The next-order boundary layer term $f^b_{1, L} \in L^\infty(\dy\dmu)$ satisfies the forced half-space equation
\begin{equation}
\begin{aligned} \label{eq:BL-next-order}
    \mu \del_y f^b_{1, L} + &\CalL f^b_{1, L} = - \del_t f^b_{0, L} \,,
\\
    f^b_{1, L}|_{y=0} &= 0 \,, \qquad \mu > 0 \,,
\\
    f^b_{1, L} \to & \, c_1 \hspace{1.3cm} \text{as $y \to \infty$,} 
\end{aligned}
\end{equation}
for some $c_1 \in \R$. If $\del_t \phi_a(t, \mu) \in L^\infty(\dmu)$ for each $t > 0$, then equation~\eqref{eq:BL-next-order} has a unique solution in $L^\infty(\dy\dmu)$ since $\del_t f^b_{0, L}(t, \cdot) \in L^\infty(\dy\dmu)$  and it decays exponentially as $y \to \infty$. The boundary layer term $f^b_L$ thus satisfies
\begin{equation}
\begin{aligned} \label{eq:f-b-L}
     \Eps \, \del_t f^b_L + \mu &\del_x f^b_L + \frac{1}{\Eps} \, \CalL f^b_L 
     = \Eps^3  \del_t f^b_{1, L} \,,
\\
     f^b_L \big|_{x=a} &= \phi_a(t, \mu) - \theta_a(t) \,, 
\hspace{3.6cm} \mu > 0 \,,
\\
     f^b_L \big|_{x=b} &=  f^b_{0, L}(t, \tfrac{b-a}{\Eps}, \mu)
                                     + \Eps^2 f^b_{1, L}(t, \tfrac{b-a}{\Eps}, \mu) \,,
\qquad \mu < 0 \,,
\\
     f^b_L \big|_{t=0} 
     &= \vpran{f^b_{0, L} + \Eps^2 f^b_{1, L}} (0, \tfrac{x-a}{\Eps}, \mu) \,.
\end{aligned}
\end{equation}
The boundary layer $f^b_R$ at $x=b$ satisfies a similar equation. Third, the initial layer term $f^I$ is constructed as
\begin{align*}
    f^I = f^I_0 + \Eps f^I_1 \,,
\end{align*}
where $f^I_0$ satisfies the initial layer equation
\begin{equation}
\begin{aligned} \label{eq:f-I-0}
    \del_\tau f^I_0 + &\CalL f^I_0 = 0 \,,
\\
    f^I_0|_{\tau = 0} = & \, \phi_0 - \vint{\phi_0} \,,
\\
   f^I_0 \to & \, 0 \qquad \text{as $\tau \to \infty$.}
\end{aligned}
\end{equation}
The existence and exponential decay of $f^I_0$ are shown in Proposition~\ref{prop:f-I-0}. The next order term $f^I_1$ is constructed as 
\begin{equation}
\begin{aligned} \label{eq:f-I-1}
    \del_\tau f^I_1 + \CalL f^I_1 &= -\mu \, \del_x f^I_0  \,,
\\
    f^I_1 \big|_{\tau = 0} & =  \, \phi_1(x, \mu) \,,
\\
   f^I_1 \to & \, 0  \,  \qquad \text{as $\tau \to \infty$} \,,
\end{aligned}
\end{equation}
for some $\phi_1$. The existence of $\phi_1$ is guaranteed by the following Lemma:
\begin{lem} \label{lem:well-posed-f-I-1}
   Let $\{p_n\}_{n=0}^\infty$ be the set of normalized Legendre polynomials. Then there exists $\phi_1(x, \cdot) \in L^2(\dmu)$ such that~\eqref{eq:f-I-1} has a unique solution. 
\end{lem}
\begin{proof}
We will solve for $f^I_1$ explicitly. By Proposition~\ref{prop:f-I-0},  
\begin{align} \label{soln:f-I-0}
     f_0^I = \sum_{n=1}^\infty e^{-\lambda_n \tau} \phi_{0, n} \, p_n
     =  \sum_{n=1}^\infty e^{-\lambda_n \tau} \vint{\phi_0 \, p_n} \, p_n\,.
\end{align}
Write $f^I_1, \phi_1$ as
\begin{align*}
    f^I_1 = \sum_{n=0}^\infty  f_{1, n} p_n \,,
\qquad
    \phi_1 = \sum_{n=0}^\infty  \phi_{1, n} p_n  \,.
\end{align*}
In order to solve for $f_{1, n}$, we multiply $p_n$ to~\eqref{eq:f-I-1} and integrate over $\mu \in [-1, 1]$. Then
\begin{align*}
    \del_\tau f_{1,n} + \lambda_n f_{1, n} = - \vint{\mu \, p_n \, \del_x f^I_0} \,,
\qquad f_{1, n} \big|_{\tau=0} = \phi_{1, n} \,,
\qquad n \geq 1 \,,
\end{align*}
and 
\begin{align*}
    \del_\tau f_{1,0} = - \vint{\mu \, \del_x f^I_0} \,,
\qquad f_{1, 0} \big|_{\tau=0} = \phi_{1, 0} \,.
\end{align*}
By~\eqref{cond:lambda_n}, it suffices to require that
\begin{align} \label{cond:mu-pn-L-1}
     \vint{\mu \, p_n \, \del_x f^I_0}  \in L^1 (\dtau) \,,
\qquad n \geq 1 \,,
\end{align}
and
\begin{align} \label{cond:mu-2-L-1}
    \int_0^\infty \vint{\mu \, \del_x f^I_0} \dtau = \phi_{1, 0} = \vint{\mu \phi_1} \,.
\end{align}
By~\eqref{soln:f-I-0}, 
\begin{align*}
     \vint{\mu \, p_n \, \del_x f^I_0} 
     = \sum_{k=n-1}^{n+1} e^{-\lambda_k \tau} \vpran{\del_x \phi_{0, k}} \vint{\mu \, p_n \, p_k} \in L^1 (\dtau) \,, 
\qquad n \geq 1 \,.
\end{align*}
In order for~\eqref{cond:mu-2-L-1} to hold, we can simply choose
\begin{align} \label{def:phi-1}
     \phi_1(x, \mu) = \frac{3}{\lambda_1} \mu \, \vint{\mu \, \del_x\phi_0(x, \mu)} \,.
\end{align}
With such $\phi_1$ equation~\eqref{eq:f-I-1} will have a unique solution.
\end{proof}

Combining~\eqref{eq:f-I-0} with~\eqref{eq:f-I-1}, we have that $f^I$ satisfies the equation
\begin{equation}
\begin{aligned} \label{eq:f-I}
     \Eps \, \del_t f^I + \mu \del_x f^I + &\frac{1}{\Eps} \, \CalL f^I 
     = \Eps \, \mu \, \del_x f^I_1(\tfrac{t}{\Eps^2}, x, \mu) \,,
\\
    f^I \big|_{x=a} 
    &= \vpran{f^I_0 + \Eps f^I_1} (\tfrac{t}{\Eps^2}, a, \mu)  \,, \qquad \mu > 0 \,,
\\
    f^I \big|_{x=b} 
    &= \vpran{f^I_0 + \Eps f^I_1} (\tfrac{t}{\Eps^2}, a, \mu) \,, \qquad \mu < 0 \,,
\\
    f^I \big|_{t=0} &=  \phi_0 - \vint{\phi_0} + \Eps \phi_1 \,.
\end{aligned}
\end{equation}
Finally, the initial-boundary layer term $f^{\IBL}$ at $x=a$ satisfies that
\begin{align*}
    \del_\tau f^{\IBL}_L + \mu \del_y &f^{\IBL}_L + \CalL f^{\IBL}_L = 0 \,,
\\
   f^{\IBL}_L|_{y = 0 } &= - f^I(\tau, a, \mu) \,, \qquad \mu > 0 \,,
\\
   f^{\IBL}_L|_{\tau = 0} &= -f^{b}_L(0, y, \mu) \,,
\\
   f^{\IBL}_L &\to 0 \,, \hspace{2.4cm} \text{as $\tau, y \to \infty$} \,.
\end{align*}
The main assumption for $f^{\IBL}_L$ and the initial-boundary layer $f^{\IBL}_R$ at $x=b$ is
\begin{equation}
\begin{aligned} \label{assump:f-IBL}
     f^{\IBL}_L(\tau, x, \mu), f^{\IBL}_R(\tau, x, \mu) &\in L^2((0, \infty) \times (0, \infty) \times (-1, 1)) \,,
\\
     f^{\IBL}_L (\tau, x, \mu), f^{\IBL}_R(\tau, x, \mu) 
 &\sim  
     \CalO \vpran{\frac{1}{\tau^{\kappa_0}}}
\qquad \text{for some $\kappa_0 > 0$ as $\tau \to \infty$,}
\end{aligned}
\end{equation}
for each fixed $x$. In terms of $(t, x, \mu)$, we have 
\begin{equation}
\begin{aligned} \label{eq:f-IBL-L}
    \Eps \, \del_t f^{\IBL}_L + \mu \del_x &f^{\IBL}_L 
    + \frac{1}{\Eps} \CalL f^{\IBL}_L = 0 \,.
\\
   f^{\IBL}_L \big|_{x = a} 
   &= - f^I(\tfrac{t}{\Eps^2}, a, \mu) \,, \hspace{1.6cm} \mu > 0 \,,
\\
   f^{\IBL}_L \big|_{x = b} 
   &= - f^{\IBL}_L (\tfrac{t}{\Eps^2}, \tfrac{b-a}{\Eps}, \mu) \,,
   \qquad \mu < 0 \,,
\\
   f^{\IBL}_L \big|_{t=0} & = - f^b_{0,L}(0, \tfrac{x-a}{\Eps}, \mu) \,.
\end{aligned}
\end{equation}

The main result for the error estimate is 
\begin{thm} \label{thm:error-compatible}
For each $T > 0$, suppose
\begin{align*}
    \phi_0 \in C^4([a, b]; L^\infty(-1, 1)) \,, 
\qquad
    \phi_a, \phi_b \in C^2([0, T]; L^\infty(-1,1)) \,.
\end{align*} 
Suppose the heat data $\theta_a, \theta_b, \theta_0$ derived from the layers are compatible such that 
\begin{align} \label{cond:compatibility}
  \del_t^k \theta_a(0) = \del_x^{2k} \theta_0 (a) \,,
\qquad
   \del_t^k \theta_b(0) = \del_x^{2k} \theta_0(b) \,,
\qquad
    k = 0, 1, 2 \,.
\end{align}
Then the approximate solution constructed in~\eqref{def:f-A-approximation} satisfies that 
\begin{align*}
     \norm{f - f^A}_{L^\infty(0, T; L^2([a, b] \times [-1, 1])} 
\leq \tilde C_{0} \sqrt{\Eps} \,,
\end{align*}
where the constant $C_{0,1}$ depend on $T, \theta, \phi_a, \phi_b, \phi_0, $ and their derivatives.
Moreover, if we assume that~\eqref{assump:f-IBL} holds for the initial-boundary layer solution, then
\begin{align*}
     \norm{f - \theta}_{L^\infty(t_0, T; L^2([x_l, x_r] \times [-1, 1])} 
\leq C_0 \max\{\sqrt{\Eps}, \, \, \Eps^{2\kappa_0}\} \,,
\end{align*}
for any $t_0 \in (0, T)$ and $[x_l, \, x_r] \subseteq (a, b)$. Here the constant $C_{0,2}$ depend on $T, t_0, x_l, x_r,  \phi_a, \phi_b, \phi_0$ and their derivatives.
\end{thm}
\begin{proof}
First note that the compatibility condition in~\eqref{cond:compatibility} guarantees that $\theta \in C^2([0, T]; C^4([a, b]))$. Let $E_f$ be the error term such that $E_f = f - f^A$. Subtract 
equations~\eqref{eq:f-int}, \eqref{eq:f-b-L}, \eqref{eq:f-I}, \eqref{eq:f-IBL-L}, and the counterpart of the boundary and initial-boundary layer corrector equations at $x=b$ from equation~\eqref{eq:kinetic}. Then 
\begin{equation}
\begin{aligned} \label{eq:error-E-f}
    \Eps \, \del_t E_f + \mu \del_x E_f& + \frac{1}{\Eps} \CalL E_f
    = R(t, x, \mu)  - \Eps^3  \del_t f^b_{1, L}
        - \Eps^3  \del_t f^b_{1, R}
        - \Eps \, \mu \, \del_x f^I_1(\tfrac{t}{\Eps^2}, x, \mu) \,,
\\
    E_f \big|_{x=a} 
     = & \, \Eps \, \CalL^{-1}(\mu) \, \del_x \theta(t, a)
           - \Eps^2 \CalL^{-1}(\mu \CalL^{-1}(\mu))
                                \vpran{\del_{xx} \theta(t, a)}
\\
     & - \vpran{f^b_{0, R} + \Eps^2 f^b_{1, R}} (t, \tfrac{b-a}{\Eps}, \mu)
         + f^{\IBL}_R (\tfrac{t}{\Eps^2}, \tfrac{b-a}{\Eps}, \mu) \,,
     \qquad \mu > 0 \,,
\\
    E_f \big|_{x=b} 
     = & \, \Eps \, \CalL^{-1}(\mu) \, \del_x \theta(t, b)
           - \Eps^2 \CalL^{-1}(\mu \CalL^{-1}(\mu))
                                \vpran{\del_{xx} \theta(t, b)}
\\
     & - \vpran{f^b_{0, L} + \Eps^2 f^b_{1, L}} (t, \tfrac{b-a}{\Eps}, \mu)
         + f^{\IBL}_L (\tfrac{t}{\Eps^2}, \tfrac{b-a}{\Eps}, \mu) \,,
     \qquad \mu < 0 \,,
\\
    E_f \big|_{t=0} 
     = & \, \Eps \, \CalL^{-1}(\mu) \, \del_x \theta(0, x)
           - \Eps^2 \CalL^{-1}(\mu \CalL^{-1}(\mu))
                                \vpran{\del_{xx} \theta(0, x)}
\\
     &    - \Eps^2 f^b_{1, L} (0, \tfrac{x-a}{\Eps}, \mu)
           - \Eps^2 f^b_{1, R} (0, \tfrac{b-x}{\Eps}, \mu)
            - \Eps \, \phi_1(x, \mu) \,.
\end{aligned}
\end{equation}
where
\begin{align*}
   R(t, x, \mu) 
   = \vpran{\Eps^2 \del_x^3 \theta} (\mu \CalL^{-1}(\mu \CalL^{-1}(\mu)))
       + \vpran{\Eps^2 \del_{tx} \theta} \vpran{\CalL^{-1}(\mu)}
        - \vpran{\Eps^3 \del_{txx} \theta} \vpran{\CalL^{-1}(\mu \CalL^{-1}(\mu))} \,.
\end{align*}
Since $\mu \CalL^{-1}(\mu \CalL^{-1}(\mu))$ is odd in $\mu$, we have $\mu \CalL^{-1}(\mu \CalL^{-1}(\mu)) \in (\NullL)^\perp$, which implies that 
\begin{align} \label{R-Null-Perp}
     R(t, x, \mu) \in (\NullL)^\perp \,.
\end{align}
Perform the basic $L^2$-estimate by multiplying~\eqref{eq:error-E-f} by $E_f$ and integrate in $x, \mu$. By the coercivity of $\CalL$ in (P3), we have \begin{equation}
\begin{aligned} \label{ineq:energy-1}
& \quad \,
    \frac{\Eps}{2} \sup_{t \in [0, T]}\|E_f(t)\|_{L^2_{x,\mu}}^2
    + \frac{\gamma_0}{\Eps} \int_0^T \norm{\CalP^\perp (E_f)}_{L^2_{x,\mu}}^2 (s) \ds
\\
& \hspace{3cm}    + \int_0^T \int_{-1}^1 \mu E_f^2(s, b, \mu) \dmu\ds
    - \int_0^T \int_{-1}^1 \mu E_f^2(s, a, \mu) \dmu\ds
\\
&\leq 
     \vpran{\int_0^T \norm{\CalP^\perp (E_f)}_{L^2_{x,\mu}}^2 (s) \ds}^{\frac 12}
     \vpran{\int_0^T \norm{R(t, \cdot, \mu)}_{L^2_{x,\mu}}^2 (s) \ds}^{\frac 12}
\\
& \hspace{3cm}
     + \sup_{t \in [0, T]} \|E_f(t)\|_{L^2_{x,\mu}}
     \vpran{\int_0^T \norm{G}_{L^2_{x, \mu}}\ds}
     +  \frac{\Eps}{2} \|E_f(0)\|_{L^2_{x,\mu}}^2 \,,
\end{aligned}
\end{equation}
where we have applied~\eqref{R-Null-Perp} and denoted $G$ for the right-hand side of the first equation of~\eqref{eq:error-E-f} excluding $R$. The bounds for $G$ and $R$ are
\begin{align*}
    \int_0^T \norm{G(s)}_{L^2_{x, \mu}}\ds
&\leq 
    \int_0^T \norm{ \Eps^3  \del_t f^b_{1, L}
        + \Eps^3  \del_t f^b_{1, R}}_{L^2_{x, \mu}}\ds
   + \int_0^T \norm{\Eps \, \mu \, \del_x f^I_1(\tfrac{s}{\Eps^2}, x, \mu)}_{L^2_{x, \mu}}\ds
\\
&\leq    C_{1,1} \, \Eps^3 (1 + T) \,,
\end{align*}
and 
\begin{align*}
    \vpran{\int_0^T \norm{R}_{L^2_{x, \mu}}^2 \ds}^{\frac 12}
\leq 
    C_{1,2} \, \Eps^2 (1 + \sqrt{T}) \,,
\end{align*}
where $C_{1,1}, C_{1,2}$ only depends on $\theta, \phi_0, \phi_a, \phi_b$ and their derivatives. By definition, the initial data term satisfies
\begin{align*}
     \frac{\Eps}{2} \|E_f(0)\|_{L^2_{x,\mu}}^2 
\leq C_{1,3} \, \Eps^3 \,, 
\end{align*}
where $C_{1,3}$ only depends on $\theta, \phi_1, f^b_{1, L}, f^b_{1, R}$ and the derivatives of $\theta$. Moreover, the boundary terms in~\eqref{ineq:energy-1} satisfy that
\begin{equation} 
\begin{aligned} \label{est:boundary-term-1}
    \int_0^T \int_{-1}^1 \mu E_f^2(s, a, \mu) \dmu\ds
\leq 
    \int_0^T \int_0^1 \mu E_f^2(s, a, \mu) \dmu\ds
\leq
    C_{1,4} \, \Eps^2 (1 + T) \,,
\\
    -\int_0^T \int_{-1}^1 \mu E_f^2(s, b, \mu) \dmu\ds
\leq 
    -\int_0^T \int_{-1}^0 \mu E_f^2(s, b, \mu) \dmu\ds
\leq
    C_{1,5} \, \Eps^2 (1 + T) \,,
\end{aligned}
\end{equation}
where $C_{1,4}, C_{1,5}$ only depend on $\theta$ and its derivatives and the bounds of $f^b_{0, R}, f^b_{0, L}, f^b_{1, R}$, $f^b_{1, L}, f^{\IBL}_R, f^{\IBL}_L$. Overall, we have
\begin{align*}
    \sup_{[0, T]} \|E_f(t)\|_{L^2([a, b] \times [-1, 1])} 
\leq \tilde C_0 (1 + T) \sqrt{\Eps} \,.
\end{align*}
By the exponential decay of the initial and boundary layer solution and the assumption~\eqref{assump:f-IBL} for the initial-boundary layer, the interior error is bounded as
\begin{align*}
    \sup_{[t_0, T]} \norm{f - \theta}_{L^2([x_l, x_r] \times [-1, 1])} 
\leq  C_0 (1 + T) \max\{\sqrt{\Eps}, \,\, \Eps^{2\kappa_0}\} \,, 
\end{align*}
for any $t_0 > 0, [x_l, x_r] \subseteq (a, b)$ and $C_0$ depends on $t_0, x_l, x_r$, $\theta, \phi_0, \phi_a, \phi_a$ together with their derivatives up to order $4$.
\end{proof}

\begin{rmk}
Unlike the stationary case treated in \cite{GJL99}, using the above analysis we cannot expect to improve the accuracy simply by simply applying a Robin boundary condition and a higher order initial layer without any addition knowledge of the decay of $f^{\IBL}_L$ and $f^{\IBL}_R$. The reason is the initial-boundary layer terms $f^{\IBL}_R (\tfrac{t}{\Eps^2}, \tfrac{b-a}{\Eps}, \mu)$ and $f^{\IBL}_L (\tfrac{t}{\Eps^2}, \tfrac{b-a}{\Eps}, \mu)$ alway contribute an order $\BigO(\Eps^2)$ in~\eqref{est:boundary-term-1} to the error estimate.
\end{rmk}


\subsection{Incompatible data} In this part we treat the case where boundary and initial data for the diffusion equation may not be compatible. In this case, one only has the boundedness of the solution $\theta$ to the diffusion equation while any of its derivatives can be unbounded near the corner. As a consequence,  the basic error estimates in Section~\ref{est:compatible-data} can not be applied directly. To circumvent this difficulty, we will modify the data for the diffusion equation, apply Theorem~\ref{thm:error-compatible} for the modified equation, and estimate the extra error induced by the modification.

We start with a simple lemma that estimates the derivatives of the solution to the diffusion equation. 
\begin{lem} \label{lemma:derivative-theta}
Suppose $\del_x^n \psi_0(0) = 0$ for $0 \leq n \leq 4$. Let $\theta \in C^2([0, T]; C^4[a, b])$ be the solution to the diffusion equation
\begin{align*}
     \del_t \theta =& \, \gamma^2 \del_{xx} \theta \,,
\qquad \gamma > 0 \,,
\\
     \theta \big|_{x=a} =& \,  \theta \big|_{x=b} = 0 \,,
\\
     \theta \big|_{t=0} & \, = \psi_0 (x) \,.      
\end{align*}
Then for any $0 \leq k \leq 4$, 
\begin{align*}
     \norm{\del_x^{k+1} \theta}_{L^2([0, T] \times [a, b])}
\leq  
     \frac{1}{\sqrt{2} \gamma} \norm{\del_x^k \psi_0}_{L^2([a, b])} \,.
\end{align*}
\end{lem}
\begin{proof}
The proof is done by using direct $L^2$ energy estimate. Specifically, for each $k \leq 4$, we differentiate the diffusion equation (in $x$) $k$ times and apply the $L^2$ estimate. Then
\begin{align} \label{eq:energy-theta-k}
     \frac{1}{2} \frac{\rm d}{\dt} \norm{\del_x^k \theta}_{L^2([a, b])}^2
= \gamma^2 \vpran{\del_x^k \theta} \vpran{\del_x^{k+1} \theta} \Big|_{a}^b
   - \gamma^2 \int_a^b \abs{\del_x^{k+1} \theta}^2 \dx \,.
\end{align}
Note that one of $k$ and $k+1$ is even. By the form of the diffusion equation, 
\begin{align*}
    \del_x^{2K} \theta \big|_a^b= \del_t^K \theta \big|_a^b = 0 \,.
\end{align*}
for any $K \in \N$.
Therefore, 
\begin{align*}
    \vpran{\del_x^k \theta} \vpran{\del_x^{k+1} \theta} \Big|_{a}^b = 0 \,,
\qquad \text{for any $0 \leq k \leq 4$.}
\end{align*}
Integrating~\eqref{eq:energy-theta-k} over $[0, T]$ and taking square root on both sides gives
\begin{align*}
    \norm{\del_x^{k+1} \theta}_{L^2([0, T] \times [a, b])}
\leq 
   \frac{1}{\sqrt{2} \gamma} \norm{\del_x^k \psi_0}_{L^2([a, b])} \,,
\end{align*}
which shows $\norm{\del_x^{k+1} \theta}_{L^2([0, T] \times [a, b])}$ is bounded by the derivatives of the initial for each $k \leq 4$. 
\end{proof}

Let $\theta$ be the solution to the incompatible diffusion equation
\begin{equation}
\begin{aligned} \label{eq:theta-incompatible}
    \del_t \theta - \vint{\mu \CalL^{-1} (\mu &)} \, \del_{xx} \theta = 0 \,,
\qquad\qquad x \in (a, b) \,,
\\
    \theta|_{x=0} =& \,  \theta_a(t) \,, 
\\
    \theta|_{x=1} =& \,  \theta_b(t) \,, 
\\
    \theta|_{t=0} = \, &  \theta_0 (x) \,, \hspace{2.5cm} x \in (a, b) \,,
\end{aligned}
\end{equation}
where $\theta_a, \theta_b \in C^2([0, T])$ and $\theta_0 \in C^4([a, b])$.

The main result regarding the incompatible data is
\begin{thm} \label{thm:error-incompatible}
Let $f$ be the solution to the kinetic equation~\eqref{eq:kinetic}. Let $\theta$ be the solution to~\eqref{eq:theta-incompatible} where $\theta_a, \theta_b, \theta_0$ are given by the layers.  Assume that~\eqref{assump:f-IBL} holds. Then there exists a constant $\hat C_0$ such that 
\begin{align*}
     \norm{f - \theta}_{L^\infty(t_0, T; L^2([x_l, x_r] \times [-1, 1])} 
\leq \hat C_0 \, \Eps^{2/5} \,,
\end{align*}
for any $t_0 \in (0, T)$ and $[x_l, \, x_r] \subseteq (a, b)$. Here $\hat C_0$ depending on $t_0, T, x_l, x_r, \phi_a, \phi_b, \phi_0$ and their derivatives.
\end{thm}
\begin{proof}
Introduce a modified initial data
$\tilde \theta_0 \in C^4([a, b])$ such that 
\begin{align*}
    \del_x^{2m} \tilde\theta_0 (a) = \del_t^m \theta_a(0) \,,
\qquad
     \del_x^{2m} \tilde\theta_0 (b) = \del_t^m \theta_b(0) \,,
\qquad m = 0, 1, 2 \,.
\end{align*}
Moreover, for some $0 < \delta < 1/4$ to be chosen, we require that 
\begin{align} \label{def:tilde-phi-0}
     \tilde\theta_0 (x) = \theta_0 (x) \,, 
\qquad x \in [a+2\delta, \, b-2\delta] \,.
\end{align}
In other words, we slightly modify $\theta_0$ near the corners such that the new initial data is compatible with the boundary data. Therefore, if we denote $\tilde \theta$ as the solution to the modified diffusion equation
\begin{align*}
    \del_t \tilde\theta - \vint{\mu \CalL^{-1} (\mu &)} \, \del_{xx} \tilde\theta = 0 \,,
\qquad\qquad x \in (a, b) \,,
\\
    \tilde\theta \big|_{x=a} =& \,  \theta_a(t) \,, 
\\
    \tilde\theta \big|_{x=b} =& \,  \theta_b(t) \,, 
\\
    \tilde\theta \big|_{t=0} = \, &  \tilde\theta_0 (x) \,, \hspace{2.4cm} x \in (a, b) \,,
\end{align*}
then $\tilde\theta \in C^2([0, T]; C^4[a, b])$. To remove the boundary data $\theta_a, \theta_b$, we introduce another modified initial data $\hat\theta_0 \in C^\infty([a,b])$ such that
\begin{align*}
    \del_x^{2l} \hat\theta_0 (a) = \del_t^l \phi_a(0) \,,
\qquad
     \del_x^{2l} \hat\theta_0 (b) = \del_t^l \phi_b(0) \,,
\qquad
     l = 0, 1, 2 \,,
\end{align*}
and
\begin{align*}
     \norm{\del_x^m \hat\theta_0}_{L^\infty([a, b])} 
     \leq C_1 
\qquad
     \text{for $0 \leq m \leq 4$} \,.
\end{align*}
where $C_1$ is a universal constant. Therefore, if we let $\hat\theta$ satisfy that
\begin{align*}
    \del_t \hat\theta - \vint{\mu \CalL^{-1} (\mu &)} \, \del_{xx} \hat\theta = 0 \,,
\qquad\qquad x \in (a, b) \,,
\\
    \hat\theta \big|_{x=0} =& \,  \theta_a(t) \,, 
\\
    \hat\theta \big|_{x=1} =& \,  \theta_b(t) \,, 
\\
    \hat\theta \big|_{t=0} = \, &  \hat\theta_0 (x) \,, \hspace{2.6cm} x \in (a, b) \,,
\end{align*}
Then $\hat\theta \in C^2([0, T]; C^4[a, b])$ and all the derivatives of $\hat\theta$ (up to order 4) are of order $\CalO(1)$. Let 
\begin{align*}
    \tilde \theta_1 = \tilde \theta - \hat \theta \,.
\end{align*}
Then $\tilde \theta_1$ satisfies
\begin{align*}
    \del_t \tilde\theta_1 - & \, \vint{\mu \CalL^{-1} (\mu)} \, \del_{xx} \tilde\theta_1 = 0 \,,
\qquad\qquad x \in (a, b) \,,
\\
    & \, \tilde\theta_1 \big|_{x=a} =  0 \,, 
\\
    & \,\tilde\theta_1 \big|_{x=b} = 0 \,, 
\\
    \tilde\theta_1 & \, \big|_{t=0} = \Psi_0(x) \,, \hspace{2.8cm} x \in (a, b) \,,
\end{align*}
where 
\begin{align*}
   \Psi_0(x) = \tilde\theta_0 (x) - \hat\theta_0(x) \,.
\end{align*} 
Note that
\begin{align*}
    \abs{\del_x^m \tilde\theta}
\leq 
    \abs{\del_x^m \tilde\theta_1} + \abs{\del_x^m \hat\theta}
\leq
    \abs{\del_x^m \tilde\theta_1} + C_1 \,,
\end{align*}
for $0 \leq m \leq 4$. Therefore, we only need to bound the derivatives of $\tilde\theta_1$ in order to obtain the bounds for the derivatives of $\tilde\theta$. By Lemma~\ref{lemma:derivative-theta}, 
\begin{align*}
    \norm{\del_x^m \tilde\theta_1}_{L^2([0, T] \times [a, b])}
\leq 
   \frac{1}{\sqrt{2} \gamma} \norm{\del_x^{m-1} \Psi_0}_{L^2([a, b])} \,,
\qquad 1 \leq m \leq 4 \,,
\end{align*}
where $\gamma^2 = \vint{\mu \CalL^{-1} (\mu)}$.
By the definition of $\Psi_0$, we have
\begin{align*} 
    \norm{\del_x^{m-1} \Psi_0}_{L^2([a, b])}
\leq
    C_{2, 1} \vpran{1 + \delta^{-(m - \frac{3}{2})}} \,.
\end{align*}
Hence, 
\begin{align} \label{est:del-m-tilde-theta}
    \norm{\del_x^m \tilde\theta}_{L^2([0, T] \times [a, b])}
\leq 
    C_{2, 2} \vpran{1 + \delta^{-(m - \frac{3}{2})}} \,,
\qquad 1 \leq m \leq 4 \,.
\end{align}
where $C_{2,1}, C_{2,2}$ only depend on $\theta_a, \theta_b, \theta_0$.
We can also obtain a pointwise bound for $\del_x \tilde\theta$ by interpolation. By $L^\infty(a,b) \hookrightarrow H^1(a,b)$, we have
\begin{equation}
\begin{aligned} \label{est:del-theta-L-infty}
    \norm{\del_x \tilde\theta}_{L^2(0, T; L^\infty(a,b))}
&\leq C_{2,3} \vpran{ \norm{\del_x \tilde\theta}_{L^2((0, T) \times (a,b))}
       + \norm{\del_{xx} \tilde\theta}_{L^2((0, T) \times (a,b))}}
\\
&\leq
      C_{2,4} \vpran{1 + \frac{1}{\sqrt{\delta}}} \,.
\end{aligned}
\end{equation}
Now we use $\tilde \theta$ instead of $\theta$ for the interior approximation and repeat the error analysis for the compatible case. The approximate solution has the same form as in~\eqref{def:f-A-approximation} with the interior approximation changed as
\begin{align*}
    \tilde f^{\inte} (t, x , \mu) 
    = \tilde\theta(t, x) - \Eps \, \CalL^{-1}(\mu) \, \del_x \tilde\theta(t, x)
       + \Eps^2 \vpran{\del_{xx} \tilde\theta(t, x)} 
           \CalL^{-1}(\mu \CalL^{-1}(\mu))  \,.
\end{align*}
We keep all the layer corrections the same as before and define
\begin{align*}
   \tilde f^A = & \tilde f^{\inte}(t, x, \mu) 
             + f^b_{L} (t, \tfrac{x-a}{\Eps}, \mu)
             + f^b_{R} (t, \tfrac{b-x}{\Eps}, \mu)
             + f^{I}(\tfrac{t}{\Eps^2}, x, \mu)
\\
           & + f^{\IBL}_L(\tfrac{t}{\Eps^2}, \tfrac{x-a}{\Eps}, \mu)
             + f^{\IBL}_R(\tfrac{t}{\Eps^2}, \tfrac{b-x}{\Eps}, \mu) \,,
\end{align*}
Denote $\tilde E_f = f - \tilde f^A$. Then the error equation for $\tilde E_f$ has the same form as~\eqref{eq:error-E-f} except two changes: $\theta$ will be replaced by $\tilde \theta$ and there is an additional term $R_1$ in the initial condition for $\tilde E_f$ given by 
\begin{align*}
     R_1(x) 
     =  \phi_0(x) - \tilde \theta_0(x) \,,
\end{align*}
where by the definition of $\tilde\theta_0$ in~\eqref{def:tilde-phi-0}, 
\begin{align*}
    R_1(x) = 0 \,, 
\qquad x \in [a + 2\delta, \,  b - 2\delta] \,.
\end{align*}
Hence, 
\begin{align} \label{est:R-1-L-2}
    \norm{R_1}_{L^2(a,b)}
\leq 
    C_{2,5} \sqrt{\delta}\,.
\end{align}
The derivative terms in the initial data for $\tilde E_f$ are given by 
\begin{align} \label{est:tilde-E-f-4}
    \int_0^1 \abs{\del_x^k \tilde\theta(0, x)}^2 \dx
   =  \int_0^1 \abs{\del_x^k \tilde\theta_0(x)}^2 \dx
\leq 
    \frac{C_{2, 6} }{\delta^{2k-1}} \,, 
\qquad  k = 1, 2 \,.
\end{align}
The first derivative terms in the boundary data for $\tilde E_f$ is bounded by~\eqref{est:del-theta-L-infty}. The second derivative terms in the boundary data for $\tilde E_f$ at $x=a$ satisfies
\begin{align} \label{est:del-xx-tilde-theta}
    \del_{xx} \tilde\theta \big|_{x=a} 
    = \del_t (\tilde\theta_1 + \hat\theta) \big|_{x=a}
    = \del_t \hat\theta \big|_{x=a}
    \sim \CalO(1) \,.
\end{align}
Similar estimates hold for derivatives terms at at $x=1$.
Combining~\eqref{est:del-m-tilde-theta}, \eqref{est:del-theta-L-infty}, \eqref{est:R-1-L-2}, \eqref{est:tilde-E-f-4}, \eqref{est:del-xx-tilde-theta}, and repeating the $L^2$-estimate for the compatible case, we have
\begin{align*}
    \sup_{t \in [0, T]} \norm{\tilde E_f}_{L^2([a, b] \times [-1, 1])}^2
\leq 
    C_{2,7} \vpran{\Eps
    + \frac{\Eps}{\delta}
    + \frac{\Eps^2}{\delta}
    + \frac{\Eps^4}{\delta^3}
    + \delta}  \,,
\end{align*}
where $C_{2,7}$ depends on $\theta_a, \theta_b, \theta_0, T$. Hence, by the fast decay of the layers, the interior error is given by 
\begin{align} \label{est:f-tilde-theta-interior}
    \sup_{t \in [t_0, T]}\norm{f - \tilde\theta}_{L^2([x_l, x_r] \times [-1, 1])}
\leq
    C_{2,8} \vpran{\sqrt{\Eps}
    + \frac{\sqrt{\Eps}}{\sqrt{\delta}}
    + \frac{\Eps}{\sqrt{\delta}}
    + \frac{\Eps^2}{\delta^{3/2}}
    + \sqrt{\delta}}
\end{align}
for any $t_0 \in (0, T)$ and $[x_l, x_r] \subseteq (a, b)$. Moreover, we have
\begin{equation}
\begin{aligned} \label{est:theta-tilde-theta}
& \quad \,
    \sup_{t \in [0, T]} \norm{(\theta - \tilde\theta)(t)}_{L^2_x}^2
\leq  
    \norm{\phi_0 - \tilde\theta_0}_{L^2_x}^2
\leq C_{2,9} \, \delta \,.
\end{aligned}
\end{equation}
Combining~\eqref{est:f-tilde-theta-interior} with~\eqref{est:theta-tilde-theta}, we have
\begin{align*}
    \sup_{t \in [t_0, T]}\norm{f - \theta}_{L^2([x_l, x_r] \times [-1,1])}
\leq
    C_{2, 10} \vpran{\sqrt{\Eps}
    + \frac{\sqrt{\Eps}}{\sqrt{\delta}}
    + \frac{\Eps}{\sqrt{\delta}}
    + \frac{\Eps^2}{\delta^{3/2}}
    + \sqrt{\delta}} \,.
\end{align*}
By choosing $\delta = \Eps^{4/5}$, we have
\begin{align*}
    \sup_{t \in [t_0, T]}\norm{f - \theta}_{L^2([x_l, x_r] \times [-1, 1])}
\leq
    \hat C_0 \, \Eps^{2/5}
\end{align*}
for any $t_0 \in (0, T)$ and $[x_l, x_r] \subseteq (a, b)$.
\end{proof}

\bibliographystyle{amsxport}
\bibliography{kinetic}

\end{document}